\documentclass[11pt]{amsart}
\usepackage[utf8]{inputenc}
\usepackage[pdftex]{graphicx}
\usepackage{amssymb}
\usepackage{amsmath}
\usepackage{amsfonts}
\usepackage{amsthm}
\usepackage{bm}
\usepackage{mathrsfs}
\usepackage{appendix}
\usepackage{enumerate}
\usepackage[left=3cm, right=3cm, bottom=3.4cm]{geometry}
\usepackage[stretch=30,shrink=30]{microtype}
\usepackage{stmaryrd}
\usepackage[normalem]{ulem} 

\usepackage{xcolor}
\definecolor{antiquefuchsia}{rgb}{0.57, 0.36, 0.51}
\definecolor{azure}{rgb}{0.0, 0.5, 1.0}

\usepackage[backref=page, urlcolor = black, colorlinks = true, linkcolor = black, citecolor = antiquefuchsia]{hyperref}

\renewcommand*{\backref}[1]{}
\renewcommand*{\backrefalt}[4]{%
    \ifcase #1 (Not cited.)%
    \or        (Cited on page~#2.)%
    \else      (Cited on pages~#2.)%
    \fi}

\makeatletter
\def\th@plain{%
	\thm@notefont{}% same as heading font
	\itshape % body font
}
\def\th@definition{%
	\thm@notefont{}% same as heading font
	\normalfont % body font
}
\makeatother

\newcommand\res{\mathop{\hbox{\vrule height 7pt width .3pt depth 0pt\vrule height .3pt width 5pt depth 0pt}}\nolimits}

\numberwithin{equation}{section}

\newtheorem{theorem}{Theorem}[section]
\newtheorem{lemma}[theorem]{Lemma}
\newtheorem{proposition}[theorem]{Proposition}
\newtheorem{corollary}[theorem]{Corollary}

\theoremstyle{definition}
\newtheorem{definition}[theorem]{Definition}
\newtheorem{assumption}[theorem]{Assumption}
\theoremstyle{remark}
\newtheorem{remark}[theorem]{Remark}

\newcommand{\N}{\mathbb{N}}

\newcommand{\R}{\mathbb{R}}
\newcommand{\C}{\mathbb{C}}

\newcommand{\mres}{\mathbin{\vrule height 1.6ex depth 0pt width
0.13ex\vrule height 0.13ex depth 0pt width 1.3ex}}

\DeclareMathOperator{\spt}{spt}

\DeclareMathOperator{\dist}{dist}

\DeclareMathOperator{\BV}{BV}

\DeclareMathOperator{\divergence}{div}

\DeclareMathOperator{\Id}{Id}

\DeclareMathOperator{\loc}{loc}
\DeclareMathOperator{\graph}{graph}
\DeclareMathOperator{\Sing}{Sing}

\usepackage[capitalise,nameinlink]{cleveref}
\crefformat{equation}{#2(#1)#3}
\crefrangeformat{equation}{#3(#1)#4 to~#5(#2)#6}
\crefmultiformat{equation}{#2(#1)#3}{ and~#2(#1)#3}{, #2(#1)#3}{ and~#2(#1)#3}

\newcommand{\toweakstar}{\overset{*}\rightharpoonup}
\DeclareMathOperator{\diverg}{div}

\newcommand{\Irm}{\mathrm{I}}

\newcommand{\Acal}{\mathcal{A}}
\newcommand{\Bcal}{\mathcal B}

\newcommand{\Ecal}{\mathcal{E}}
\newcommand{\Fcal}{\mathcal{F}}
\newcommand{\Gcal}{\mathcal{G}}
\newcommand{\Hcal}{\mathcal{H}}

\newcommand{\Lcal}{\mathcal{L}}
\newcommand{\Mcal}{\mathcal{M}}

\newcommand{\Scal}{\mathcal{S}}

\newcommand{\Ffrak}{\mathfrak{F}}

\newcommand{\Sfrak}{\mathfrak{S}}

\newcommand{\Cscr}{\mathscr{C}}

\newcommand{\Pscr}{\mathscr{P}}

\newcommand{\Wscr}{\mathscr{W}}

\newcommand{\Abf}{\mathbf{A}}
\newcommand{\Bbf}{\mathbf{B}}
\newcommand{\Cbf}{\mathbf{C}}
\newcommand{\Dbf}{\mathbf{D}}
\newcommand{\Ebf}{\mathbf{E}}

\newcommand{\Gbf}{\mathbf{G}}
\newcommand{\Hbf}{\mathbf{H}}
\newcommand{\Ibf}{\mathbf{I}}

\newcommand{\Lbf}{\mathbf{L}}

\newcommand{\Sbf}{\mathbf{S}}

\newcommand{\Reg}{\mathrm{Reg}}
\newcommand{\eps}{\varepsilon}
\newcommand{\pbf}{\mathbf{p}}
\newcommand{\mbf}{\boldsymbol{m}}
\newcommand{\bphi}{\boldsymbol{\varphi}}
\newcommand{\bPhi}{\boldsymbol{\Phi}}

\newcommand{\hbf}{\mathbf{h}}
\newcommand{\Err}{\mathrm{Err}}

\newcommand{\spn}{\mathrm{span}}
\newcommand{\proofstep}[1]{\textit{#1}}

\title[Singularities and tangent cones for semicalibrated currents]{Rectifiability of the singular set and uniqueness of tangent cones for semicalibrated currents} 

\author[Paul Minter]{Paul Minter}
\address{Department of Pure Mathematics and Mathematical Statistics, University of Cambridge}
\email{\url{pdtwm2@cam.ac.uk}}

\author[Davide Parise]{Davide Parise}
\address{University of California San Diego, Department of Mathematics, 9500 Gilman Drive \#0112, La Jolla, CA 92093-0112, United States of America \& Simons Laufer Mathematical Institute, 17 Gauss Way, Berkeley, CA 94720-5070, United States of America}
\email{\url{dparise@ucsd.edu}}

\author[Anna Skorobogatova]{Anna Skorobogatova}
\address{Simons Laufer Mathematical Sciences Institute, 17 Gauss Way, Berkeley, CA 94720-5070, United States of America}
\email{\url{as110@princeton.edu}}

\author[Luca Spolaor]{Luca Spolaor}
\address{University of California San Diego, Department of Mathematics, 9500 Gilman Drive \#0112, La Jolla, CA 92093-0112, United States of America}
\email{\url{lspolaor@ucsd.edu}}

\date{\today}

\begin{document}

\begin{abstract}
We prove that the singular set of an $m$-dimensional integral current $T$ in $\mathbb{R}^{n + m}$, semicalibrated by a $C^{2, \kappa_0}$ $m$-form $\omega$ is countably $(m - 2)$-rectifiable. Furthermore, we show that there is a unique tangent cone at $\mathcal{H}^{m - 2}$-a.e. point in the interior singular set of $T$. Our proof adapts techniques that were recently developed in \cite{DLSk1, DLSk2, DMS} for area-minimizing currents to this setting. 
\end{abstract}

\maketitle

\tableofcontents

\section{Introduction} 
In this article, we study the structure of interior singularities of semicalibrated integral currents in $\R^{m+n}$. Let us recall the basic definitions. 
\begin{definition}\label{d:semical}
    Let $m, n \geq 2$ be positive integers. A semicalibration in $\R^{m+n}$ is a $C^1$-regular $m$-form such that $\Vert \omega \Vert_c \leq 1$, where $\Vert \cdot \Vert_c$ denotes the comass norm on $\Lambda^m(\R^{m+n})$. An $m$-dimensional integral current $T$ in $\R^{m+n}$ (denoted $T\in \Ibf_m(\R^{m+n})$) is \textit{semicalibrated} by $\omega$ if $\omega(\vec{T}) = 1$ at $\Vert T \Vert$-a.e. point, where $\vec T = \frac{d T}{d\|T\|}$ denotes the polar of the canonical vector measure associated to $T$ (also denoted by $T$, abusing notation) and $\|T\|$ denotes the canonical mass measure associated to $T$.
\end{definition}

Note that we may assume that the ambient space is Euclidean, equipped with the Euclidean metric, in place of a sufficiently regular Riemannian manifold (as is often assumed when studying the regularity properties of area-minimizing currents). Indeed, this is because the presence of an ambient submanifold of $\R^{m+n}$ in which $T$ is supported may be instead incorporated into the semicalibration; see \cite[Lemma 1.1]{DLSS-uniqueness}. 

We say that $p\in \spt(T)\setminus \spt(\partial T)$ is an interior regular point if there exists a neighborhood of $p$ in which $T$ is, up to multiplicity, a smooth embedded submanifold of $\R^{m+n}$. We denote the interior regular set by $\Reg(T)$, and we refer to its complement in $\spt(T)\setminus \spt(\partial T)$ (which is a relatively closed set) as the interior singular set, denoted by $\Sing(T)$.

The regularity of area-minimizing currents and more specifically calibrated currents (where the {semicalibrating} $m$-form $\omega$ is closed) has been studied extensively \cite{Almgren_regularity,DLS_MAMS,DLS16centermfld,DLS16blowup}. Semicalibrated currents form a natural class of almost area-minimizing currents for which the underlying differential constraint has more flexibility with respect to deformations than that for calibrated currents. Typical examples of these objects are given by almost complex cycles in almost complex manifolds. Semicalibrated currents exhibit much stronger regularity properties than general almost-minimizing currents (see \cite{GS}), and have thus far been shown to share the same interior regularity as area-minimizing integral currents. Indeed, in the series of works \cite{DLSS1,DLSS2,DLSS3} by De Lellis, Spadaro and the fourth author, it was shown that interior singularities of two dimensional semicalibrated currents are isolated, much like those of two dimensional area-minimizing integral currents. It was further shown by the fourth author in \cite{Spolaor_15} that the interior singular set of an $m$-dimensional semicalibrated current has Hausdorff dimension at most $m-2$, which is consistent with Almgren's celebrated dimension estimate on the interior singular set of area-minimizing integral currents. In the case of \textit{special Legendrian cycles}, i.e. when the ambient space is $S^5 \subset \mathbb{C}^3$ and the semicalibration $\omega$ has a specific form inherited from the complex structure, it was already shown by Bellettini and Rivière that the singular set consisted only of isolated singularities \cite{BellettiniRiviere}. To the best of the authors' knowledge, the first instance of the word \textit{semicalibrated} appears in the work \cite{PumbergerRivière} by Pumberger and Rivière. There the authors prove uniqueness of tangent cones for semicalibrated integral 2-cycles using slicing techniques. These ideas were later generalized in \cite{CaniatoRivière} by Caniato and Rivière to the case of pseudo-holomorphic cycles. Furthermore, we mention work of Tian and Rivière \cite{TianRivière} showing uniqueness of tangent cones for positive integral $(1, 1)$-cycles in arbitrary almost Kähler manifolds and that the singular set consists of isolated points. Finally, in \cite{Bellettini} Bellettini proved uniqueness of tangent cone for positive integral $(p,p)$-cycles in arbitrary almost complex manifolds. 

The aim of this article is to further improve on these and establish a structural result for the interior singular set of semicalibrated currents, analogous to that obtained in the works \cite{DLSk1,DLSk2,DMS,KW1,KW2,KW3}. More precisely, our main result is the following.

\begin{theorem}\label{t:main}
    Let $T$ be an $m$-dimensional integral current in $\R^{m+n}$, semicalibrated by a $C^{2,\kappa_0}$ $m$-form $\omega$ for some $\kappa_0 \in (0,1)$. Then $\Sing(T)$ is countably $(m-2)$-rectifiable and there is a unique tangent cone to $T$ at $\Hcal^{m-2}$-a.e. point in $\Sing(T)$.
\end{theorem}

This result is interesting both from a geometric and an analytic point of view. On the geometric side, calibrated submanifolds have been central objects of study in several areas of differential geometry and mathematical physics since the seminal work of Harvey and Lawson \cite{HaLa} (we refer the reader to \cite[6.1]{MorganGMT} for a brief history of calibrations). Two primary examples are holomorphic subvarieties and special Lagrangians in Calabi-Yau manifolds, which play an important role in string theory (especially regarding mirror symmetry, cf. \cite{Jo,SYZ}), but they also emerge naturally in gauge theory (see \cite{Ti}). Semicalibrations are a natural generalization of calibrations, removing the condition $d\omega = 0$ on the calibrating form {which }is rather rigid and in particular very unstable under deformations. In fact semicalibrations were considered already in \cite{Ti} (cf. Section 6 therein) and around the same time they became rather popular in string theory when several authors directed their attention to non-Calabi-Yau manifolds (the subject is nowadays known as “flux compactification”, cf. \cite{Gr}): in that context the natural notion to consider is indeed a special class of semicalibrating forms (see for instance the works \cite{Gu1,Gu2}, where these are called quasi calibrations). The fine structure of the singular set in the $2$-dimensional case has found applications to the Castelnuovo's bound and the Gopakumar–Vafa finiteness conjecture in the recent works \cite{DoWa, DoWaIo}.

From an analytic point of view, it exhibits a striking difference with the setting of area-minimizing currents regarding notions of frequency function. In the work \cite{KW1}, Krummel and Wickramasekera introduced an intrinsic version of Almgren's frequency function for an area-minimizing current, known as \emph{planar frequency}. Under suitable decay hypotheses, they were able to show that the planar frequency in the area-minimizing setting is almost monotone, which then played a pivotal role in their analysis of interior singularities. However, in the semicalibrated setting one does \emph{not} expect almost monotonicity of the planar frequency function under the same hypotheses as in \cite{KW1}, and indeed in Part \ref{pt:counterex} we provide a simple counterexample demonstrating this. Intuitively, the reason for this is that the semicalibration condition is more flexible. In particular, the graph of any $C^{1,\alpha}$ single-valued function is a semicalibrated current (although with a semicalibrating form less regular than the one in Definition \ref{d:semical}) and, at such a level of generality, these currents are not expected to exhibit unique continuation properties, and consequently an almost monotone planar frequency function. We have been unable to adapt the approach of Krummel and Wickramasekera to the present setting, which is ultimately why we follow the ideas in \cite{DLSk1, DLSk2, DMS}. Whether or not one can prove Theorem \ref{t:main} utilizing the ideas in \cite{KW1}, and in particular finding a suitable semicalibrated `planar frequency', is an interesting question.

Finally we remark that Theorem \ref{t:main} is optimal in light of recent work of Liu \cite{Liu}.

\subsection{Structure of the article and comparison to \cite{DLSk1,DLSk2,DMS}}
In Part \ref{pt:sing-deg}, we recall the singularity degree as introduced in \cite{DLSk1}, and verify that its properties remain valid in the semicalibrated setting. Part \ref{pt:NV} is then dedicated to treating flat singular points of singularity degree strictly larger than 1, for which we may exploit the rectifiable Reifenberg methods of Naber \& Valtorta, {similarly to \cite{DLSk2}}. In Part \ref{pt:deg1}, we then treat points of singularity degree 1 and the lower strata (the latter just for the uniqueness of tangent cones), {following \cite{DMS}}. Finally, in Part \ref{pt:counterex} we present the example that demonstrates the failure of almost-monotonicity for the intrinsic planar frequency as introduced in \cite{KW1}, and draw some comparisons with the area-minimizing setting of \cite{KW1, KW2}. 

Although throughout this article we mostly follow the methods of the works \cite{DLSk1,DLSk2,DMS} of the first and third authors with Camillo De Lellis, there are a number of important differences:
\begin{itemize}
    \item Due to the presence of the semicalibration, the corresponding error term in the first variation of $T$ must be taken into account for all variational estimates. In particular, this creates an additional term in Almgren's frequency function in this setting (see Section \ref{ss:compactness}), which must be taken care of when establishing the BV estimate Theorem \ref{t:BV}. The existence of such variational errors was already taken into consideration in the works \cite{DLSS1,DLSS2,DLSS3,Spolaor_15}.
    \item When taking coarse blow-ups (see Section \ref{s:coarse}), we observe that we may assume that the term $\|d\omega\|^2_{C^0} r^{2-2\delta_3}$ is infinitesimal relative to the tilt excess $\Ebf(T,\Bbf_r)$ for $\delta_3\in (0,\delta_2)$; note that this subquadratic scaling is stronger than having the same assumption with the natural quadratic scaling of $\|d\omega\|^2_{C^0}$. The latter would be the analogue of the corresponding assumption in \cite{DLSk1}, but we require this stronger assumption for Part \ref{pt:deg1} (see Case 2 in the proof of Lemma \ref{l:decay}), and we verify that it indeed holds.
    \item We modify the original construction in \cite{DLSk2} of the intervals of flattening adapted to a given geometric sequence of radii in Part \ref{pt:NV}, instead providing one that avoids requiring the separate treatment of the case of a single center manifold and infinitely many. This modified procedure will further be useful in the forthcoming work \cite{CS}.
    \item When $T$ is merely semicalibrated, extra care needs to be taken when applying the harmonic approximation. Indeed, note that in order to apply \cite[Theorem 3.1]{DLSS1}, we require the stronger hypothesis $\|d\omega\|^2_{C^0} \leq \eps_{23} \Ebf(T,\Cbf_1)$ in place of $\Abf^2 \leq \Ebf(T,\Cbf_1)^{1/2 + 2\delta}$ in the area-minimizing case (this difference was already present in \cite{Spolaor_15}). In particular, this affects the two regimes in the case analysis within the proof of Lemma \ref{l:decay}. In order to maintain the validity of Case 1 therein, we must require that $\|d\omega\|^2_{C^0} \leq \eps_{23} \Ebf(T,\Cbf_1)$, in place of $\Abf^3 \leq \Ebf(T,\Cbf_1)$. This in turn affects the treatment of Case 2 therein; see the second bullet point above.      
    \item In order to obtain quadratic errors in $\|d\omega\|_{C^0}$ in all estimates exploiting the first variation of $T$ in Part \ref{pt:deg1}, we must employ an analogous absorption trick to that pointed out in \cite[Remark 1.10]{Spolaor_15} for area-minimizing currents. This makes arguments in Section \ref{s:nonconc} more delicate.
\end{itemize}

\subsection*{Acknowledgments} 
The authors would like to thank Camillo De Lellis for useful discussions. This research was conducted during the period P.M. was a Clay Research Fellow. D.P. acknowledges the support of the AMS-Simons Travel Grant, furthermore part of this research was performed while D.P. was visiting the Mathematical Sciences Research Institute (MSRI), now becoming the Simons Laufer Mathematical Sciences Institute (SLMath), which is supported by the National Science Foundation (Grant No. DMS-1928930). L.S. acknowledges the support of the NSF Career Grant DMS 2044954. We would like to thank Frank Morgan, and Tristan Rivière for pointing out relevant references. 

\section{Preliminaries and notation}
Let us first introduce some basic notation. $C, C_0, C_1, \dots$ will denote constants which depend only on $m,n, Q$, unless otherwise specified. For $x\in \spt(T)$, the currents $T_{x,r}$ will denote the rescalings $(\iota_{x,r})_\sharp T$, where $\iota_{x,r} (y):= \frac{y-x}{r}$ and $\sharp$ denotes the pushforward. We will typically denote (oriented) $m$-dimensional subspaces of $\R^{m+n}$ (often simply referred to as planes) by $\pi, \varpi$. For $x\in\spt(T)$, $\Bbf_r(x)$ denotes the open $(m+n)$-dimensional Euclidean ball of radius $r$ centered at $p$ in $\R^{m+n}$, while for an $m$-dimensional plane $\pi\subset\R^{m+n}$ passing through $x$, $B_r(x,\pi)$ denotes the open $m$-dimensional disk $\Bbf_r(x)\cap \pi$. $\Cbf_r(x,\pi)$ denotes the $(m+n)$-dimensional cylinder $B_r(x,\pi)\times \pi^\perp$ of radius $r$ centered at $x$. We let $\mathbf{p}_{\pi}: \mathbb R^{m+n}\to \pi$ denote the orthogonal projection onto $\pi$, while $\mathbf{p}_{\pi}^\perp$ denotes the orthogonal projection onto $\pi^\perp$. The plane $\pi$ is omitted if clear from the context; if the center $x$ is omitted, then it is assumed to be the origin. $\omega_m$ denotes the $m$-dimensional Hausdorff measure of the $m$-dimensional unit disk $B_1(\pi)$. The Hausdorff distance between two subsets $A$ and $B$ of $\R^{m+ n}$ will be denoted by $\dist (A,B)$. $\Theta(T,x)$ denotes the $m$-dimensional Hausdorff density of $T$ at $x\in \spt(T)$. For $Q\in \N$, $\Acal_Q(\R^n)$ denotes the metric space of $Q$-tuples of vectors in $\R^n$, equipped with the $L^2$-Wasserstein distance $\Gcal$ (see e.g. \cite{DLS_MAMS}). Given a map $f = \sum_{i=1}^Q \llbracket f_i \rrbracket$ taking values in $\Acal_Q(\R^n)$, we use the notation $\boldsymbol\eta\circ f$ to denote the $\R^n$-valued function $\frac{1}{Q}\sum_{i=1}^Q f_i$.

As for area-minimizing currents, we primarily focus our attention on the \emph{flat singular points} of $T$, namely, those at which there exists a flat tangent cone $Q\llbracket \pi_0 \rrbracket$ for some $m$-dimensional (oriented) plane $\pi_0$. By localizing around a singular point and rescaling, we may without loss of generality work under the following underlying assumption throughout.
\begin{assumption}\label{a:main} 
$m\geq 3$, $n \geq 2$ are integers. $T$ is an $m$-dimensional integral current in $\Bbf_{7\sqrt{m}}$ with $\partial T\mres \Bbf_{7\sqrt{m}} = 0$. There exists a $C^{2,\kappa_0}$ semicalibration $\omega$ on $\R^{m+n}$ such that $T$ is semicalibrated by $\omega$ in $\Bbf_{7\sqrt{m}}$, with
\begin{equation}\label{e:domega}
    \|d\omega\|_{C^{1,\kappa_0}(\Bbf_{7\sqrt{m}})} \leq \bar \eps,
\end{equation}
where $\bar\eps \leq 1$ is a small positive constant which will be specified later.
\end{assumption}

Recall that if $T$ satisfies Assumption \ref{a:main}, then in particular $T$ is $\Omega$-minimial as in \cite[Definition 1.1]{DLSS1} for some $\Omega >0$, namely
\begin{equation} \label{eq: Omega Minimal}
    \mathbf{M}(T) \leq \mathbf{M}(T + \partial S) + \Omega \mathbf{M}(S), 
\end{equation}
for every $S \in \mathbf{I}_{m + 1}(\R^{m+n})$ with compact support, and in particular one can take $\Omega = \Vert d\omega \Vert_{C^0}$. In addition, if $T$ satisfies Assumption \ref{a:main} then we have the first variation identity
\begin{equation} \label{eq: first variation}
    \delta T(X) = T(d\omega \mres X),  
\end{equation}
where $X \in C_{c}^{\infty}(\mathbb{R}^{m+n} \setminus \spt(\partial T); \R^{m+n})$, and where $\delta T$ denotes the first variation of $T$:
\begin{align} \label{first variation}
    \delta T(X) = \int \divergence_{\vec T(q)} X(q) \, d \Vert T \Vert(q)\, ,
\end{align}
where $\vec T(q)$ is the oriented (approximate) tangent plane to $T$ at $q$. 

Recall that the tilt excess $\Ebf(T,\Cbf_r(x,\pi),\varpi)$ relative to an $m$-dimensional plane $\varpi$ in a cylinder $\Cbf_r(x,\pi)$ is defined by
\[
    \Ebf(T,\Cbf_r(x,\pi),\varpi) := \frac{1}{2\omega_m r^{m}}\int_{\Cbf_r(x,\pi)} |\vec{T} - \vec{\varpi}|^2 d\|T\|.
\]
The (optimal) tilt excess in $\Cbf_r(x,\pi)$ is in turn defined by
\[
    \Ebf(T,\Cbf_r(x,\pi)) := \inf_{\text{$m$-planes $\varpi$}} \Ebf(T,\Cbf_r(x,\pi),\varpi).
\]
The quantities $\Ebf(T,\Bbf_r(x),\varpi)$ and $\Ebf(T,\Bbf_r(x))$ are defined analogously.

\subsection{Intervals of flattening and compactness procedure}\label{ss:compactness}
As in \cite{DLS16blowup,Spolaor_15}, we introduce a countable collection of disjoint intervals of radii $(s_j,t_j] \subset (0,1]$, for $j\in \N\cup \{0\}$ and $t_0 =1$, referred to as \emph{intervals of flattening}, such that for $\eps_3 > 0$ fixed as in \cite{Spolaor_15} we have
\[
    \Ebf(T, \Bbf_{6\sqrt{m}r}) \leq \eps_3^2, \qquad \Ebf(T_{0,t_j}, \Bbf_r) \leq C \boldsymbol{m}_{0,j} r^{2-2\delta_2} \qquad \forall r\in \left(\tfrac{s_j}{t_j},3\right],
\]
where
\begin{equation}\label{e:m_0}
    \boldsymbol{m}_{0,j}:=\max\{\Ebf(T,\Bbf_{6\sqrt{m}t_j}), \bar\eps^2 t_j^{2-2\delta_2}\},
\end{equation}
and $\delta_2$ is fixed as in \cite{Spolaor_15}. Observe that this definition of $\mbf_{0,j}$ comes from the observation that if $T$ is $\|d\omega\|_{C^{0}}$-minimal in $\Bbf_{7\sqrt{m}}$, then $T_{0,t_j}$ is $t_j\|d\omega\|_{C^0}$-minimal in $\Bbf_{7\sqrt{m}}$, together with \eqref{e:domega}, the estimates in \cite[Theorem 1.4]{DLSS1} and the observation that 
\begin{equation}\label{e:domega-balls}
    \|\iota_{0,t_j}(d\omega)\|_{C^0(\Bbf_{6\sqrt{m}})} = \|d\omega \circ \iota_{0,t_j}\|_{C^0(\Bbf_{6\sqrt{m}})} = \|d\omega\|_{C^0(\Bbf_{6\sqrt{m}t_j})} \leq C \|d\omega\|_{C^0(\Bbf_{6\sqrt{m}})} t_j^2.
\end{equation}
We will therefore henceforth work under the following assumption, allowing us to indeed iteratively produce the above sequence of intervals.

\begin{assumption}\label{a:main-2}
    $T$ and $\omega$ are as in Assumption \ref{a:main}. The origin is a flat singular point of $T$ and $\Theta(T,0) = Q\in \N_{\geq 2}$. The parameter $\bar \eps$ is chosen small enough to ensure that $\boldsymbol{m}_{0,0} \leq \eps_3^2$.
\end{assumption}

Following the procedure in \cite[Section 6.2]{Spolaor_15} (see also \cite[Section 2]{DLS16blowup}) with this amended choice of $m_{0,j}$, we use the center manifold construction in \cite[Part I, Section 2]{Spolaor_15} to produce a sequence $\Mcal_j$ of center manifolds for the rescalings $T_{0,t_j}$ with corresponding normal approximations $N_j:\Mcal_j \to \Acal_Q(T\Mcal_j^\perp)$, whose multigraphs agree with $T_{0,t_j}$ in $\Bbf_3\setminus \Bbf_{s_j/t_j}$ over an appropriately large proportion of $\Mcal_j\cap(\Bbf_3\setminus \Bbf_{s_j/t_j})$. By a rotation of coordinates, we may assume that the $m$-dimensional planes $\pi_j$ over which we parameterize $\Mcal_j$ are identically equal to the same fixed plane $\pi_0 \equiv \R^m\times \{0\} \subset \R^{m+n}$. We refer the reader to \cite[Proposition 6.5]{Spolaor_15} or \cite[Proposition 2.2]{DLS16blowup} for the basic properties of the intervals of flattening. Given a center manifold $\Mcal$ and a point $x\in \Mcal$, we will let $\Bcal_r(x)$ denote the geodesic ball $\Bbf_r(x)\cap\Mcal$ of radius $r$ in $\Mcal$. It will always be clear from context which particular center manifold we are using for such a ball.

Given a flat singular point of $T$ with density $Q\in \N$ (denoted by $x\in \Ffrak_Q(T)$), we will use the terminology \emph{blow-up sequence of radii around $x$} to refer to a sequence of scales $r_k\downarrow 0$ such that $T_{x,r_k}\mres \Bbf_{6\sqrt{m}}\toweakstar Q\llbracket \pi \rrbracket$ for some $m$-dimensional plane $\pi$. If $x=0$, we will simply call this a blow-up sequence of radii, with no reference to the center. Observe that for any blow-up sequence of radii $r_k$, for each $k$ sufficiently large there exists a unique choice of index $j(k)$ such that $r_k \in (s_{j(k)},t_{j(k)}]$.

Under the validity of Assumption \ref{a:main-2}, given a blow-up sequence of radii $r_k$, we will henceforth adopt the notation
\begin{itemize}
    \item $T_k$ for the rescaled currents $T_{0,t_{j(k)}}\mres \Bbf_{6\sqrt{m}}$;
    \item $\Mcal_k$ and $N_k$ rescpectively for the center manifolds $\Mcal_{j(k)}$ and the normal approximations $N_{j(k)}$;
    \item $\bphi_k$ for the map parameterizing the center manifold $\Mcal_k$ over $B_3(\pi_0)$ (see \cite[Theorem 2.13]{Spolaor_15});
    \item $\pbf$ for the orthogonal projection map to $\Mcal$ (see \cite[Assumption 2.1]{DLS16centermfld}).
\end{itemize}
In addition, let $s=\tfrac{\bar s_k}{t_{j(k)}} \in \left(\frac{3 r_k}{t_{j(k)}}, \frac{3r_k}{t_{j(k)}}\right]$ be the scale at which the reverse Sobolev inequality \cite[Corollary 7.9]{Spolaor_15} (see also \cite[Corollary 5.3]{DLS16blowup}) holds for $r=\frac{r_k}{t_{j(k)}}$. Then let $\bar r_k = \tfrac{2\bar s_k}{3t_{j(k)}}\in \left(\tfrac{r_k}{t_{j(k)}}, \frac{2r_k}{t_{j(k)}}\right]$, and in turn define the corresponding additionally rescaled objects
\[
    \bar T_k = (T_k)_{0,\bar r_k} = (\iota_{0,\bar r_k t_{j(k)}})_\sharp T \mres \Bbf_{6\sqrt{m}\bar{r}_k^{-1}}, \qquad \bar\Mcal_k = \iota_{0,\bar r_k}(\Mcal_k),
\]
together with the maps $\bPhi_k(x) := (x,\bphi_k(\bar r_k x))$ parameterizing the graphs of the rescaled center manifolds and the rescaled normal approximations $\bar N_k:\bar\Mcal_k\to \R^{m+n}$ defined by
\begin{equation}\label{e:rescaled-N}
    \bar N_k(x) := \frac{N_k(\bar r_k x)}{\bar r_k}.
\end{equation}
Consequently we let $u_k: B_3 \equiv B_3(\pi_0)\to \Acal_Q(\R^{m+n})$ be defined by
\[
    u_k := \frac{\bar N_k\circ \mathbf{e}_k}{\|\bar N_k \|_{L^2(\Bcal_{3/2})}},
\]
where $\mathbf{e}_k$ denotes the exponential map from $B_3\subset \pi_0\cong T_{\bar r_k^{-1} \mathbf{\Phi}_k(0)}\bar \Mcal_k$ to  $\bar \Mcal_k$. In light of the reverse Sobolev inequality \cite[Corollary 7.9]{Spolaor_15} implies that the sequence $u_k$ is uniformly bounded in $W^{1,2}(B_{3/2})$. Then, by \cite[Theorem 8.2]{Spolaor_15} (see also \cite[Theorem 6.2]{DLS16blowup}), up to extracting a subsequence, there exists a Dir-minimizer $u\in W^{1,2}(B_{3/2}(\pi_0);\Acal_Q(\pi_0^\perp))$ such that
\begin{itemize}
    \item $\boldsymbol\eta\circ u =0$;
    \item $\|u\|_{L^2(B_{3/2})}=1$;
    \item $u_k \to u$ strongly in $L^2\cap W^{1,2}_{\loc} (B_{3/2})$.
\end{itemize}

Recall that for Dir-minimizers $u: \Omega \subset \R^m \to \Acal_Q(\R^m)$ on an open domain $\Omega$, we may consider a regularized variant of Almgren's frequency function, defined by
\[
    I_u(x,r) := \frac{rD_u(x,r)}{H_u(x,r)}, \qquad r\in (0,\dist(x,\partial \Omega))
\]
where
\[
    H_u(x,r) = -\int \frac{|u(y)|^2}{|y-x|} \phi'\left(\frac{|y-x|}{r}\right)\, dy, \qquad D_u(x,r) = \int |Du(y)|^2 \phi\left(\frac{|y-x|}{r}\right)\, dy,
\]
and $\phi:[0,\infty)\to[0,1]$ is a monotone Lipschitz function that vanishes for all $t$ sufficiently large and is identically equal to $1$ for all $t$ sufficiently small. Similarly to the classical frequency, which formally corresponds to taking $\phi=\mathbf{1}_{[0,1]}$, $r\mapsto I_u(x,r)$ is monotone non-decreasing for each $x\in \Omega$, and takes a constant value $\alpha$ if and only if $u$ is radially $\alpha$-homogeneous about $x$ (see e.g. \cite[Section 3.5]{DLS_MAMS}). In particular, the limit
\[
    I_{u}(x,0) := \lim_{r\downarrow 0} I_u(x,r)
\]
exists and in fact is independent of the choice of $\phi$. We will henceforth fix the following convenient choice of $\phi$:
\begin{equation}\label{e:def_phi}
\phi (t) =
\left\{
\begin{array}{ll}
1 \qquad &\mbox{for $0\leq t \leq \frac{1}{2}$},\\
2-2t \quad &\mbox{for $\frac{1}{2}\leq t \leq 1$},\\
0 &\mbox{otherwise}\, .
\end{array}
\right.
\end{equation}
When $x=0$, we omit the dependency on $x$ for $I$, $H$ and $D$.

Now let us define the natural regularized frequency associated to the graphical approximations for a semicalibrated current $T$ satisfying Assumption \ref{a:main-2}. Given a center manifold $\Mcal\equiv \Mcal_j$ and a corresponding normal approximation $N:\Mcal \to \Acal_Q(\R^{m+n})$, we define the regularized frequency $\Ibf_N$ of $N$ at a given center $x\in\Mcal$ and scale $r>0$ by
\[
    \Ibf_N(x,r) := \frac{r\mathbf{\Gamma}_N(x,r)}{\Hbf_N(x,r)},
\]
where $\mathbf{\Gamma}_N(x,r) = \Dbf_N(x,r) + \Lbf_N(x,r)$ with
\begin{align*}
    \Hbf_N(x,r) &:= {-}\int_\Mcal \frac{|N(y)|^2}{d(y,x)} |\nabla_{y} d(y,x)|^2\phi'\left(\frac{d(y,x)}{r}\right)\, d\Hcal^m(y), \\
    \Dbf_N(x,r) &= \int |DN(y)|^2 \phi\left(\frac{d(y,x)}{r}\right)\, d\Hcal^m(y), \\
    \Lbf_N(x,r) &= \sum_{i=1}^Q \sum_{l=1}^m (-1)^{l+1} \int \langle D_{\xi_l} N_i(y) \wedge \hat\xi_l(y) \wedge N_i(y), d\omega\circ\iota_{0,t_j}(y)\rangle \phi\left(\frac{d(y,x)}{r}\right) d\Hcal^m(y).
\end{align*}
Here, $\hat\xi_l = \xi_1\wedge \cdots \wedge \xi_{l-1}\wedge \xi_{l+1}\wedge\cdots \wedge \xi_m$ for an orthonormal frame $\{\xi_j\}_{j=1}^m$ of $T\Mcal$ and $d$ is the geodesic distance on the center manifold. We will often write $\nabla d(x,y)$ to denote the derivative ${\nabla}_y d(x,y)$. If $x=0$, we will omit the dependency on the center. Recall that the presence of the additional term $\Lbf_N$ in the frequency (in contrast to that for area-minimizing integral currents) is due to the term $T(d\omega \mres X)$ in the first variation $\delta T(X)$ for $T$; see \cite{Spolaor_15} for more details. 
\begin{remark}
    Note in $\Lbf$ above the presence of the scaling $\iota_{0,t_j}$. This is due to the error in the first variation being for $T_{0,t_j}$, and not for $T$. This is consistent with the quadratic scaling we expect for the $\|d\omega\|_0$ term in our definition of $\mbf_0$. 
\end{remark}

\part{Singularity degree of flat singular points}\label{pt:sing-deg}

\section{Main results}
Following \cite{DLSk1}, we define a \emph{fine blow-up} $u$ to be any Dir-minimizer obtained through the compactness procedure in Section \ref{ss:compactness} along a blow-up sequence of radii $r_k$, and we let
\[
    \Fcal(T,0) := \{I_u(0) : \text{$u$ is a fine blow-up along some sequence $r_k \downarrow 0$}\}
\]
denote the set of \emph{frequency values} of $T$ at 0. We further recall the notion of \emph{singularity degree} introduced in \cite{DLSk1}:

\begin{definition}
    The singularity degree of $T$ at 0 is defined as
    \[
        \Irm(T,0) := \inf \Fcal(T,0).
    \]
\end{definition}
Of course, one may analogously define the set of frequency values and the singularity degree at any other point $x\in \Ffrak(T)$ by instead considering fine blow-ups taken around the center point $x$ in place of 0, and thus all of the results in this part clearly hold for any $x\in \Ffrak_Q(T)$ in place of $0$.

Let us now state the main result of this part, which concerns the main properties of the singularity degree.

\begin{theorem}\label{t:sing-degree-main}
    Suppose that $T$ satisfies Assumption \ref{a:main-2}. Then
    \begin{itemize}
        \item[(i)] $\Irm(T,0) \geq 1$ and $\Fcal(T,0) = \{\Irm(T,0)\}$;
        \item[(ii)] All fine blow-ups are radially homogeneous with degree $I(T,0)$;
        \item[(iii)] if $s_{j_0}=0$ for some $j_0\in \N$ then $\lim_{r\downarrow 0} \Ibf_{N_{j_0}}(r) = \Irm(T,0)$;
        \item[(iv)] if, conversely, there are infinitely many intervals of flattening $(s_k,t_k]$, the functions $\Ibf_{N_j}$ converge uniformly to the constant function $\Irm(T,0)$ when $\Irm(T,0)>1$, while when $\Irm(T,0)=1$, $\lim_{k \to \infty} \Ibf_{j(k)}(\tfrac{r_k}{t_{j(k)}}) = \Irm(T,0) =1$ for every blow-up sequence of radii $r_k$;
        \item[(v)] if $\Irm(T,0)>1$ then the rescalings $T_{0,r}$ converge polynomially fast to a unique flat tangent cone $Q\llbracket \pi \rrbracket$ as $r\downarrow 0$;
        \item[(vi)] if additionally $\Irm(T,0)> 2-\delta_2$ then $s_{j_0} = 0$ for some $j_0\in\N$;
        \item[(vii)] if $\Irm(T,0) < 2-\delta_2$ then there are infinitely many intervals of flattening and $\inf_j \frac{s_j}{t_j} > 0$.
    \end{itemize}
\end{theorem}

In Part 2, we will then follow the arguments in \cite{DLSk2} (based on the seminal work \cite{NV_Annals}) to prove the following.
\begin{theorem}
    Let $T$ and $\omega$ be as in Theorem \ref{t:main}. Then the set $\{x\in \Ffrak(T): \Irm(T,x)>1\}$ is countably $(m-2)$-rectifiable.
\end{theorem}

Recall that by the work \cite{NV_varifolds} of Naber \& Valtorta, for each $k=0,\dots, m$, the \emph{$k$-th stratum} $\Scal^{(k)}(T)$, defined to be the set of all points $x\in \spt(T)\setminus \spt(\partial T)$ such that for any tangent cone $S$ at $x$ we have
\[
    \dim(\{y:(\tau_y)_\sharp S = S\}) \leq k,
\]
is countably $k$-rectifiable. Here, $\tau_y(p) := p+y$ denotes the map that translates by $y$.

In Part 3 we then complete the proof of Theorem \ref{t:main} by showing the following (cf. \cite{DMS}).
\begin{theorem} \label{t: main negligible}
    Let $T$ and $\omega$ be as in Theorem \ref{t:main}. Then the set $\{x\in \Ffrak(T): \Irm(T,x)=1\}$ is $\Hcal^{m-2}$-null. Moreover, the tangent cone is unique at $\Hcal^{m-2}$-a.e. point in $\Scal^{(m-2)}(T)$.
\end{theorem}
{Combining these results then gives Theorem \ref{t:main}.}

The starting point for studying the properties of Almgren's frequency function with respect to varying normal approximations is the following result, which provides uniform upper and lower frequency bounds over all the intervals of flattening around a given flat singular point.

\begin{theorem}\label{t:freqbds}
    Suppose that $T$ satisfies Assumption \ref{a:main-2}. Then there exist constants $c_0 = c_0(m,n,Q,T) > 0$, $C=C(m,n,Q,T)>0$ and $\alpha=\alpha(m,n,Q)>0$, such that
    \begin{align}
        \Ibf_{N_j}(r) &\geq c_0 \qquad \forall r\in \left(\frac{s_j}{t_j},3\right], \\
        \Ibf_{N_j}(a) &\leq e^{Cb^\alpha} \Ibf_{N_j}(b) \qquad \forall r\in \left(\frac{s_j}{t_j},3\right]. \label{e: freq-am}
    \end{align}
    Moreover,
    \[
        0< \inf_j \inf_{r\in (\tfrac{s_j}{t_j}, 3]}\Ibf_{N_j}(r) \leq \sup_j \sup_{r\in (\tfrac{s_j}{t_j},3]}\Ibf_{N_j}(r) < +\infty.
    \]
\end{theorem}
The uniform upper bound of Theorem \ref{t:freqbds} was established in \cite[Theorem 7.8]{Spolaor_15} (see also \cite[Theorem 5.1]{DLS16blowup}). For the uniform lower bound, we refer the reader to \cite[Theorem 7.8]{Sk21}. Although this is only proven therein in the case where $T$ is area-minimizing, one may easily observe that the proof in fact works in exactly the same way when $T$ is merely semicalibrated, since all of the preliminary results required (e.g. \cite[Theorem 1.5]{Spolaor_15} and the estimates \cite[Proposition 7.5]{Spolaor_15}) hold here also. In particular, observe that including the term $\Lbf_N$ in the frequency ensures that the variational error terms for the frequency are of exactly the same form as those in the area-minimizing case.

\begin{remark}
Given Theorem \ref{t:freqbds}, observe that the arguments in \cite{Sk21}, rewritten for semicalibrated currents (namely, replacing all results from \cite{DLS14Lp,DLS16centermfld,DLS16blowup} with their counterparts from \cite{Spolaor_15}), yields a local upper Minkowski dimension estimate of $m-2$ as obtained in \cite{Sk21} for area-minimizing integral currents.
\end{remark}

In order to derive the conclusions of Theorem \ref{t:sing-degree-main}, we wish to sharpen the uniform bounds of Theorem \ref{t:freqbds} to a quantitative control on the radial variations of the frequency function, across uninterrupted strings of intervals of flattening. With this in mind, we recall the following definition of the \emph{universal frequency} from \cite{DLSk1}.

\begin{definition}\label{d:univ-freq}
    Suppose that $T$ is as in Assumption \ref{a:main-2} and let $\{(s_k,t_k]_{k=j_0}^J\}$ be a sequence of intervals of flattening with coinciding endpoints (i.e. $s_k = t_{k+1}$ for $k=j_0,\dots,J-1$), with corresponding center manifolds $\Mcal_k$ and normal approximations $N_k$. For $r\in (s_J, t_{j_0}]$, let
    \begin{align*}
        \Ibf(r) &:= \Ibf_{N_k}\big(\tfrac{r}{t_k}\big)\mathbf{1}_{(s_k,t_k]}(r), \\
        \Hbf(r) &:= \Hbf_{N_k}\big(\tfrac{r}{t_k}\big)\mathbf{1}_{(s_k,t_k]}(r), \\
        \Dbf(r) &:= \Dbf_{N_k}\big(\tfrac{r}{t_k}\big)\mathbf{1}_{(s_k,t_k]}(r), \\
        \Lbf(r) &:= \Lbf_{N_k}\big(\tfrac{r}{t_k}\big) \mathbf{1}_{(s_k,t_k]}(r). 
    \end{align*}
    We refer to $\Ibf$ as the \emph{universal frequency function}, whenever it is well-defined.
\end{definition}

We have the following frequency BV estimate on the universal frequency function, which is not only a crucial tool for the proof of Theorem \ref{t:sing-degree-main}, but will also be useful in its own right in establishing the rectifiability of the points with singularity degree strictly larger than 1 in Part 2.

\begin{theorem}\label{t:BV}
    There exists $\gamma_4=\gamma_4(m,n,Q)>0$ and $C=C(m,n,Q)>0$ such that the following holds. Let $\{(s_k,t_k]_{k=j_0}^J\}$ be a sequence of intervals of flattening with coinciding endpoints. Then $\log(\Ibf +1)\in \BV((s_J,t_{j_0}])$ with the quantitative estimate
    \begin{equation}
        \left|\left[\frac{d\log(\Ibf+1)}{dr}\right]_-\right|\big((s_J,t_{j_0}]\big) \leq C \sum_{k=j_0}^{J} \mbf_{0,k}^{\gamma_4}.
    \end{equation}
    In addition, if $(a,b]\subset (s_k,t_k]$ for some interval of flattening $(s_k,t_k]$, we have
    \begin{equation}
        \left|\left[\frac{d\log(\Ibf+1)}{dr}\right]_-\right|\big((a,b]\big) \leq C \left(\frac{b}{t_k}\right)^{\gamma_4}\mbf_{0,k}^{\gamma_4}.
    \end{equation}
\end{theorem}

\section{Coarse blow-ups}\label{s:coarse}
It will be convenient to consider an alternative type of blow-up to a fine blow-up, avoiding reparameterization to center manifolds; we follow the setup of \cite[Section 3.1]{DLSk1}. Consider a blow-up sequence of radii $r_k$ and the associated sequence $T_{0,r_k}$ of rescaled currents. We may assume without loss of generality that $T_{0,r_k} \toweakstar Q\llbracket \pi_0 \rrbracket$ in $\Bbf_{4}$.
\begin{remark}
    Compared to the set up of \cite{DLSk1}, we are taking $M = 1/2$, which is sufficient for our purposes here. 
\end{remark}

Notice that for $\bar r_k := \frac{r_k}{t_k}$, in light of the stopping condition \cite[(6.10)]{Spolaor_15} for the intervals of flattening, we have $\Bbf_L\subset\Cbf_{2\bar r_k}$ for any Whitney cube $L\in \Wscr^{(j(k))}$ with $L\cap\bar{B}_{\bar r_k}(\pi_0)\neq \emptyset$ (see \cite[Section 2.2]{Spolaor_15}). Let $\varpi_k$ denote a sequence of planes such that $\Ebf(T_{0,r_k},\Bbf_{4\bar r_k},\varpi_k) = \Ebf(T_{0,r_k},\Bbf_{4\bar r_k})$. Observe that for $k$ sufficiently large, the height bound \cite[Theorem 1.5]{Spolaor_15} guarantees that
\[
    \Ebf(T_{0,r_k}, \Cbf_{2},\varpi_k) \leq \Ebf(T_{0,r_k}, \Bbf_4) =: E_k \to 0.
\]
In particular, $\varpi_k \to \pi_0$ (locally in Hausdorff distance). By replacing $T_{0,r_k}$ by its pushforward under a rotation mapping $\varpi_k$ to a plane parallel to $\pi_0$, which is converging to the identity, we may therefore assume that $\varpi_k = \pi_0$.

We may then ensure that for all $k$ sufficiently large we have $E_k < \eps_1$, where $\eps_1>0$ is the threshold of \cite[Theorem 1.4]{DLSS1}, which in turn yields a sequence of Lipschitz approximations 
\begin{equation}\label{e:f_k-coarse}
    f_k:B_{1/2}(\pi_0)\to \Acal_Q(\pi_0^\perp)
\end{equation}
for $T_{0,r_k}$. Define the normalizations
\[
    \bar f_k := \frac{f_k}{E_k^{1/2}}.
\]
We will work under the additional assumption that
\begin{equation}\label{e:domega-scaling}
    \|d\omega\|_{C^{1,\kappa_0}}^2 r_k^{2-2\delta_3} = o(E_k)\,,
\end{equation}
for a fixed choice of parameter $\delta_3 \in (0, \delta_2)$.
\begin{remark}
    Notice that this is a slightly stronger hypothesis that the corresponding assumption \cite[(9)]{DLSk1}. The reason for asking for merely almost-quadratic scaling on the left-hand side of \eqref{e:domega-scaling} will become apparently in Part \ref{pt:deg1}; see Remark \ref{r:error-loss}.
\end{remark}
Note that \eqref{e:domega-scaling} need not necessarily hold in general, but we will only need to consider cases where it is indeed true.

In light of \cite[Theorem 1.4, Theorem 3.1]{DLSS1}, we may thus conclude that up to extracting a subsequence, there exists a Dir-minimizer $\bar f: B_1(\pi_0) \to \Acal_Q(\pi_0^\perp)$ with $\bar f(0) = Q\llbracket 0 \rrbracket$ such that
\[
    \bar f_k \to \bar f \qquad \text{in $W^{1,2}_{\loc}\cap L^2(B_1(\pi_0))$}.
\]
Recalling \cite{DLSk1}, we refer to such a map $\bar f$ as a \emph{coarse blow-up} of $T$ at $0$, and we say that $\bar f$ is non-trivial if it is not identically equal to $Q\llbracket 0 \rrbracket$. We in turn define the average free part
\[
    v(x) := \sum_{i=1}^Q \llbracket \bar f_i(x) - \boldsymbol\eta\circ\bar f(x)\rrbracket
\]
for $\bar f$. As usual, one may analogously define a coarse blow-up and its average-free part at another point $x\in \Ffrak(T)$ under the assumption \eqref{e:domega-scaling}.

Observe that unlike for a fine blow-up, {it could be} that $\bar f \equiv Q\llbracket \boldsymbol\eta \circ \bar f \rrbracket$. Indeed, one may construct examples of such behavior from holomorphic curves in $\C^2$ that are of the form $\{(w,z): w^Q = z^p\}$ for non-integer ratios $p/Q$ larger than 2 (see e.g. \cite[Remark 4.2]{DLSk1}, and \cite{DL-survey-JDG}).

We have the following frequency lower bound for coarse blow-ups, which follows from a Hardt-Simon type estimate (see \cite[Theorem 3.2]{DLSk1}).

\begin{theorem}\label{t:coarse}
    Let $T$ be as in Assumption \ref{a:main-2}. If $\bar f$ is a non-trivial coarse blow-up and $v$ is its average-free part, then $I_{\bar f}(0)\geq 1$ and if $v$ is not identically zero, then $I_v(0)\geq 1$.
\end{theorem}
We refer the reader to \cite{DLSk1} for the proof of this, with the observation that the only differences in the argument therein are
\begin{itemize}
    \item as in \cite{DLSk1}, the fact that $\bar f$ is non-trivial is equivalent to the existence of a radius $\rho>0$ and a constant $\bar c>0$ such that
    \[
        \liminf_{k\to\infty} \frac{\Ebf(T_{0,r_k},\Cbf_\rho,\pi_k)}{E_k} \geq \bar c.
    \]
    \item in the proof of the preliminary lemma \cite[Lemma 3.3]{DLSk1}, the estimate (9) therein is replaced with \eqref{e:domega-scaling} here, and the Lipschitz approximation and surrounding estimates are instead taken from \cite{DLSS1};
    \item the estimate (15) therein for a semicalibrated current follows from \cite[Proposition 2.1]{DLSS-uniqueness} in place of the classical monotonicity formula for mass ratios for stationary integral varifolds (namely, there is the presence of a higher order error term, which vanishes as the inner radius is taken to zero).
\end{itemize}

In light of Theorem \ref{t:coarse}, the validity of the lower bound on the singularity degree in Theorem \ref{t:sing-degree-main} is therefore reduced to the following.

\begin{proposition}\label{p:coarse=fine}
    Suppose that $T$ is as in Assumption \ref{a:main-2}. Let $r_k \in (s_{j(k)},t_{j(k)}]$ be a blow-up sequence of radii with
    \begin{equation}\label{e:intervals-nondegen}
        \liminf_{k\to\infty} \frac{s_{j(k)}}{r_k} >0.
    \end{equation}
    Then \eqref{e:domega-scaling} holds, and the coarse blow-up $\bar f$ along (a subsequence of) $r_k$ is well-defined. Moreover, for the average-free part $v$ of $\bar f$ and a corresponding fine blow-up $u$ along (a further subsequence of) $r_k$, we have
    \[
        v=\lambda u \qquad \text{for some $\lambda >0$.}
    \]
    In particular, $I_u(0) \geq 1$.
    \end{proposition}

    \begin{remark}
        In the case where \eqref{e:intervals-nondegen} fails, we might have that $v$ is trivial while $u$ is not, precisely because of the possibility that $\bar f = Q\llbracket \boldsymbol\eta\circ \bar f\rrbracket$ in this case; see the above discussion.
    \end{remark}

    Before discussing the proof of Proposition \ref{p:coarse=fine}, let us point out the following result, which one obtains as a simple consequence a posteriori, after obtaining the conclusions of Theorem \ref{t:sing-degree-main}, in light of the classification of homogeneous harmonic functions. This corollary will be exploited in Part 3.

    \begin{corollary}\label{c:coarse}
        Let $T$ be as in Assumption \ref{a:main-2} and suppose that $\Irm(T,0) < 2-\delta_2$. Then, assuming the conclusions of Theorem \ref{t:sing-degree-main}, any coarse blow-up $\bar f$ at 0 is non-trivial, average-free and $\Irm(T,0)$-homogeneous. Moreover, for each $\gamma > 2(\Irm(T,0) - 1)$ we have the lower decay bound
        \[
            \liminf_{r\downarrow 0} \frac{\Ebf(T,\Bbf_r)}{r^\gamma} > 0,
        \]
    and there exists $r_0= r_0(Q,m,n, T)>0$ such that
    \[
        \Ebf(T,\Bbf_r) \geq \left(\frac{r}{s}\right)^\gamma \Ebf(T,\Bbf_s) \qquad \forall r < s < r_0.
    \]
    \end{corollary}
    Observe that given the conclusions (ii) and (vii) of Theorem \ref{t:sing-degree-main}, the proof of Corollary \ref{c:coarse} is exactly the same as that of \cite[Corollary 4.3]{DLSk1}, when combined again with the aforementioned observation that $T_{0,r_j}$ is $r_j\|d\omega\|_{C^0}$-minimal in $\Bbf_{7\sqrt{m}}$, which allows one to verify that the property \eqref{e:domega-scaling} is preserved under rescalings of blow-up sequences (see \cite[Lemma 3.3]{DLSk1}). We thus omit the details here.

    \subsection{Proof of Proposition \ref{p:coarse=fine}}
    The proof of Proposition \ref{p:coarse=fine} follows analogous reasoning as that of \cite[Proposition 4.1]{DLSk1}.

    First of all, recall the following Lemma from \cite{DLSk1}. Note that it does not rely on any properties of the center manifold other than its regularity, so clearly remains unchanged in this setting.

    \begin{lemma}\label{l:reparametrization-lemma}
    There are constants $\kappa=\kappa (m,n,Q)>0$ and $C=C (m,n,Q)>0$ with the following property. Consider:
    \begin{itemize}
        \item A Lipschitz map $g: \mathbb R^m \supset B_2 \to \mathcal{A}_Q (\mathbb R^n)$ with $\|g\|_{C^0} + {\rm Lip}\, (g) \leq \kappa$;
        \item A $C^2$ function $\boldsymbol{\varphi} : B_2 \to \mathbb R^n$ with $\boldsymbol{\varphi} (0) = 0$ and $\|D\boldsymbol{\varphi}\|_{C^1} \leq \kappa$;
        \item The function $f (x) = \sum_i \llbracket \boldsymbol{\varphi} (x) + g_i (x)\rrbracket$ and the manifold $\mathcal{M} := \{(x, \boldsymbol{\varphi }(x))\}$;
        \item The maps $N, F: \mathcal{M}\cap \Cbf_{3/2} \to \mathcal{A}_Q (\mathbb R^{m+n})$ given by \cite[Theorem 5.1]{DLS_multiple_valued}, satisfying $F (p) = \sum_i \llbracket p+ N_i (p)\rrbracket$, $N_i(p)\perp T_p \mathcal{M}$, and $\boldsymbol{T}_F \mres \Cbf_{5/4} = \Gbf_f \mres \Cbf_{5/4}$.
    \end{itemize}
    If we denote by $\tilde{g}$ the multi-valued map $x\mapsto \tilde{g} (x) = \sum_i \llbracket (0, g_i (x))\rrbracket \in \mathcal{A}_Q (\mathbb R^{m+n})$, then
    \begin{equation*}
        \mathcal{G} (N (\boldsymbol{\varphi} (x)), \tilde{g} (x)) \leq C \|D \boldsymbol{\varphi}\|_{C^0} (\|g\|_{C^0} + \|D \boldsymbol{\varphi}\|_{C^0}) \qquad \forall x\in B_1\, .
    \end{equation*}
    \end{lemma}

    In addition, we also have the same comparison estimates as in \cite[Lemma 4.5]{DLSk1}:

    \begin{lemma}\label{l:comparison}
        Suppose that the assumptions of Proposition \ref{p:coarse=fine} are satisfied along a sequence of radii $r_k$. Then \eqref{e:domega-scaling} holds, and
        \begin{itemize}
            \item[(i)] for $\hbf_k = \|\bar N_k\|_{L^2(\Bcal_{3/2})}$ with $\bar N_k$ as in \eqref{e:rescaled-N}, we have
            \begin{equation}
                0 < \liminf_{k\to\infty}\frac{\hbf_k^2}{E_k} \leq \limsup_{k\to\infty} \frac{\hbf_k^2}{E_k} < +\infty;
            \end{equation}
            \item[(ii)] for $f_k$ defined as in \eqref{e:f_k-coarse} and the map $\bar{\bphi}_k$ on $B_2(0,\pi_k)$ such that 
            $$\graph(\bar{\bphi}_k)\cap \Cbf_{3/2}(0,\pi_k) = \iota_{0,r_k/t_{j(k)}}(\Mcal_k),$$
            we have
            \begin{equation}
                \int_{B_{3/2}}|\bar{\bphi}_k- \boldsymbol\eta\circ f_k|^2 = o(E_k).
            \end{equation}
        \end{itemize}
    \end{lemma}

    Observe that the arguments in proof of Lemma \ref{l:comparison} remains completely unchanged from that of \cite[Lemma 4.5]{DLSk1}, after replacing the application of the relevant preliminary results from \cite{DLS14Lp, DLS16centermfld} with their analogues in \cite{Spolaor_15}. In particular, we emphasize the following:
    \begin{itemize}
        \item[(1)] Given the estimate \cite[(30)]{DLSk1}, which remains valid herein since the properties of the Whitney decomposition remain unchanged, we conclude that \eqref{e:domega-scaling} holds for any $\delta_3\in (0,\delta_2)$.
        \item[(2)] The estimates \cite[Proposition 5.2 (5.2), Proposition 4.4 (i)]{DLS16centermfld} are replaced by \cite[Proposition 4.7 (4.25), Proposition 4.8 (ii)]{Spolaor_15} respectively; namely, despite the constructions of the respective center manifolds for area-minimizing integral currents and semicalibrated currents being different, the relevant comparison and derivative estimates are still satisfied.
    \end{itemize}

    With Lemma \ref{l:reparametrization-lemma} and Lemma \ref{l:comparison} at hand, the proof of Proposition is exactly the same as that in \cite{DLSk1}.

    \section{Improved frequency lower bound}
    This section is dedicated to the proof of the lower bound $\Irm(T,0) \geq 1$ in Theorem \ref{t:sing-degree-main}(i). This maybe be equivalently restated as follows.

    \begin{theorem}\label{t:freq-lb-1}
        Suppose that $T$ satsfies Assumption \ref{a:main-2}. Then $I_u(0) \geq 1$ for any fine blow-up $u$.
    \end{theorem}

    The proof of Theorem \ref{t:freq-lb-1} follows very similar reasoning to \cite[Theorem 5.1]{DLSk1}, now that we have Theorem \ref{t:freqbds}, which generalizes \cite[Theorem 5.2]{DLSk1} to the semicalibrated setting. Nevertheless, we repeat the details here.

    \subsection{Proof of Theorem \ref{t:freq-lb-1}}
    Let $r_k \in (s_{j(k)},t_{j(k)}]$ be a blow-up sequence of radii which generates a fine blow-up $u$. Up to extracting a subsequence, we have three cases:
    \begin{itemize}
        \item[(a)] there exists $J \in \N$ such that $s_J = 0$ and $\{r_k\} \subset (0,t_J]$;
        \item[(b)] $\#\{j(k):k\in \N\} = \infty$ and $\lim_{k\to\infty} \frac{s_{j(k)}}{r_k} = 0$;
        \item[(c)] $\#\{j(k):k\in \N\} = \infty$ and $\lim_{k\to\infty} \frac{s_{j(k)}}{r_k} > 0$.
    \end{itemize}
    
    \proofstep{Case (a):} Let $\Mcal$, $N$ denote respectively the center manifold and normal approximation associated to the interval of flattening $(0,t_J]$, and let us omit dependency on $N_J$ for $\Ibf$ and related quantities.

    In light of the almost-monotonicity \eqref{e: freq-am} of Theorem \ref{t:freqbds}, the limit $I_0 := \lim_{r\downarrow 0} \Ibf(r)$ exists and lies in $[c_0, \infty)$. Furthermore, the strong $W^{1,2}_{\loc} \cap L^2$-convergence of $u_k$ to $u$ as in Section \ref{ss:compactness} implies that $I_u(0) = I_u(r) \equiv I_0$ for each $r\in (0,3/2)$. It therefore remains to check that $I_0 \geq 1$.

    First of all, observe that the stopping criteria for the intervals of flattening (see e.g. \cite[Remark 7.4]{Spolaor_15}) guarantees that
    \begin{equation}\label{e:D-decay}
        \Dbf(r) \leq Cr^{m+2-2\delta_2} \qquad \forall r\in (0,1].
    \end{equation}
    Together with \cite[(7.9)]{Spolaor_15}, we additionally obtain
    \begin{equation}\label{e:L-decay}
        \Lbf(r) \leq Cr^{m+3-2\delta_2} \qquad \forall r\in (0,1].
    \end{equation}
    On the other hand, observe that since $\Ibf(r) \geq \frac{I_0}{2}$ for every $r>0$ sufficiently small, the estimates \cite[(7.5), (7.8)]{Spolaor_15} can be rewritten as
    \[
        \left|\partial_r\log\left(\frac{\Hbf(r)}{r^{m-1}}\right) - \frac{2\Ibf(r)}{r}\right| \leq \frac{Cr^{\gamma_3}\Ibf(r)}{r} \qquad \forall r\in (0,r_0],
    \]
    for some $r_0 = r_0(I_0)>0$ sufficiently small. Here, $\gamma_3$ is as in \cite{Spolaor_15}. In particular, given $\eps >0$, we have
    \[
        \frac{2 I_0 - \eps}{r} \leq \partial_r\log\left(\frac{\Hbf(r)}{r^{m-1}}\right) \leq \frac{2 I_0 + \eps}{r} \qquad \forall r\in(0,r_1],
    \]
    for some $r_1=r_1(\eps) \in (0, r_0]$. This in turn yields
    \[
        \liminf_{r\downarrow 0} \frac{\Hbf(r)}{r^{m-1+ 2I_0 +\eps}} \geq \frac{\Hbf(r_1)}{r_1^{m-1+2I_0 + \eps}} > 0,
    \]
    which then further gives the consequence
    \[
        \liminf_{r\downarrow 0} \frac{\mathbf{\Gamma}(r)}{r^{m-2+ 2I_0 +\eps}} > 0.
    \]
    Recalling that $\mathbf{\Gamma} = \Dbf(r) + \Lbf(r)$ and combining with the decay \eqref{e:D-decay}, \eqref{e:L-decay}, we must therefore have
    \[
        I_0 \geq 2-\delta_2,
    \]
    which in particular implies the desired lower bound of 1, in this case.

    \proofstep{Case (b):} Let $u$ denote the fine blow-up generated by (a subsequence of) $r_k$. Notice that
    \begin{equation}\label{e:freq-conv}
        I_u(\rho) = \lim_{k\to\infty} \Ibf_{N_{j(k)}} \left(\frac{\rho r_k}{t_{j(k)}}\right) \qquad \forall \rho \in (0,1],
    \end{equation}
    in light of the strong $W^{1,2}\cap L^2$-convergence described in Section \ref{ss:compactness}. In particular, since the stopping conditions for the intervals of flattening guarantee that
    \[
        \Ebf(T,\Bbf_{s_{j(k)}}) = \Ebf(T_{0,t_{j(k)}},\Bbf_{s_{j(k)}/t_{j(k)}}) \leq C\eps_3^2 \left(\frac{s_{j(k)}}{t_{j(k)}}\right)^{2-2\delta_2} \longrightarrow 0 \qquad \text{as $k \to \infty$,}
    \]
    we deduce that $\{s_{j(k)}\}$ is a blow-up sequence of radii. Applying Proposition \ref{p:coarse=fine}, we conclude that the fine blow-up $\tilde u$ generated by (a subsequence of) $s_{j(k)}$ satisfies $I_{\tilde u}(0) \geq 1$. Combining this with \eqref{e:freq-conv} (which additionally holds with $u$ replaced by $\tilde u$ and $r_k$ replaced by $s_{j(k)}$), we conclude that
    \[
        \liminf_{k\to\infty} \Ibf_{N_{j(k)}}\left(\frac{s_{j(k)}}{t_{j(k)}}\right) \geq 1.
    \]
    Combining this with the almost-monotonicity \eqref{e: freq-am} of the frequency, we easily conclude that for $\delta>0$ arbitrary, there exists $\bar \rho >0$ such that
    \[
        \liminf_{k\to\infty} \Ibf_{N_{j(k)}}\left(\frac{\rho r_k}{t_{j(k)}}\right) \geq 1-\delta \qquad \forall \rho \in \left(\frac{s_{j(k)}}{t_{j(k)}}, \bar\rho\right),
    \]
    from which the desired conclusion follows immediately. See \cite[Section 5.3]{DLSk1} for details.
    
    \proofstep{Case (c):} In this case, the hypotheses of Proposition \ref{p:coarse=fine} hold, so applying this proposition, we immediately obtain the desired conclusion.

\section{Frequency BV estimate}\label{s:BV}
This section is dedicated to the proof of Theorem \ref{t:BV}. To begin with, we state a sharper formulation of the variational identities \cite[Proposition 7.5]{Spolaor_15}. Let $T$ satisfy Assumption \ref{a:main-2}, and let $\Mcal$ be a center manifold for $T$ with associated normal approximation $N$. Given $x\in \Mcal$, set
    \begin{align*}
	   \Ebf_N(x,r) &= -\frac{1}{r}\int_{\Mcal} \phi'\left(\frac{d(x,y)}{r}\right)\sum_i N_i(y)\cdot DN_i(y)\nabla d(x,y)\, dy , \\
	   \Gbf_N(x,r) &= -\frac{1}{r^2} \int_{\Mcal} \phi'\left(\frac{d(x,y)}{r}\right) \frac{d(x,y)}{|\nabla d(x,y)|^2} \sum_i |DN_i(y) \cdot \nabla d(x,y)|^2 \,d y , \\
	   \mathbf{\Sigma}_N(x,r) &= \int_{\Mcal} \phi\left(\frac{d(x,y)}{r}\right)|N(y)|^2 \,d y .
    \end{align*}

    \begin{proposition}\label{p:firstvar}
        There exist $\gamma_4 = \gamma_4(m,n,Q)>0$ and $C = C(m,n,Q)>0$ such that the following holds. Suppose that $T$ satisfies Assumption \ref{a:main-2}. Let $(s,t]$ be an interval of flattening for $T$ around $0$ with associated center manifold $\Mcal$ and normal approximation $N$ and let $\mbf_0$ be as in \eqref{e:m_0} for this interval. Then $\Dbf_N(0,\cdot)$, $\Hbf_N(0,\cdot)$ are absolutely continuous on $(\tfrac{s}{t},3]$ and for almost every $r\in (\tfrac{s}{t},3]$ we have
        \begin{align}
		&\partial_r \Dbf_N(0,r) = - \int_{\Mcal} \phi'\left(\frac{d(y)}{r}\right) \frac{d(y)}{r^2} |DN(y)|^2 \, dy \label{e:firstvar1} \\
		&\partial_r \Hbf_N(0,r) - \frac{m-1}{r} \Hbf_N (0,r) = O(\boldsymbol{m}_0) \Hbf_N (0,r) + 2 \Ebf_N(0,r), \label{e:firstvar2}\\
		&|\mathbf{\Gamma}_N(0,r) - \Ebf_N(0,r)| \leq \sum_{j=1}^5 |\Err_j^o| \leq C\boldsymbol{m}_0^{\gamma_4}\Dbf_N(0,r)^{1+\gamma_4} + C\boldsymbol{m}_0\mathbf{\Sigma}_N(0,r), \label{e:firstvar3}\\
        & \qquad\qquad\qquad\qquad\qquad\leq C\boldsymbol{m}_0^{\gamma_4}\Dbf_N(0,r)^{1+\gamma_4} + C\mbf_0 r^2 \Dbf_N(0,r), \notag\\
		&\left|\partial_r \Dbf_N(0,r)  - (m-2) r^{-1} \Dbf_N(0,r)- 2\Gbf_N(0,r)\right| \leq 2 \sum_{j=1}^5 |\Err_j^i|  + C \boldsymbol{m}_0\Dbf_N(0,r) \label{e:firstvar4}\\
		&\qquad \leq Cr^{-1}\boldsymbol{m}_0^{\gamma_4}\Dbf_N(0,r)^{1+\gamma_4} + C\boldsymbol{m}_0^{\gamma_4}\Dbf_N(0,r)^{\gamma_4}\partial_r \Dbf_N(0,r) +C\boldsymbol{m}_0 \Dbf_N(0,r),\notag \\
        &|\Lbf_N(0,r)| \leq C\mbf_0^{1/4} r \Dbf_N(0,r), \qquad |\partial_r\Lbf_N(0,r)| \leq C\mbf_0^{\gamma_4} (r^{-1}\partial_r\Dbf_N(0,r) \Hbf_N(0,r))^{\gamma_4}\label{e:firstvar5}
	   \end{align}
    where $\Err_j^o$, $\Err_j^i$ are the variational errors as in \cite[Section 3.3]{DLS16blowup}.
    \end{proposition}
    The estimates in Proposition \ref{p:firstvar} follow by the same reasoning as their weaker counterparts in \cite[Proposition 7.5]{Spolaor_15}, together with the following observations:
    \begin{itemize}
        \item the error estimates in \cite[Section 4]{DLS16blowup} (more precisely, see \cite[Lemma 4.5]{DLS16blowup}) can be optimized so as to gain a factor of $\mbf_0^{\gamma_4}$ on the right-hand side, in light of the following estimates on the geodesic distance $d$ on each center manifold $\Mcal$:
        \begin{enumerate}[(i)]
	       \item $d(x,y) = |x-y| + O\left(\boldsymbol{m}_0^{1/2} |x-y|^2\right)$, 
	       \item $|\nabla d(x,y)| = 1 + O\left(\boldsymbol{m}_0^{1/2}d (x,y)\right)$,
	       \item $\nabla^2 (d^2) = g + O(\boldsymbol{m}_0 d)$, where $g$ is the metric induced on $\mathcal{M}$ by the Euclidean ambient metric.
        \end{enumerate}
        \item the estimates \cite[(7.25), (7.26)]{Spolaor_15} in fact immediately yield the estimates in \eqref{e:firstvar5}, in light of \eqref{e:domega-balls}.
    \end{itemize}
    These estimates are a simple consequence of the $C^{3,\kappa}$-estimates for each center manifold; see e.g. \cite{DLDPHM}. We therefore omit the details here.

    The estimates of Proposition \ref{p:firstvar} in turn gives rise to the following almost-monotonicity estimate for the frequency relative to a given center manifold.

    \begin{corollary}\label{c:freq-mon}
        There exists $C=C(m,n,Q)>0$ such that the following holds. Suppose that $T$ satisfies Assumption \ref{a:main-2}. Let $(s,t]$, $x$, $\Mcal$, $N$ and $\gamma_4$ be as in Proposition \ref{p:firstvar}. Then $\Ibf_N(0,\cdot)$ is absolutely continuous on $(\tfrac{s}{t},3]$ and for almost-every $r\in (\tfrac{s}{t},3]$ we have
        \[
            \partial_r \log (1 +\Ibf_N(0,r)) \geq - C\mbf_0^{\gamma_4} \left(1 + \frac{\Dbf_N(0,r)^{\gamma_4}}{r} + \Dbf_N(0,r)^{\gamma_4-1}\partial_r\Dbf_N(0,r)\right)
        \]
    \end{corollary}

    Now observe that \cite[Lemma 6.6]{DLSk1} can in fact be stated for a general manifold that is the graph of a sufficiently regular function as follows, with the proof remaining completely unchanged.
    
    \begin{lemma}[Curvilinear excess expansion]
        There exists a dimensional constant $C = C(m,n,Q) > 0$ such that the following holds. Let $\Mcal=\graph(\boldsymbol{\varphi}_r)$ be a $C^3$ $m$-dimensional $C^3$ submanifold of $\R^{m+n}$, where $\boldsymbol{\varphi}_r \in C^3(\overline{B_r(0,\pi)};\pi^\perp)$. Let $f: B_r(0,\pi)\to\Acal_Q(\pi^\perp)$ be a Lipschitz map. Then we have
	\begin{align*}
		\left| \int_{\Cbf_r(0,\pi)} \right.&\left.|\vec\Gbf_{f}(z) - \vec{\Mcal}\circ\mathbf{p}(z)|^2\phi\left(\frac{|\mathbf{p}_{\pi} (z)|}{r}\right) \, d\|\Gbf_{f}\|(z) - \int_{B_r(0,\pi)} \Gcal\big(Df, Q\llbracket D\boldsymbol{\varphi}_r\rrbracket\big)^2 \, dy \right| \\
		&\leq C\int_{B_r(0,\pi)} (|Df|^4 + |D\boldsymbol{\varphi}_r|^4) \phi\left(\frac{|y|}{r}\right)\, dy  \\
		&\qquad +C \int_{\Cbf_r(0,\pi)} \Big| \vec{\Mcal}(\mathbf{p}(z)) - \vec{\Mcal}\big(\boldsymbol{\varphi}_r(\mathbf{p}_{\pi}(z))\big) \Big| \, d\|\Gbf_{f}\|(z).
	\end{align*}
    \end{lemma}

    Thus, we immediately deduce the analogue of \cite[Corollary 6.7]{DLSk1} when $T$ is semicalibrated.

    \begin{corollary}\label{c:curv-excess}
        There exists a dimensional constant $C = C(m,n,Q) > 0$ such that the following holds. Let $T$ satisfy Assumption \ref{a:main-2}. Let $(s,t]$ be an interval of flattening for $T$ around $0$ with corresponding center manifold $\Mcal$ and normal approximation $N$, let $\boldsymbol{m}_0$ be as in~\eqref{e:m_0} for $(s,t]$. Let $\boldsymbol{\varphi}$ be the parameterizing map for $\Mcal$ over $B_3(\pi)$ and let $f: B_1(\pi) \to \Acal_Q(\pi^\perp)$ be a $\pi$-approximation for $T_{0,t}$ in $\Cbf_4(0,\pi)$ according to \cite{DLS14Lp}. For $\bar{r}=\frac{s}{t}$, let $f_L: B_{8r_L}(p_L,\pi_L) \to \Acal_Q(\pi_L^\perp)$ be a $\pi_L$-approximation for $T_{0,t}$ corresponding to a Whitney cube $L$ as in~\cite[Section~2.1~(Stop)]{DLS16blowup}. Let $\pi_{\bar{r}}$ be such that $\Ebf(T_{0,t},\Bbf_{6\sqrt{m}\bar{r}}) = \Ebf(T_{0,t},\Bbf_{6\sqrt{m}\bar{r}}, \pi_{\bar{r}})$ and let $B^L := B_{8r_L}(p_L,\pi_L)$. Let $f_{\bar{r}}:B_{\bar{r}}(0,\pi_{\bar{r}}) \to \Acal_Q(\pi_{\bar{r}}^\perp)$ be the map reparameterizing ${\rm gr}\, (f_L)$ as a graph over $\pi_{\bar{r}}$ and let $\boldsymbol{\varphi}_{\bar{r}}, \boldsymbol{\varphi}_L$ be the maps reparameterizing $\graph (\boldsymbol{\varphi})$ as graph over $\pi_{\bar{r}}, \pi_L$ respectively. Then we have
	\begin{align}
		\left| \int_{B_1(0,\pi)} \right.&\left.\Gcal(Df, Q\llbracket D\boldsymbol{\varphi} \rrbracket)^2\phi\left(|y|\right)\,d y - \int_{\Bbf_1 \cap \Mcal} |DN|^2 \phi\left(d(y)\right)\,d y \right| \label{eq:graphDir1} \\
		&\leq C\int_{B_1(0,\pi)} (|Df|^4 + |D\boldsymbol{\varphi}|^4)d y  + C\boldsymbol{m}_0^{1+\gamma_2} + C \int_{\Bbf_1 \cap \Mcal} (|\Abf_{\Mcal}|^2|N|^2 + |DN|^4)\notag\\
		&\qquad +C \int_{\Cbf_1(0,\pi)} \Big| \vec{\Mcal}(\mathbf{p}(z)) - \vec{\Mcal}\big(\boldsymbol{\varphi}(\mathbf{p}_{\pi}(z))\big) \Big| \, d\|\Gbf_{f}\|(z),\notag
	\end{align}
	and
	\begin{align}
		\left| \int_{B_{\bar{r}}(0,\pi_{\bar{r}})} \right.&\left.\Gcal(Df_{\bar{r}}, Q\llbracket D\boldsymbol{\varphi}_{\bar{r}} \rrbracket)^2\phi\left(\frac{|y|}{\bar{r}}\right)\, d y - \int_{\Bbf_{\bar{r}}\cap\Mcal} |DN|^2 \phi\left(\frac{d(y)}{\bar{r}}\right)d y \right| \label{eq:graphDir2}\\
		&\leq C\int_{B_{\bar{r}}(0,\pi_{\bar{r}})} (|Df_{\bar{r}}|^4 + |D\boldsymbol{\varphi}_{\bar{r}}|^4)\, d y + C\int_{B^L} (|Df_{L}|^4 + |D\boldsymbol{\varphi}_L|^4)\, d y   \notag\\
		&\qquad+ C\boldsymbol{m}_0^{1+\gamma_2}\bar{r}^{m+2+\gamma_2} + C \int_{\Bcal^L} (|\Abf_{\Mcal}|^2|N|^2 + |DN|^4) \notag\\
		&\qquad +C \int_{\Cbf_{\bar{r}}(0,\pi_{\bar{r}})} \Big| \vec{\Mcal}(\mathbf{p}(z)) - \vec{\Mcal}\big(\boldsymbol{\varphi}(\mathbf{p}_{\pi_{\bar{r}}}(z))\big) \Big|\, d\|\Gbf_{f_{\bar{r}}}\|(z),\notag
	\end{align}
    where $\Abf_\Mcal$ denotes the second fundamental form of $\Mcal$ and $\gamma_2$ is as in \cite{Spolaor_15}.
    \end{corollary}

    In light of the estimates of Corollary \ref{c:curv-excess}, we further wish to control the difference between an orthogonal projection to a center manifold and the image on $\Mcal$ of an orthogonal projection to a plane over which the center manifold is parameterized.
    
    \begin{lemma}\label{l:cm-proj}
        There exists a constant $C=C(m,n,Q) > 0$ such that the following holds. Suppose that $T$, $\Mcal$, $\boldsymbol{m}_0$, $\bar{r}$, $f$, $f_{\bar{r}}$, $\pi$, $\pi_{\bar{r}}$, $\boldsymbol{\varphi}_{\bar{r}}$, $\gamma_2$ are as in Corollary~\ref{c:curv-excess}. Then we have
	\begin{align}
		&\int_{\Cbf_{\bar{r}}(0,\pi_{\bar{r}})} \Big| \vec{\Mcal}(\mathbf{p}(z)) - \vec{\Mcal}\big(\boldsymbol{\varphi}_{\bar{r}}(\mathbf{p}_{\pi_{\bar{r}}}(z))\big) \Big| \, d\|\Gbf_{{f}_{\bar r}}\|(z) \leq C\bar{r}^{m+1}\boldsymbol{m}_0^{1+\gamma_2}, \label{e:cmLip1}\\
		&\int_{\Cbf_{1}(0,\pi)} \Big| \vec{\Mcal}(\mathbf{p}(z)) - \vec{\Mcal}\big(\boldsymbol{\varphi}(\mathbf{p}_\pi(z))\big) \Big| \, d\|\Gbf_{f}\|(z) \leq C\boldsymbol{m}_0^{1+\gamma_2}.\label{e:cmLip2}
	\end{align}
    \end{lemma}

    The proof of Lemma \ref{l:cm-proj} follows exactly as that of \cite[Lemma 6.9]{DLSk1}, exploiting the estimates of \cite[Theorem 1.4]{DLSS1} and \cite[Corollary 3.5]{Spolaor_15} in place of their respective analogues in \cite{DLS14Lp,DLS16centermfld}.

    We further have the following comparison estimate between neighboring center manifolds.

    \begin{lemma}\label{l:compare-cm}
        There exists a constant $C=C(m,n,Q) > 0$ such that the following holds. Suppose that $T$ satisfies Assumption \ref{a:main-2}. Let $\Mcal_{k-1}, \Mcal_{k}$ be successive center manifolds for $T$ associated to neighboring intervals of flattening $(t_k,t_{k-1}]$ and $(t_{k+1}, t_k]$ around $0$. Let $\bphi_{k-1},\bphi_k$ denote their respective parameterizing maps and let $N_{k-1},N_k$ denote their normal approximations. Assume that $\Ebf(T_{0,t_k},\Bbf_{6\sqrt{m}},\pi_k) = \Ebf(T_{0,t_k},\Bbf_{6\sqrt{m}})$ for some plane $\pi_k$ and let $\tilde{\boldsymbol{\varphi}}_{k-1}$ be the map reparametrizing $\graph(\boldsymbol{\varphi}_{k-1})$ as a graph over $\pi_k$.  
	   Letting $\tilde{\boldsymbol{\varphi}}_k := \tilde{\boldsymbol{\varphi}}_{k-1}\left(\frac{t_k}{t_{k-1}}\cdot \right)$, we have
	   \begin{equation}\label{e:comparison-cm-1}
	       \int_{B_1} |D\boldsymbol{\varphi}_k - D\tilde{\boldsymbol{\varphi}}_{k}|^2 \leq C \boldsymbol{m}_{0,k}^{3/2}.
	   \end{equation}
        and
        \begin{equation}\label{e:comparison-cm-2}
	       \int_{B_2} |\boldsymbol{\varphi}_k - \tilde{\boldsymbol{\varphi}}_{k}|^2 \leq C \boldsymbol{m}_{0,k}\, .
	    \end{equation}
        \end{lemma}
        \begin{proof}[Proof of Lemma \ref{l:compare-cm}]
            The proof follows the same reasoning as that of \cite[Lemma 6.8]{DLSk1}. Nevertheless, let us provide an outline here and highlight the differences. First of all, observe that by a rotation of coordinates, we may without loss of generality assume that $\pi_{k-1}=\pi_k\equiv \pi_0$ and $\tilde\bphi_{k-1} = \bphi_{k-1}$.
            
            Let $\eta \in C_c^\infty(B_2(\pi_0);[0,1])$ be a cutoff function satisfying $\eta \equiv 1$ on $B_1$. Via an integration by parts, we have
            \begin{align*}
		      \int_{B_1} |D\boldsymbol{\varphi}_k - D\tilde{\boldsymbol{\varphi}}_{k}|^2  &\leq \int_{B_{2}} |D\boldsymbol{\varphi}_k - D\tilde{\boldsymbol{\varphi}}_{k}|^2 \eta \\
		      &= -\int_{B_{2}} (\boldsymbol{\varphi}_k - \tilde{\boldsymbol{\varphi}}_{k})\eta \Delta (\boldsymbol{\varphi}_k - \tilde{\boldsymbol{\varphi}}_{k})  - \int_{B_{2}\setminus B_1} D\eta \cdot (\boldsymbol{\varphi}_k - \tilde{\boldsymbol{\varphi}}_{k}) D(\boldsymbol{\varphi}_k - \tilde{\boldsymbol{\varphi}}_{k}) \\
		      &\leq C\left(\boldsymbol{m}_{0,k}^{1/2} + \frac{t_k}{t_{k-1}}\boldsymbol{m}_{0,k-1}^{1/2}\right) \int_{B_{2}} |\boldsymbol{\varphi}_k - \tilde{\boldsymbol{\varphi}}_{k}|.
	       \end{align*}
        
        Now recall that the construction procedure for the intervals of flattening guarantees that
        \begin{equation}\label{e:excess-decay-quadratic}
            \left(\frac{t_k}{t_{k-1}}\right)^{2-2\delta_2}\boldsymbol{m}_{0,k-1} \leq C \mbf_{0,k}.
        \end{equation}
        Thus, \eqref{e:comparison-cm-1} follows from \eqref{e:comparison-cm-2}, and so it suffices to demonstrate the latter. Given a Lipschitz approximation $f_k:B_3(\pi_0)\to\Acal_Q(\pi_k^\perp)$ for $T_{0,t_k}\mres \Cbf_{4}(0,\pi_0)$ as in \cite[Theorem 1.4]{DLSS1}, which can indeed be considered for $\eps_3$ sufficiently small since $\Ebf(T_{0,t_k},\Cbf_{4}(0,\pi_0)) \leq C\mbf_{0,k}$, observe that it suffices to show
        \begin{align}
            \int_{B_2} |\bphi_k - \boldsymbol\eta\circ f_k| &\leq C\mbf_{0,k}, \label{e:comparison-avg1}\\
            \int_{B_2} |\tilde \bphi_k - \boldsymbol\eta\circ f_k| &\leq C\mbf_{0,k}.\label{e:comparison-avg2}
        \end{align}
        In fact, notice that \eqref{e:comparison-avg2} will follow from exactly the same argument as \eqref{e:comparison-avg1}, when combined with \eqref{e:excess-decay-quadratic}. Indeed, this is due to the fact that for $\tilde{f}_k:=f_{k-1}(\tfrac{t_k\cdot}{t_{k-1}})$ and $f_{k-1}$ as above but for $T_{0,t_{k-1}}\mres \Cbf_4(0,\pi_0)$, we have $\Gbf_{f_k} \equiv \Gbf_{f_k} \equiv T_{0,t_k}$ on $K\times \pi_0^\perp$ for a closed set $K\subset B_2$ with 
        \begin{equation}\label{e:graphicality-set}
            |B_2\setminus K| \leq C\mbf_{0,k}^{1+\beta_0},
        \end{equation}
        where $\beta_0>0$ is as in \cite[Theorem 1.4]{DLSS1}.

        Now let us demonstrate the validity of \eqref{e:comparison-avg1}. By the construction of $\Mcal_k$, $B_2$ is covered by a disjoint union of the contact set $\mathbf\Gamma \subset\pi_0$ and Whitney cubes $\Wscr':=\{L\in \Wscr: L\cap B_2\neq\emptyset\}$. 
        % Note that for each $L\in\Wscr'$, we have $\ell(L)\leq 2^{-N_0}$ for $N_0$ as in \cite[Assumption 2.5]{Spolaor_15} and so $L\subset B_3$.
        Therefore we have
        \begin{equation}
            \int_{B_2} |\bphi_k -\boldsymbol\eta\circ f_k|\leq \underbrace{\int_{\mathbf{\Gamma}\cap B_2} |\bphi_k -\boldsymbol\eta\circ f_k|}_{(A)} + \underbrace{\sum_{L\in\Wscr'} \int_{L\cap B_3} |\bphi_k -\boldsymbol\eta\circ f_k|}_{(B)}.
        \end{equation}
        Firstly, we have
        \[
            |(A)| \leq C \mbf_{0,k}^{1+\beta_0},
        \]
        due to \eqref{e:graphicality-set}, together with the fact that $\mathbf\Gamma\subset K$ (see \cite[Definition 3.3]{Spolaor_15}) and the estimates \cite[Theorem 1.4(1.6), Corollary 3.5]{Spolaor_15}. 
        
        Meanwhile, for $(B)$, we argue as follows. For each $L\in\Wscr'$, let $\pi_L$ denote the optimal plane associated to $L$ as in \cite[Lemma 2.9]{Spolaor_15}, with corresponding $\pi_L$-approximation $f_L$, associated tilted $L$-interpolating function $h_L$ as in \cite[Definition 2.10]{Spolaor_15} and (straight) $L$-interpolating function $g_L$. Note that $h_L$ (and hence $g_L$) are constructed via a different smoothing procedure here, in comparison to that in \cite{DLS16centermfld} where $T$ is area-minimizing. More precisely, here $h_L$ is constructed by solving a suitable PDE with boundary data $\boldsymbol\eta \circ f_L$, while in \cite{DLS16centermfld} it is constructed via convolution of $\boldsymbol\eta \circ f_L$. Nevertheless, by \cite[Proposition 4.7, Proposition 4.8(vi)]{Spolaor_15}, we still have the key estimates
        \[
            \int_L |\bphi_k - g_L| \leq C\mbf_{0,k} \ell(L)^{m+3+\beta_2/3},
        \]
        for $\beta_2>0$ as in \cite{Spolaor_15}, and
        \[
            \int_{B_{2\sqrt{m}\ell(L)(p_L,\pi_L)}} |h_L - \boldsymbol\eta\circ f_L| \leq C \mbf_{0,k}\ell(L)^{m+3+\beta_2},
        \]
        where $p_L$ is the center of $L$ and $\ell(L)$ is the side-length of $L$. Combining these with a reparameterization from $\pi_0$ to $\pi_L$ and the tilting estimate \cite[Proposition 4.1(iii)]{Spolaor_15}, then summing over $L\in\Wscr'$, the conclusion follows; see \cite{DLSk1} for the details.
    \end{proof}

\subsection{Proof of Theorem \ref{t:BV}}
    With all of the preliminary estimates of this section at hand, we are now in a position to conclude the BV estimate of Theorem \ref{t:BV}. Let $\Dbf_k:=\Dbf_{N_k}$, $\Hbf_k:=\Hbf_{N_k}$, $\Lbf_k:=\Lbf_{N_k}$, $\mathbf\Gamma_k$ and let $\bar\Dbf_k := r^{-(m-2)} \Dbf_k$, $\bar\Hbf_k := r^{-(m-1)} \Hbf_k$, $\bar\Lbf_k := r^{-(m-2)} \Lbf_k$ , $\bar{\mathbf\Gamma}_k:= \bar\Dbf_k + \bar\Lbf_k$ denote their respective scale-invariant quantities. Let us first consider the jumps of $\Ibf$ at the radii $t_k$. We have
    \begin{align*}
        |\Ibf(t_k^+) - \Ibf(t_k^-)| &=\left|\frac{\bar\Dbf_{k-1}(\tfrac{t_k}{t_{k-1}}) - \bar\Dbf_k(1)}{\bar\Hbf_k(1)}\right| + \left|\frac{\bar\Lbf_{k-1}(\tfrac{t_k}{t_{k-1}}) - \bar\Lbf_k(1)}{\bar\Hbf_k(1)}\right| \\
        &\qquad+ |\bar{\mathbf\Gamma}_{k-1}(\tfrac{t_k}{t_{k-1}})| \left|\frac{1}{\bar\Hbf_{k-1}(\tfrac{t_k}{t_{k-1}})} - \frac{1}{\bar\Hbf_{k}(1)}\right|,
    \end{align*}
    where $\Ibf(t_k^+):= \Ibf_{k-1}(\tfrac{t_k}{t_{k-1}})$ and $\Ibf(t_k^-) := \Ibf_k(1)$.
    Now in light of \eqref{e:firstvar5} and \eqref{e:excess-decay-quadratic}, we have
    \[
        \left|\frac{\bar\Lbf_{k-1}(\tfrac{t_k}{t_{k-1}}) - \bar\Lbf_k(1)}{\bar\Hbf_k(1)}\right| \leq C\mbf_{0,k}^{1/4}\left|\frac{\bar\Dbf_{k-1}(\tfrac{t_k}{t_{k-1}}) - \bar\Dbf_k(1)}{\bar\Hbf_k(1)}\right|,
    \]
    and
    \[
        |\bar{\mathbf\Gamma}_{k-1}(\tfrac{t_k}{t_{k-1}})| \left|\frac{1}{\bar\Hbf_{k-1}(\tfrac{t_k}{t_{k-1}})} - \frac{1}{\bar\Hbf_{k}(1)}\right| \leq C\mbf_{0,k}^{1/4} \bar\Dbf_{k-1}(\tfrac{t_k}{t_{k-1}}) \left|\frac{1}{\bar\Hbf_{k-1}(\tfrac{t_k}{t_{k-1}})} - \frac{1}{\bar\Hbf_{k}(1)}\right|.
    \]
    Thus, by the exact same reasoning as that for the estimates \cite[Proof of Proposition 6.2, (60), (61)]{DLSk1}, we obtain
    \[
        |\Ibf(t_k^+) - \Ibf(t_k^-)| \leq C \mbf_{0,k}^{\gamma_4}\big(1+ \Ibf(t_k^+)\big).
    \]
    When combined with the elementary identity $\log w \leq w-1$ for $w>0$, we obtain
    \[
        |\log(1+ \Ibf(t_k^+)) - \log(1+ \Ibf(t_k^-))| \leq \frac{|\Ibf(t_k^+) - \Ibf(t_k^-)| }{1+ \Ibf(t_k^+)} \leq C\mbf_{0,k}^{\gamma_4}
    \]
    On the other hand, recall from Corollary \ref{c:freq-mon} that $\Ibf\big|_{(s_k,t_k)}$ is absolutely continuous and
    \[
        \partial_r \log(1+\Ibf(r)) \geq -\underbrace{\frac{C}{t_k}\mbf_{0,k}^{\gamma_4}\left(1+ (\tfrac{r}{t_k})^{-1}\Dbf_k(\tfrac{r}{t_k})^{\gamma_4} + \Dbf_k(\tfrac{r}{t_k})^{\gamma_4-1}\partial_r \Dbf(\tfrac{r}{t_k})\right)}_{=: \ \nu_k(r)} \qquad \forall r\in(s_k,t_k).
    \]
    Thus, we may introduce a suitable function $\mathbf\Omega$ as in \cite[Proof of Proposition 6.2]{DLSk1}, whose distributional derivative is the measure
    \[
        C\sum_{k=j_0}^J \mbf_{0,k}^{\gamma_4}\left(\delta_{t_k} + \nu_k(r)\mathbf{1}_{(s_k,t_k)}\Lcal^1\right), 
    \]
    so that in addition $\log(\Ibf + 1) +\mathbf\Omega$ is monotone non-decreasing. Since
    \[
        |\partial_r \mathbf\Omega|((s_J,t_{j_0}]) \leq C\sum_{k=j_0}^J \mbf_{0,k}^{\gamma_4},
    \]
    and $\big|[\partial_r\log(\Ibf+1)]_-\big|((s_J,t_{j_0}]) \leq |\partial_r \mathbf\Omega|((s_J,t_{j_0}])$, the proof is complete.
    \qed

    \section{Proof of Theorem \ref{t:sing-degree-main}}
    Now that we have demonstrated that Theorem \ref{t:BV} holds, we are in a position to complete the proof of the remaining conclusions (ii)-(vii) of Theorem \ref{t:sing-degree-main}. The starting point is the following tilt excess decay result.

    \begin{proposition}\label{p:tilt-decay}
        For any $I_0{> 1}$, there exist constants $C=C(I_0,m,n,Q)>0$, $\alpha=\alpha(I_0,m,n,Q)>0$ such that the following holds. Let $T$ satisfy Assumption \ref{a:main-2} and suppose that $\Irm(T,0)\geq I_0$. Then there exists $r_0=r_0(I_0,m,n,Q,T) > 0$ (depending also on the center point $0$) such that
        \begin{equation}
            \mathbf{E} (T, \mathbf{B}_r) \leq C \left(\frac{r}{r_0}\right)^{\alpha} \max \{\mathbf{E} (T, \mathbf{B}_{r_0}), \bar\varepsilon^2 r_0^{2-2\delta_2}\}\qquad \forall r\in (0,r_0) \, .
        \end{equation}
        Furthermore, if $C$ is permitted to additionally depend on $\alpha$, one may choose $\alpha$ to be any positive number smaller than $\min\{2(\Irm(T,0)-1),2-2\delta_2\}$.
    \end{proposition}

    The proof of this, and all its necessary preliminaries, is the same as that of \cite[Proposition 7.2]{DLSk1}.

    Observe that a combination of the excess decay in Proposition \ref{p:tilt-decay} and the universal frequency BV estimate of Theorem \ref{t:BV} immediately implies the conclusions (iii), (v), (vi) of Theorem \ref{t:sing-degree-main}, together with (ii), (iv) in the case when $\Irm(T,0)>1$. Indeed, observe that Proposition \ref{p:tilt-decay} guarantees that if $\Irm(T,0)>1$, there exists an index $j_0=j_0(r_0)$ large enough such that $t_{k+1}=s_k$ for each $k\geq j_0$. Combining this with Theorem \ref{t:BV} and the observation that $\tfrac{t_{k+1}}{t_k} \leq 2^{-5}$, we obtain a uniform $\BV$-estimate of the form
    \[
        \left|\left[\frac{d\log(\Ibf+1)}{dr}\right]_-\right|\big((0,t_{j_0}]\big) \leq C \sum_{k=j_0}^{\infty} \mbf_{0,k}^{\gamma_4} \leq C \sum_{k=j_0}^\infty 2^{-5\alpha\gamma_4(k-j_0)} \mbf_{0,j_0}^{\gamma_4} \leq C \mbf_{0,j_0}^{\gamma_4}.
    \]
    for $\alpha=\alpha(\Irm(T,0),m,n,Q)>0$ as in Proposition \ref{p:tilt-decay}, where $C=C(\Irm(T,0),m,n,Q)$. In particular $I_0:=\lim_{r\downarrow 0} \Ibf(r)$ exists, and
    \[
        I_u(\rho) = I_0 \qquad \forall \rho\in(0,1],
    \]
    for every fine blow-up $u$. Note that in particular, if $s_{j_0}=0$ for some $j_0\in \N$, Corollary \ref{c:freq-mon} alone provides the desired conclusion that $I_0=\lim_{r\downarrow 0} \Ibf_{N_{j_0}}(r)$.
    
    It remains to verify the conclusions (vi), (vii) of Theorem \ref{t:sing-degree-main}, as well as the conclusions (ii) and (iv) when $\Irm(T,0)=1$. These all follow by the same reasoning as \cite[Section 8, Section 9]{DLSk1}, with the use of the Lipschitz approximation \cite[Theorem 4.1]{DLSS1} in place of \cite[Theorem 2.4]{DLS14Lp} where needed, so we do not include the details here.
    
\part{Rectifiability of points with singularity degree $> 1$}\label{pt:NV}

\section{Subdivision} \label{s: subdivision}
In this part, we prove Theorem \ref{t:main}. We will work under Assumption \ref{a:main-2} throughout. We will be exploiting the results of the preceding part, centered around points $x\in\Ffrak_Q(T)$ (namely, applying the results therein to $T_{x,1}$). 

It will be useful to produce a countable subdivision of the set $\Ffrak_Q(T)$ as follows. First of all, we may write
\[
    \Ffrak_Q(T) \cap \Bbf_1 = \bigcup_{K \in \N} \Sfrak_K,
\]
for 
\[
\Sfrak_K := \{y\in \Ffrak_Q (T) : \Irm (T, y)\geq 1+2^{-K}\}\cap \bar{\mathbf{B}}_1\, .
\]
In light of Proposition \ref{p:tilt-decay}, we will further decompose each piece $\Sfrak_K$ based on the initial scale $r_0$. Let us rewrite the statement of this proposition applied to points in $\Sfrak_K$. Note that we may ensure that $C=1$, up to further decreasing $r_0$ if necessary (dependent on the exponent $\alpha$).

\begin{proposition}\label{p:tilt-decay-2}
    Let $T$ be as in Assumption \ref{a:main-2}, let $K\in \N$ and let $x\in \Sfrak_K$. For $\mu=2^{-K-1}$, there exists $r_0=r_0(x,m,n,Q,K)>0$ such that
\begin{equation}\label{e:excessdecay-quantitative-higherI-2}
\mathbf{E} (T, \mathbf{B}_r(x)) \leq \left(\frac{r}{s}\right)^{2\mu} \max \{\mathbf{E} (T, \mathbf{B}_{s}(x)), \bar\varepsilon^2 s^{2-2\delta_2}\}\, \qquad \forall 0<r< s < r_0\, .
\end{equation}
\end{proposition}
We may thus decompose $\Sfrak_K$ as follows.

\begin{definition}\label{d:bits}
Let $T$ be as in Assumption \ref{a:main-2} and let $\eps_4$ be a small positive constant which will be specified later. For every $K\in \mathbb N\setminus \{0\}$ define $\mu = \mu (K) := 2^{-K-1}$. Given $K, J\in \mathbb N$, let $\Sfrak_{K,J}$ (which implicitly also depends on $\eps_4$) denote the collection of those points $x\in \Sfrak_K$ for which 
\begin{align}
\mathbf{E} (T, \Bbf_r (x)) &\leq \left(\frac{r}{s}\right)^{2\mu} \mathbf{E} (T, \Bbf_s (x))
\qquad \forall 0<r\leq s\leq \frac{6\sqrt{m}}{J}\label{e:eub1}\\
\mathbf{E} (T, \Bbf_{6 \sqrt{m} J^{-1}}) &\leq \varepsilon_4^2\label{e:eub2}.
\end{align}
\end{definition}

\begin{remark}
    Observe that for each $K,J\in\N$, the set $\Sfrak_{K,J}$ is closed, in light of upper semicontinuity of the singularity degree.
\end{remark}

Notice that by rescaling, it suffices to prove the $(m-2)$ rectifiability of $\Sfrak_{K,J}$ with $K\in\N$ fixed and $J=1$. More precisely, the remainder of this part will be dedicated to the proof of the following.

\begin{theorem}\label{t:main-low-one-piece}
There exists $\eps_4(m,n,Q)>0$ such that the following holds. Let $T$ be as in Assumption \ref{a:main-2}. Then $\Sfrak := \Sfrak_{K,1}$ (which, recall, depends on $\eps_4$) is $(m-2)$-rectifiable and has the $(m-2)$-dimensional local Minkowski content bound
\begin{equation}\label{e:Minkowski-low}
    |\Bbf_r(\Sfrak)| \leq Cr^{n+2} \qquad \forall r\in (0,1],
\end{equation}
for some $C=C(m,n, Q, T,K,\gamma_4,\eps_4)>0$.
\end{theorem}
\begin{remark}
    The constant $C$ in Theorem \ref{t:main-low-one-piece} is implicitly also dependent on a uniform bound on Almgren's frequency function $\Ibf_{x,k}$ over all $k\in\N$, defined relative to center manifolds that will be adapted to a given geometric sequence of scales; see Corollary \ref{c:ulb-frequency} below.
\end{remark}

\section{Adapted intervals of flattening, universal frequency, radial variations}

In order to prove Theorem \ref{t:main-low-one-piece}, we follow a strategy which is much analogous to that in \cite{DLSk2}, relying on the celebrated rectifiable Reifenberg techniques of Naber \& Valtorta \cite{NV_Annals}. We begin by decomposing the interval of scales $(0,1]$ around each point $x\in\Sfrak$ into countably many sub-intervals whose endpoints are given by a fixed geometric sequence and construct a center manifold for each of them, hence use it to compute a corresponding frequency function. In \cite{DLSk2}, since we treat separately the points for which there are finitely many intervals of flattening, these sub-intervals are be \emph{comparable in length} to the intervals of flattening. Here, however, we treat the points with finitely many intervals of flattening together with those points that have infinitely many intervals of flattening. We may do this by ``artificially" stopping and restarting the center manifold procedure around each point $x\in\Ffrak_Q(T)$ with $\Irm(T,x)> 2-\delta_2$, and simply setting the new center manifold to be the rescaling of the existing one, if there is no need to change the center manifold at a given endpoint of the chosen geometric sequence of scales. We will then define a corresponding universal frequency function as in Definition \ref{d:univ-freq}, but relative to this fixed geometric sequence of scales, in place of the original intervals of flattening. The frequency variation estimates and quantitative BV estimate in the preceding part, together with the excess decay of Proposition \ref{p:tilt-decay-2}, will be key to providing us with the necessary quantitative bounds on the points in $\Sfrak$.

\subsection{Center manifolds}\label{ss:adaptedintervals}
Let us begin by adapting the intervals of flattening from Section \ref{ss:compactness} around each point in $\Sfrak$, to a given fixed geometric sequence of scales.

Fix a constant $\gamma\in (0,\tfrac{1}{2}]$, whose choice will specified later, depending only on $m$, $n$, and $Q$. 

Consider a point $x\in \Sfrak$ with corresponding intervals of flattening $\{(t_{k+1},t_k]\}_{k\geq 0}$ that have associated center manifolds $\Mcal_{x,k}$ and normal approximations $N_{x,k}$, together with a geometric blow-up sequence of scales $\{\gamma^j\}_{j\geq 0}$. 

For $j=0$, let $\tilde\Mcal_{x,0} =\Mcal_{x,0}$ and $\tilde N_{x,0} = N_{x,0}$. For $j=1$, if $\gamma$ lies in the same interval of flattening to $t_0=1$, let 
\[
    \tilde\Mcal_{x,1} := \iota_{0,\gamma}(\Mcal_{x,0}), \qquad \tilde N_{x,1}(x) := \frac{N_{x,0}(\gamma x)}{\gamma}.
\]
Otherwise, let $\tilde \Mcal_{x,1}$ be the center manifold associated to $T_{x,\gamma}\mres\Bbf_{6\sqrt{m}}$, with corresponding normal approximation $N_{x,1}$. 

For each $j\geq 2$, define $\tilde\Mcal_{x,j}$ inductively as follows. If $\gamma^{j}$ lies in the same flattening to $\gamma^{j-1}$, let
\[
    \tilde\Mcal_{x,j} := \iota_{0,\gamma}(\Mcal_{x,j-1}), \qquad \tilde N_{x,j}(x) := \frac{N_{x,j-1}(\gamma x)}{\gamma}.
\]
Otherwise, let $\tilde \Mcal_{x,j}$ be the center manifold associated to $T_{x,\gamma^j}\mres\Bbf_{6\sqrt{m}}$, with corresponding normal approximation $N_j$.

It follows from Definition \ref{d:bits} that around any $x\in \Sfrak$, we may replace the procedure in Section \ref{ss:compactness} with the intervals $(\gamma^{k+1},\gamma^k]$ in place of $(s_k,t_k]$, and with $\boldsymbol{m}_{0,k}$ therein instead defined by  
\begin{equation}\label{e:rid-of-m0}
\boldsymbol{m}_{x,k} = \mathbf{E} (T_{x, \gamma^k}, \Bbf_{6\sqrt{m}}) =
\mathbf{E} (T, \Bbf_{6\sqrt{m} \gamma^{k}} (x))\, .
\end{equation}
Observe that in particular, if $x\in\Sfrak$ originally has finitely many intervals of flattening with $(0,t_{j_0}]$ being the final interval, it will nevertheless have infinitely many adapted intervals of flattening, but for all $k$ sufficiently large, $\tilde\Mcal_{x,k}$ and $\tilde N_{x,k}$ are arising as rescalings of $\Mcal_{x,j_0}$ and $N_{x,j_0}$ respectively.

Abusing notation, let us henceforth simply write $\Mcal_{x,k}$ for the center manifold $\tilde\Mcal_{x,k}$, with its corresponding normal approximation $N_{x,k}$. 

We will henceforth denote by $d$ the geodesic distance on the center manifold $\mathcal{M}_{x,k}$, which is in fact dependent on $x$ and $k$. However, since this dependence is not important and it will always be clear from context which center manifold we are taking the geodesic distance on, we will omit it. Let $\pi_{x,k}$ denote the plane used to construct the graphical parametrization $\boldsymbol{\varphi}_{x,k}$ of the center manifold $\Mcal_{x,k}$ and let $\mathscr{W}^{x,k}$ denote the collection of Whitney cubes associated to $\Mcal_{x,k}$ as in \cite[Section 2.2]{Spolaor_15}. Note that the center manifold $\mathcal{M}_{x,k}$ does not necessarily contain the origin $0 = \iota_{x, \gamma^{-k}} (x)$. However we use the point $p_{x,k}:= (0, \boldsymbol{\varphi}_{x,k} (0)) \in \pi_{x,k}\times \pi_{x,k}^\perp$ as a proxy for it.

Fix $\eta\in (0,\tfrac{1}{2}]$, to be determined later. We observe the following simple consequence of our adapted intervals of flattening and associated center manifolds.

\begin{proposition}\label{p:good-cubes}
Let $\gamma, \eta>0$ be two fixed constants and let $c_s = \frac{1}{64\sqrt{m}}$. Upon choosing $N_0$ in \cite[Section 2.2]{Spolaor_15} sufficiently large and adjusting accordingly the constants $C_e$, $C_h$ and $\varepsilon_2$ in \cite{Spolaor_15} accordingly, we can ensure that for every $w\in \mathcal{M}_{x,k}$ and every $r\in [\eta \gamma, 3]$, any $L\in \mathscr{W}^{x,k}$ with $L\cap B_r (\mathbf{p}_{\pi_{x,k}} (w), \pi_{x,k})\neq \emptyset$ satisfies $\ell (L) \leq c_s r$. Moreover, we have the following dichotomy. Either
\begin{itemize}
    \item[(a)] there is a positive constant $\bar c_s= \bar c_s(K,m,n,Q,\eta) \in (0, c_s]$ such that $B_\gamma (0, \pi_{x,k})$ intersects a cube $L\in \mathscr{W}^{x,k}$ with $\ell (L) \geq \bar c_s \gamma$, which violates the excess condition (EX) of \cite{Spolaor_15};
    \item[(b)] $\Mcal_{x,k+1}= \iota_{0,\gamma}(\Mcal_{x,k})$.
\end{itemize}
\end{proposition}

\begin{proof}
   The proof follows immediately from the contruction. 
\end{proof}

We have the following $\BV$ estimate on the universal frequency function adapted to $\{\gamma^k\}_k$ (cf. Theorem \ref{t:BV}).

\begin{proposition}\label{p:bv-adapted}
	There exists $\bar\eps(m,n,Q)\in]0,\eps]$ such that for any $\eps_4\in (0,\bar\eps]$, there exists $C=C(m,n,Q,\gamma_4, K,\gamma)$ such that the following holds for every $x\in \Sfrak$:
	\begin{equation}\label{eq:bv}
		\left|\left[\frac{d\log (1 + \Ibf(x,\cdot))}{dr}\right]_- \right|\big([0, 1]\big) \leq C\sum_k  \boldsymbol{m}_{x,k}^{\gamma_4} \leq C\mbf_{x,1}^{\gamma_4}
	\end{equation}
\end{proposition}

Observe that in light of Proposition \ref{p:good-cubes}, the estimate in Proposition \ref{p:bv-adapted} follows by the same argument as that in Theorem \ref{t:BV}. Indeed, for any $k$ such that $\gamma^{k}$ lies in the same original interval of flattening as $\gamma^{k-1}$, one may estimate the jump
\[
|\log (1 + \Ibf(x,\gamma^k))^+ - \log (1+\Ibf (x, \gamma^{k})^-|
\]
due to Proposition \ref{p:good-cubes}(a) (see \cite[Remark 6.3]{DLSk1}). Meanwhile, if instead $\Mcal_{x,k+1}= \iota_{0,\gamma}(\Mcal_{x,k})$ holds, the adapted universal frequency function is absolutely continuous on $(\gamma^{k+1},\gamma^{k-1}]$, and we instead simply use the variation estimate of Corollary \ref{c:freq-mon}.

\subsection{Universal frequency function} 
For each center manifold $\mathcal{M}_{x,k}$, we define the frequency function for the associated normal approximation, as in Part \ref{pt:sing-deg}. We define analogous quantities to those therein, namely
\begin{align*}
	\Dbf_{x,k}(w,r) &:= \int_{\Mcal_{x,k}} |D N_{x,k}(z)|^2 \phi\left(\frac{d(w,z)}{r}\right)\, dz; \\
	\Hbf_{x,k}(w,r) &:= - \int_{\Mcal_{x,k}} \frac{|\nabla d(w,z)|^2}{d(w,z)} |N_{x,k} (z)|^2\phi'\left(\frac{d(w,z)}{r}\right) \, dz; \\
    \Lbf_{x,k}(w,r) &:= \sum_{i=1}^Q \sum_{l=1}^m (-1)^{l+1} \int_{\Mcal_{x,k}}\langle D_{\xi_l} (N_{x,k})_i(z) \wedge \hat\xi_l(z) \wedge (N_{x,k})_i(z), d\omega(z) \rangle \phi\left(\frac{d(w,z)}{r}\right) d\Hcal^m(z);\\
    \mathbf{\Gamma}_{x,k}(w,r)&:= \Dbf_{x,k}(w,r)+ \Lbf_{x,k}(w,r); \\
	\Ibf_{x,k}(w,r) &:= \frac{r \mathbf{\Gamma}_{x,k}(w,r)}{\Hbf_{x,k}(w,r)}\, .
\end{align*}
We further define the quantities
\begin{align*}
	\Ebf_{x,k}(w,r) &:= -\frac{1}{r} \int_{\Mcal_{x,k}} \phi'\left(\frac{d(w,z)}{r}\right)\sum_i (N_{x,k})_i(z)\cdot D(N_{x,k})_i(z)\nabla d(w,z) \ dz; \\
	\Gbf_{x,k}(w,r) &:= -\frac{1}{r^2} \int_{\Mcal_{x,k}} \phi'\left(\frac{d(w,z)}{r}\right) \frac{d(w,z)}{|\nabla d(w,z)|^2} \sum_i |D( N_{x,k})_i(z) \cdot \nabla d(w,z)|^2\, dz; \\
	\mathbf{\Sigma}_{x,k}(w,r) &:= \int_{\Mcal_{x,k}} \phi\left(\frac{d(w,z)}{r}\right)|N_{x,k}(z)|^2\, dz.
\end{align*}

We are now in a position to introduce the universal frequency function adapted to the geometric sequence $\{\gamma^k\}_k$, analogously to that defined in the preceding part.

\begin{definition}[Universal frequency function adapted to $\{\gamma^j\}_j$]\label{def:univfreq}
For $r \in (\gamma^{k+1}, \gamma^k]$ and $x\in \Sfrak$, define
	\begin{align*}
		\Ibf(x, r) &:= \Ibf_{x,k}\left(p_{x,k},\tfrac{r}{\gamma^k}\right), \\
		\Dbf(x,r) &:= \Dbf_{x,k}\left(p_{x,k}, \tfrac{r}{\gamma^{k}}\right), \\
        \Lbf(x,r) &:= \Lbf_{x,k}\left(p_{x,k}, \tfrac{r}{\gamma^{k}}\right), \\
		\Hbf(x,r) &:= \Hbf_{x,k}\left(p_{x,k}, \tfrac{r}{\gamma^{k}}\right).
	\end{align*}
\end{definition}

\subsection{Radial frequency variations}
As an immediate consequence of the total variation estimate and the fact that $\Sfrak$ is closed, we infer the existence of an uniform upper bound for the frequency $\mathbf{I} (x,r)$ over all $x\in \Sfrak$. We also infer the existence of the limit $\mathbf{I} (x,0) = \lim_{r\downarrow 0} \mathbf{I} (x,r)$. We can then argue as in Part \ref{pt:sing-deg} to show that $\mathbf{I} (x,0) = \Irm (x,0) \geq 1+ 2^{-K}$. In turn, upon choosing $\tilde{\varepsilon}$ sufficiently small we infer the following.

\begin{corollary}\label{c:ulb-frequency}
For $\bar\eps$ as in Proposition \ref{p:bv-adapted} and any $\eps_4\in ]0,\bar\eps]$, there exists $C=C(m,n,Q,\gamma_4,K,\gamma,\bar\eps)$ such that the following holds:
\[
1+ 2^{-K-1} \leq \mathbf{I} (x,r) \leq C \qquad \forall x\in \Sfrak, \forall r \in ]0, 1]\, .
\]
\end{corollary}

By a simple contradiction and compactness argument, we obtain the same consequence as that in Corollary \ref{c:ulb-frequency} for points sufficiently close to $\Sfrak$ at the appropriate scales.

\begin{corollary}\label{c:ulb-frequency-2}
There exists $\eps^*\in(0,\bar\eps]$ such that for any $\eps_4\in (0,\eps^*]$ and any $x\in\Sfrak$, there exists $C_0=C_0(\gamma,\eta,m,n,Q,\eps^*,K)>0$, such that the following holds for every $w\in \mathcal{M}_{x,k}$ and every $r \in (\eta \gamma, 4]$:
\[
C_0^{-1} \leq \mathbf{I}_{x,k} (w,r) \leq C_0\, .
\]
\end{corollary}

Let us now record the following simplified variational estimates, which may be easily deduced from those in Corollary \ref{c:freq-mon}, Proposition \ref{p:firstvar}, but for the intervals of flattening adapted to $\{\gamma^k\}_k$.

\begin{lemma}\label{l:simplify-low}
Let $\bar\eps$ be as in Proposition \ref{p:bv-adapted}. Suppose that $T$, $\eps_4$, $x$, $\Mcal_{x,k}$ and $ N_{x,k}$ are as in Corollary \ref{c:ulb-frequency-2}. Then there exist constants $C$ dependent on $K$, $\gamma$, $\eta$ and $\bar\eps$ but not on $x, k$, such that the following estimates hold for every $w \in \Mcal_{x,k}\cap\Bbf_1$ and any $\rho, r \in (\eta \gamma, 4]$.
\begin{align}
&C^{-1} \leq \Ibf_{x,k} (w,r) \leq C \label{eq:simplify1-low} \\  
&C^{-1}r \Dbf_{x,k}(w,r) \leq C^{-1} r \mathbf{\Gamma}_{x,k} (w,r) \leq \Hbf_{x,k} (w,r) \leq C r \mathbf{\Gamma}_{x,k} (w,r) \leq C r \Dbf_{x,k}(w,r) \label{eq:simplify2-low} \\
&\mathbf{\Sigma}_{x,k} (w,r) \leq C r^2 \Dbf_{x,k} (w,r) \label{eq:simplify3-low} \\
&\Ebf_{x,k} (w,r) \leq C \Dbf_{x,k} (w,r) \label{eq:simplify4-low} \\
&\frac{\Hbf_{x,k}(w,\rho)}{\rho^{m-1}} = \frac{\Hbf_{x,k}(w,r)}{r^{m-1}}\exp\left(-C\int_\rho^r \Ibf_{x,k}(w,s)\,\frac{ds}{s} - O ({\boldsymbol{m}}_{x,k}) (r-\rho)\right) \label{eq:simplify5-low} \\
&\Hbf_{x,k} (w, r) \leq C \Hbf_{x,k} (w, \textstyle{\frac{r}{4}}) \label{eq:simplify6-low} \\
&\Hbf_{x,k} (w,r) \leq C r^{m+3 - 2\delta_2} \label{eq:simplify7-low} \\
&\Gbf_{x,k} (w,r) \leq C r^{-1} \Dbf_{x,k} (w,r) \label{eq:simplify8-low} \\
&|\partial_r \Dbf_{x,k} (w,r)| \leq C r^{-1} \Dbf_{x,k} (w,r) \label{eq:simplify9-low} \\
&|\partial_r \Hbf_{x,k} (w,r)| \leq C \Dbf_{x,k} (w,r)\, .  \label{eq:simplify10-low}
\end{align}
In particular:
\begin{align}
&|\mathbf{\Gamma}_{x,k} (w,r) - \Ebf_{x,k} (w,r)| \leq C {\boldsymbol{m}}_{x,k}^{\gamma_4} r^{\gamma_4} \Dbf_{x,k} (w,r) \label{eq:simplify11-low} \\
&\left|\partial_r \Dbf_{x,k} (w,r) - \frac{m-2}{r} \Dbf_{x,k} (w,r) - 2 \Gbf_{x,k} (w,r)\right| \leq C {\boldsymbol{m}}_{x,k}^{\gamma_4} r^{\gamma_4-1} \Dbf_{x,k} (w,r) \label{eq:simplify12-low} \\
&\left|\partial_r\Hbf_{x,k}(w,r) - \frac{m-1}{r}\Hbf_{x,k}(w,r) - 2\Ebf_{x,k}(w,r)\right| \leq C\mbf_{x,k}\Hbf_{x,k}(w,r) \label{e:H-radial-var} \\
&|\Lbf_{x,k}(w,r)| \leq C\mbf_{x,k}^{\gamma_4} r\Dbf_{x,k}(w,r) \label{e:L-bound} \\
&|\partial_r\Lbf_{x,k}(w,r)| \leq C\mbf_{x,k}^{\gamma_4}(r^{-1}\partial_r\Dbf_{x,k}(w,r)\Hbf_{x,k}(w,r))^{1/2} \label{e:L-der-bd}\\
&\partial_r \Ibf_{x,k} (w,r) \geq - C {\boldsymbol{m}}_{x,k}^{\gamma_4} r^{\gamma_4-1}\, . \label{eq:simplify13-low}
\end{align}
Here, $\gamma_4$ is as in Part \ref{pt:sing-deg}.
\end{lemma}

Observe that \eqref{eq:simplify1-low} is the consequence of Corollary \ref{c:ulb-frequency-2}, while the remaining estimates are an easy consequence of Corollary \ref{c:freq-mon}, Proposition \ref{p:firstvar}, combined with the construction of the sequence of adapted center manifolds and associated normal approximations. We refer the reader to \cite[Proof of Lemma 10.8]{DLSk2} for a more in-depth explanation.

\section{Spatial frequency variations}
A key aspect of the proof of Theorem \ref{t:main-low-one-piece} is a quantitative control on how much a given normal approximation $N = N_{x,k}$ deviates from being homogeneous on average between two scales, in terms of the frequency pinching. The latter is defined in the following way.

\begin{definition}\label{d:freqpinch-low}
	Let $T$ and $\Sfrak$ be as in Theorem \ref{t:main-low-one-piece}, let $x\in \Sfrak$ and let $\Mcal_{x,k}$ and $N_{x,k}$ be as in Section \ref{ss:adaptedintervals}. Consider $w \in \Mcal_{x,k}\cap\Bbf_1$ and a corresponding point $y = x + \gamma^k w$. Let $\rho, r>0$ be two radii satisfying 
 \begin{align}\label{e:inequality-scales}
 \eta \gamma^{k+1} \leq \rho
\leq r < 4\gamma^k\, .
 \end{align}
 We define the frequency pinching $W_\rho^r(x,k,y)$ around $y$ between the scales $\rho$ and $r$ by
	\[
		W_\rho^r(x,k,y) :=\left|\Ibf_{x,k}\Big(w,\gamma^{-k} r \Big) - \Ibf_{x,k}\Big(w,\gamma^{-k} \rho\Big)\right|.
	\]
\end{definition}

We have the following comparison to homogeneity for the normal approximations $N_{x,k}$.

\begin{proposition}\label{p:distfromhomog-low} 
	Assume $T$ and $\Sfrak$ are as in Theorem \ref{t:main-low-one-piece}. Let $x\in \Sfrak$ and $k \in \N$. Then there exists $C = C (m,n,Q,K,\gamma, \eta)$ such that, for any $w \in \mathcal{M}_{x,k} \cap \Bbf_1$ and any radii $r,\rho>0$ satisfying
  \begin{align}\label{e:inequality-scales-low}
 4 \eta \gamma^{k+1} \leq \rho
\leq r < 2\gamma^k\, ,
 \end{align}
 the following holds. Let $y= x + \gamma^k w$ and let $\Acal_{\frac{\rho}{4}}^{2r}(w) :=\left(\Bbf_{2r/\gamma^{k}}(w)\setminus \bar{\Bbf}_{\frac{\rho}{4\gamma^k}}(w)\right)\cap\Mcal_{x,k}$. Then
	\begin{align*}
	&\int_{\Acal_{\frac{\rho}{4}}^{2r}(w)} \sum_i \left|D(N_{x,k})_i(z)\frac{d(w,z)\nabla d(w,z)}{|\nabla d(w,z)|} - \Ibf_{x,k}(w,d (w,z)) (N_{x,k})_i(z)|\nabla d(w,z)|\right|^2\, \frac{d z}{d(w,z)} \\
	&\qquad\leq C \Hbf_{x,k}\left(w,\tfrac{2r}{\gamma^k}\right) \left(W_{\rho/8}^{4r}(x,k,y) + \boldsymbol{m}_{x,k}^{\gamma_4} \left(\frac{r}{\gamma^{k}}\right)^{\gamma_4}\right)\log\left(\frac{16r}{\rho}\right).
\end{align*}
\end{proposition}

\begin{proof}
    The argument is analogous to that of \cite{DLSk1}, but taking into account the additional term $\Lbf_{x,k}$ in the frequency. We include the details here for the purpose of clarity.

    At the risk of abusing notation, we will omit dependency on $x$, $w$ and $k$ for all quantities, for simplicity. For instance, we simply write $\Ibf(s)$ for the quantity $\Ibf_{x,k}(w,\gamma^{-k} s)$, and $W^{4r}_{\rho/8}$ for the pinching $W^{4r}_{\rho/8}(x,k,y)$.

    Invoking the estimates of Lemma \ref{l:simplify-low}, 
    %and recalling from Section \ref{s:BV} that $|\nabla d(w,z)| = 1 + O(\mbf^{1/2}|z-w|)$,
    we obtain
    \begin{align*}
        W^{4r}_{\rho/4}(y) &\geq \int_{\rho/4}^{4r} \partial_s \Ibf(s) \, ds = \int_{\rho/4}^{4r} \frac{\mathbf\Gamma(s) + s\partial_s\mathbf{\Gamma}(s) - \Ibf(s)\mathbf\Gamma(s)\partial_s\Hbf(s)}{\Hbf(s)} \, ds \\
        &\geq 2\int_{\rho/4}^{4r} \frac{s\Gbf(s)-\Ibf(s)}{\Hbf(s)} \, ds - C\mbf^{\gamma_4}\int_{\rho/4}^{4r}\frac{s^{\gamma_4}\Dbf(s)}{\Hbf(s)} + \frac{s^{1+\gamma_4}\Dbf(s)^2}{\Hbf(s)^2} \, ds \\
        &\geq 2\int_{\rho/4}^{4r} \frac{s\Gbf(s)-\Ibf(s)}{\Hbf(s)} \, ds - C\mbf^{\gamma_4}((4r)^{\gamma_4} - (\rho/4)^{\gamma_4}).
    \end{align*}
    Now notice that we may write
    \begin{align*}
        \int_{\rho/4}^{4r} &\frac{s\Gbf(s)-\Ibf(s)}{\Hbf(s)} \, ds \\
        &= \int_{\rho/4}^{4r} \frac{1}{s\Hbf(s)} \int_{\Mcal} -\phi'\left(\frac{d(w,z)}{s}\right)\frac{1}{d(w,z)}\left(\sum_j \left|DN_j \, \frac{d(w,z)\nabla d(w,z)}{|\nabla d(w,z)|}\right|^2\right. \\
        &\left.\qquad - 2\Ibf(s)\sum_j N_j \cdot\left(DN_j \, d(w,z)\nabla d(w,z)\right) + \Ibf(s)^2|N(z)|^2|\nabla d(w,z)|^2\right)\, d z \, d s \\
        &= \int_{\rho/4}^{4r} \frac{1}{s\Hbf(s)} \int -\phi'\left(\frac{d(w,z)}{s}\right)\frac{\xi (w,z,s)}{d(w,z)} \, d z \, d s,
    \end{align*}
    where
    \begin{align*}
        \xi(w,z,s) &= \sum_j \left| DN_j \frac{d(w,z)\nabla d(w,z)}{|\nabla d(w,z)|} - \Ibf(s) N_j(z)|\nabla d(w,z)|\right|^2\, .
	\end{align*}
    Thus,
    \begin{align*}
		W_{\rho/4}^{4r}(y) \geq 2\int_{\rho/4}^{4r} \frac{1}{s\Hbf(s)} \int_{\Acal_{s/2}^{s}(w)}\frac{\xi(w,z,s)}{d(w,z)} \, d z \, d s - C\boldsymbol{m}^{\gamma_4}(r^{\gamma_4}-\rho^{\gamma_4}).
    \end{align*}
    
    Let
    \[
        \zeta(w,z) := \sum_j \left| D N_j (z)\frac{d(w,z)\nabla d(w,z)}{|\nabla d(w,z)|} - \Ibf(d (w,z))  N_j(z)|\nabla d(w,z)| \right|^2.
    \]
    We then have
    \[
        \zeta(w,z) \leq 2\xi(w,z,s) + 2|\Ibf(s) - \Ibf(d(w,z))|^2| N(z)|^2 \leq 2\xi(w,z,s) + C W_{d (w,z)}^s(y) | N(z)|^2.
    \]
    Let us now control $W_{d (w,z)}^s(y)$ by $W_{\rho/8}^{4r}(y)$. In light of the quantitative almost-monotonicity \eqref{eq:simplify13-low} for $\Ibf$, for any radii $\eta \gamma^{k+1} < s < t \leq \gamma^{k}$ we have
    \begin{equation}\label{eq:frequencymtn-low}
        \Ibf(s) \leq \Ibf(t) + C{\boldsymbol{m}}^{\gamma_4} (t^{\gamma_4}- s^{\gamma_4}).
    \end{equation}
    For $s\in [\tfrac{\rho}{4},4r]$ this therefore yields
    \begin{equation}\label{eq:monopinching-low}
        W_{d (w,z)}^s(y) \leq W_{\rho/8}^{4r}(y) +C{\boldsymbol{m}}^{\gamma_4} ((4r)^{\gamma_4} -(\rho/8)^{\gamma_4}).
    \end{equation}
    Now observe that the estimate \eqref{eq:simplify6-low} gives
    \begin{align*}
        \int_{\Acal_{\rho/4}^{2r}(w)} \int_{d(w,z)}^{2d(w,z)} \frac{1}{s^2\Hbf(s)} \zeta(w,z) \, d s \, d z &\geq \frac{1}{\Hbf(2r)}\int_{\Acal_{\rho/4}^{2r}(w)} \zeta(w,z)\int_{d(w,z)}^{2d(w,z)} \frac{1}{s^2} \, d s \, d z \\
        &\geq \frac{1}{2\Hbf(2r)}\int_{\Acal_{\rho/4}^{2r}(w)} \frac{\zeta(w,z)}{d(w,z)} \, d z.
    \end{align*}
    Combining this with the preceding estimates, we arrive at
    \begin{align*}
        W^{4r}_{\rho/4}(y) &+ \log\left(\frac{16r}{\rho}\right)W_{\rho/8}^{4r}(y) \\
        &\geq C\int_{\rho/4}^{4r} \frac{1}{s\Hbf(s)}\int_{\Acal_{s/2}^s(w)} \frac{\zeta(w,z)}{d(w,z)}\, d z \, d s - C{\boldsymbol{m}}^{\gamma_4}r^{\gamma_4}\log\left(\frac{16r}{\rho}\right) - C{\boldsymbol{m}}^{\gamma_4}r^{\gamma_4} \\
        &\geq C \int_{\Acal_{\rho/4}^{2r}(w)} \int_{d(w,z)}^{2d(w,z)} \frac{1}{s^2\Hbf(s)} \zeta(w,z) \, d s \, d z - C{\boldsymbol{m}}^{\gamma_4}r^{\gamma_4}\log\left(\frac{16r}{\rho}\right)  - C{\boldsymbol{m}}^{\gamma_4}r^{\gamma_4} \\
        &\geq \frac{C}{\Hbf(2r)}\int_{\Acal_{\rho/4}^{2r}(w)} \frac{\zeta(w,z)}{d(w,z)} \, d z - C{\boldsymbol{m}}^{\gamma_4}r^{\gamma_4}\log\left(\frac{16r}{\rho}\right) - C{\boldsymbol{m}}^{\gamma_4}r^{\gamma_4}\\
        &\geq \frac{C}{\Hbf(2r)}\int_{\Acal_{\rho/4}^{2r}(w)} \frac{\zeta(w,z)}{d(w,z)} \, d z - C{\boldsymbol{m}}^{\gamma_4}r^{\gamma_4}\log\left(\frac{16r}{\rho}\right) - C{\boldsymbol{m}}^{\gamma_4}r^{\gamma_4}.
    \end{align*}
    Rearranging and again making use of \eqref{eq:frequencymtn-low}, this yields the claimed estimate.
\end{proof}

We will further require the following spatial variation estimate for the frequency, with control in terms of frequency pinching.

\begin{lemma}\label{l:spatialvarI-low}
	Let $T$ be as in Theorem \ref{t:main-low-one-piece}, let $x\in \Sfrak$ and $k \in \N$. Let $x_1,x_2 \in \Bbf_1 \cap  \Mcal_{x,k}$, $y_i = x + \gamma^k x_i$ and let $d(x_1,x_2) \leq \frac{\gamma^{-k} r}{8}$, where $r$ is such that 
	\[
	    8\eta\, \gamma^{k+1} < r \leq \gamma^k\, .
	\]
	Then there exists $C= C (m,n,Q,\gamma, \eta) > 0$ such that for any $z_1,z_2 \in [x_1,x_2]$, we have
	\begin{align*}
	&\left|\Ibf_{x,k}\left(z_1,\tfrac{r}{\gamma^{k}}\right) - \Ibf_{x,k}\left(z_2,\tfrac{r}{\gamma^{k}}\right)\right| \\
	&\qquad\leq C \left[\left(W_{r/8}^{4r}(x,k,y_1)\right)^{1/2} + \left(W_{r/8}^{4r}(x,k,y_2)\right)^{1/2} + \boldsymbol{m}_{x,k}^{\gamma_4/2} \left(\tfrac{r}{\gamma^{k}}\right)^{\gamma_4/2}\right]\frac{\gamma^{k}d(z_1,z_2)}{r}\, .
	\end{align*}
\end{lemma}

To prove Lemma \ref{l:spatialvarI-low}, we need the following spatial variational identities for $\Dbf_{x,k}$, $\Lbf_{x,k}$ and $\Hbf_{x,k}$.

\begin{lemma}\label{l:spatialvarDH-low}
    Suppose that $T$ is as in Theorem \ref{t:main-low-one-piece}, let $x\in \Sfrak$ and let $k \in \N$. Let $v$ be a continuous vector field on $\Mcal_{x,k}$. For any $w\in \Mcal_{x,k}\cap\Bbf_1$ and any $\eta \gamma \leq r \leq 2$, letting $\nu_w(z):= \nabla d(w,z)$, we have
    \begin{align}
        \partial_v \Dbf_{x,k}(w,r) &= -\frac{2}{r} \int_{\Mcal_{x,k}} \phi'\left(\frac{d(w,z)}{r}\right) \sum_i \langle \partial_{\nu_w} ({N}_{x,k})_i(z), \partial_v ({N}_{x,k})_i(z)\rangle \, d\Hcal^{m}(z) \label{e:spatial-var-D} \\
        &\qquad\qquad
        + O\big({\boldsymbol{m}}_{x,k}^{\gamma_4}\big)r^{\gamma_4 -1}\Dbf_{x,k}(w,r),\\
        \partial_v \Hbf_{x,k}(w,r) &= - 2 \sum_i \int_{\Mcal_{x,k}} \frac{|\nabla d(w,z)|^2}{d(w,z)}\phi'\left(\frac{d (w,z)}{r}\right)\langle \partial_v ({N}_{x,k})_i(z), ({N}_{x,k})_i(z) \rangle \, d\Hcal^{m}(z), \label{e:spatial-var-H} \\
        |\partial_v \Lbf_{x,k}(w,r)| &\leq C\mbf_{x,k}^{1/2} \|v\|_{C^0} \left(\Hbf_{x,k}(w,r)\partial_r \Dbf_{x,k}(w,r)\right)^{1/2}  \label{e:spatial-var-L}
    \end{align}
\end{lemma}

\begin{proof}
    The proof of \eqref{e:spatial-var-D} and \eqref{e:spatial-var-H} can be found in \cite[Lemma 11.4]{DLSk1}. To see the validity of \eqref{e:spatial-var-L}, omitting dependency on $x,k$ for simplicity, we simply write
    \[
        \partial_v \Lbf(w,r) = \sum_{i=1}^Q \sum_{l=1}^m (-1)^{l+1} \int_\Mcal \langle D_{\xi_l} N_i(z) \wedge \hat\xi_l(z) \wedge N_i, d\omega(z) \rangle \phi'\left(\frac{d(w,z)}{r}\right) \frac{\nabla d(w,z)}{r}\cdot v(z) \, d\Hcal^{m}(z).
    \]
    Combining with an application of Cauchy-Schwarz, this in turn yields the estimate
    \[
        |\partial_v \Lbf(w,r)| \leq C\|d\omega\|_{C^0(\Bbf_{\gamma^k}(x))} \|v\|_{C^0} \left(\Hbf(w,r)\partial_r \Dbf(w,r)\right)^{1/2}.
    \]
    Recalling \eqref{e:domega-balls}, the conclusion follows immediately.
\end{proof}

\begin{proof}[Proof of Lemma \ref{l:spatialvarI-low}]
    The majority of the proof follows in the same way as that of \cite[Lemma 5.4]{DLSk2} (cf. \cite[Lemma 11.3]{DLSk2}), but due to the additional error terms present when $T$ is semicalibrated, we repeat the full argument here.
    
    We will as usual omit dependency on $x,k$ for all objects. Let $x_1,x_2$ be as in the statement of the lemma and let $w$ lie in the geodesic segment $[x_1,x_2]\subset \Mcal$. Given a continuous vector field $v$ on $\Mcal$ and $\rho \in (8\eta\gamma, 1]$, we have
    \[
        \partial_v \Ibf(w,\rho) = \frac{\rho(\partial_v \Dbf(w,\rho) + \partial_v \Lbf(w,\rho))}{\Hbf(w,\rho)} - \frac{\Ibf(w,\rho)\partial_v\Hbf(w,\rho)}{\Hbf(w,\rho)}.
    \]
    Let $\mu_w$ be the measure on $\Mcal$ with density
    \[
        d\mu_w(z) = -\frac{|\nabla d(w,z)|}{d(w,z)}\phi'\left(\frac{d(w,z)}{r}\right)d\Hcal^m(z).
    \]
    Now let
    \[
        \eta_w(z) := d(w,z)\frac{\nabla d(w,z)}{|\nabla d(w,z)|} = \frac{d(w,z)}{|\nabla d(w,z)|}\nu_w(z),
    \]
    and choose $v$ to be the vector field
    \[
        v(z) = d(x_1,x_2)\frac{\nabla d(x_1,z)}{|\nabla d(x_1,z)|}.
    \]
    Applying Lemma \ref{l:spatialvarDH-low} with this choice of $v$ and exploiting the estimates of Lemma \ref{l:simplify-low}, for each $\rho = r\gamma^{-k}\in (8\eta\gamma,1]$ we have
    \begin{align*}
    	\partial_v \Ibf(w,\rho) &= \frac{2}{\Hbf(w,\rho)}\int_\Mcal \sum_i \langle \partial_{\eta_w} N_i, \partial_v N_i \rangle \, d\mu_w -\frac{2\Ibf(w,\rho)}{\Hbf(w,\rho)}\int_\Mcal |\nabla d(w,\cdot)| \sum_i \langle \partial_v N_i, N_i \rangle \, d\mu_w\\
    	&\qquad + C\mbf^{1/2} \rho^2 \Hbf(w,\rho)^{-1/2}(\partial_\rho \Dbf(w,\rho))^{1/2} + C\mbf^{\gamma_4} \rho^{\gamma_4}.
    \end{align*}
	Now observe that $v$ parameterizes the geodesic line segment $[x_1,x_2]\subset \Mcal$. Thus,
	\begin{align*}
		\partial_v N_i(z) &= d(x_1,z) \frac{\nabla d(x_1,z)}{|\nabla d(x_1,z)|} DN_i(z) - d(x_2,z) \frac{\nabla d(x_2,z)}{|\nabla d(x_2,z)|} DN_i(z) \\
		&=\partial_{\eta_{x_1}} N_i(z) - \partial_{\eta_{x_2}} N_i(z) \\
		&= \underbrace{\left(\partial_{\eta_{x_1}} N_i(z) - \Ibf(x_1, d (x_1, z))N_i(z)\right)}_{:= \Ecal_{1,i}} - \underbrace{\left(\partial_{\eta_{x_2}} N_i(z) - \Ibf(x_2, d (x_2, z))N_i(z)\right)}_{:= \Ecal_{2,i}} \\
		&\qquad + \underbrace{\Ibf(x_1,d (x_1, z)) - \Ibf(x_2,d (x_2, z))}_{:= \Ecal_3}N_i(z).
	\end{align*}
Combining this with the above calculation and once again the estimates in Lemma \ref{l:simplify-low}, we therefore obtain
\begin{align*}
\partial_v \Ibf(w,\rho) &= \frac{2}{\Hbf(w,\rho)} \int_{\Mcal} \sum_i \langle \partial_{\eta_w} N_i, \Ecal_{1,i} - \Ecal_{2,i} \rangle \, d \mu_w - 2\frac{\Ibf(w,\rho)}{\Hbf(w,\rho)}\int_\Mcal |\nabla d(w,\cdot)| \sum_i \langle \Ecal_{1,i} - \Ecal_{2,i}, N_i \rangle d\mu_w \\
&\qquad + \frac{2}{\Hbf(w,\rho)} \int_\Mcal \Ecal_3\sum_i \langle \partial_{\eta_w} N_i, N_i \rangle - 2\frac{\Ibf(w,\rho)}{\Hbf(w,\rho)} \int_\Mcal |\nabla d(w,\cdot)| \Ecal_3\sum_i |N_i|^2 d\mu_w \\
&\qquad + C \mbf^{1/2} \rho^2\Hbf(w,\rho)^{-1/2}(\partial_\rho\Dbf(w,\rho))^{1/2} + C\boldsymbol{m}^{\gamma_4}\rho^{\gamma_4} \\
&= \frac{2}{\Hbf(w,\rho)} \int_{\Mcal} \sum_i \langle \partial_{\eta_w} N_i, \Ecal_{1,i} - \Ecal_{2,i} \rangle d \mu_w - 2\frac{\Ibf(w,\rho)}{\Hbf(w,\rho)}\int_\Mcal |\nabla d(w,\cdot)| \sum_i \langle N_i, \Ecal_{1,i} - \Ecal_{2,i} \rangle d\mu_w \\
&\qquad + \frac{2\Ecal_3}{\Hbf(w,\rho)} \left(\int_\Mcal \sum_i \langle \partial_{\eta_w} N_i, N_i \rangle d\mu_w - \rho\Dbf(w,\rho)\right) \\
&\qquad + C\boldsymbol{m}^{\gamma_4}\rho^{\gamma_4}\Ibf(w,\rho).
\end{align*}
where in the last inequality, we have used that $\partial_\rho \Dbf(\rho) \leq \tfrac{1}{r} \Dbf(\rho)$ and \eqref{eq:simplify1-low}. Recalling that we aim to control the spatial frequency variation in terms of frequency pinching at the endpoints $x_1$ and $x_2$, let us rewrite $\Ecal_3$ in the following form:
\begin{align*}
\Ecal_3 &\leq |\left(\Ibf(x_1, d (x_1,z)) -\Ibf(x_1,\rho)\right)| + |\left(\Ibf(x_1,\rho) - \Ibf(x_2,\rho)\right)| + |\left(\Ibf(x_2,\rho) - \Ibf(x_2, d(x_2, z))\right)| \\
&=  W^{\gamma^k d (x_1, z)}_r(y_1) + W_{\gamma^k d (x_2, z)}^r(y_2) + |\Ibf(x_1,\rho) - \Ibf(x_2,\rho)|.
\end{align*}
Combining this with the Cauchy-Schwartz inequality and the estimates in Lemma \ref{l:simplify-low}, we have
\begin{align*}
\partial_v \Ibf(w,\rho) &\leq C\left[\int \sum_i\left(|\Ecal_{1,i}|^2 + |\Ecal_{2,i}|^2\right) \, d \mu_w\right]^{1/2}\left(\frac{1}{\Hbf(w,\rho)} \left[\int \sum_i |\partial_{\eta_w} N_i|^2 d\mu_w\right]^{1/2} +\frac{\Ibf(w,\rho)}{\Hbf(w,\rho)^{1/2}}\right) \\
&\qquad + C\boldsymbol{m}^{\gamma_4}\rho^{1+\gamma_4}\frac{|\Ibf(x_1,\rho) - \Ibf(x_2,\rho)|}{\Hbf(w,\rho)}\left(\Dbf(w,\rho) + |\Lbf(w,\rho)|\right) \\
&\qquad + C\boldsymbol{m}^{\gamma_4}\rho^{1+\gamma_4}\frac{W^{\gamma^k d (x_1, z)}_r (y_1) + W_{\gamma^k d (x_2, z)}^r(y_2)}{\Hbf(w,\rho)}\left(\Dbf(w,\rho) + |\Lbf(w,\rho)|\right) \\
&\qquad + C\boldsymbol{m}^{\gamma_4}\rho^{\gamma_4} \\
&\leq C\left[\int \sum_i\left(|\Ecal_{1,i}|^2 + |\Ecal_{2,i}|^2\right) \, d \mu_w\right]^{1/2}\left(\frac{1}{\Hbf(w,\rho)} \left[\int \sum_i |\partial_{\eta_w} N_i|^2 d\mu_w\right]^{1/2} +\frac{\Ibf(w,\rho)}{\Hbf(w,\rho)^{1/2}}\right) \\
&\qquad + C\boldsymbol{m}^{\gamma_4}\rho^{\gamma_4}\left(|\Ibf(x_1,\rho) - \Ibf(x_2,\rho)| + W^{\gamma^k d (x_1, z)}_r (y_1) + W_{\gamma^k d (x_2, z)}^r(y_2)\right) \\
&\qquad+ C\boldsymbol{m}^{\gamma_4}\rho^{\gamma_4}
\end{align*}
Applying Proposition \ref{p:distfromhomog-low}, for $\ell = 1,2$ we further have
\begin{align*}
\int \sum_i|\Ecal_{\ell,i}|^2 \, d \mu_w &= - \int \sum_i|\Ecal_{\ell,i}|^2(z) \frac{|\nabla d(w,z)|}{d(w,z)} \phi'\left(\frac{d(w,z)}{\rho}\right) \, d\Hcal^m(z) \\
&\leq C\Hbf(x_\ell,2\rho) (W_{r/8}^{4r}(y_\ell)+\boldsymbol{m}^{\gamma_4} \rho^{\gamma_4}).
\end{align*}

Together with the doubling estimate \eqref{eq:simplify5-low} (which applies since $d(x_\ell,w)\leq \rho$) and the uniform upper frequency bound \eqref{eq:simplify1-low}, we thus obtain the estimate
\begin{align*}
\partial_v \Ibf(w,\rho) &\leq C\left[(W_{r/8}^{4r}(y_1)+\boldsymbol{m}^{\gamma_4} \rho^{\gamma_4})^{1/2} + (W_{r/8}^{4r}(y_2)+\boldsymbol{m}^{\gamma_4} \rho^{\gamma_4})^{1/2}\right] + C\boldsymbol{m}^{\gamma_4}\rho^{\gamma_4}. \\
&\leq C \left[W_{\frac{r}{8}}^{4r}(y_1)^{1/2} + W_{\frac{r}{8}}^{4r}(y_2)^{1/2} \right] + C\boldsymbol{m}^{\gamma_4/2}\rho^{\gamma_4/2}.
\end{align*}
Integrating this inequality over the geodesic segment $[z_1,z_2]\subset \Mcal$, the proof is complete.
    
\end{proof}

\section{Quantitative spine splitting}\label{s:splitting}
Following the notation of \cite{DLSk2}, for a finite set of points $X=\{x_0,\dots, x_k\}$ we let $V(X)$ denote the affine subspace given by
\[
	V(X) := x_0 + \spn\left(\{x_1-x_0,\dots,x_k-x_0\}\right).
\]
We recall the following quantitative notions of linear independence and spanning, first introduced in \cite{DLMSV}.

\begin{definition}
	We say that a set $X = \{x_0, x_1,\dots,x_k\} \subset \Bbf_r(w)$ is $\rho r$-linearly independent if 
	\[
		d(x_i, V (\{x_0, \ldots, x_{i-1}\})) \geq \rho r \qquad \mbox{for all $i=1, \ldots , k$}
	\]
	We say that a set $F \subset \Bbf_r(w)$ $\rho r$-spans a $k$-dimensional affine subspace $V$ if there is a $\rho r$-linearly independent set of points $ X= \{x_i\}_{i=0}^k \subset F$ such that $V= V (X)$.
\end{definition}

We have the following two quantitative splitting results (cf. \cite[Section 6]{DLSk2}).

\begin{lemma}\label{l:Q-pts-clearing}
	Suppose that $T$ is as in Theorem \ref{t:main-low-one-piece} and let $\rho,\bar \rho \in (0,1]$, $\tilde \rho \in (\eta,1]$ be given radii. There exists $\eps^*=\eps^*(m,n,Q,\gamma,K,\rho,\tilde\rho,\bar\rho)>0$ such that for $\eps_4\leq \eps^*$, the following holds. Suppose that for some $x\in \Sfrak$ and $r\in [\gamma^{k+1},\gamma^k]$, there exists a collection of points $X=\{x_i\}_{i=0}^{m-2}\subset\Bbf_r(x)\cap\Sfrak$ satisfying the properties
	\begin{itemize}
		\item $X$ is $\rho r$-linearly independent;
		\item the nearest points $z_i$ to $x_i$ such that $\gamma^{-k}(z_i-x) \in \Mcal_{x,k}$ satisfy
		\[
			W^{2r}_{\tilde\rho r}(x,j(k),z_i) < \eps^*.
		\]
		Then $\Sfrak \cap (\Bbf_r\setminus \Bbf_{\bar \rho r}(V(X)))=\emptyset$.
	\end{itemize}
\end{lemma}

\begin{lemma}\label{l:spine-splitting}
	Suppose that $T$ is as in Theorem \ref{t:main-low-one-piece}, and let $\rho,\bar \rho \in (0,1]$, $\tilde \rho \in (\eta,1]$ be given radii. For any $\delta>0$, there exists $\eps^\dagger=\eps^\dagger(m,n,Q,\gamma,K,\rho,\tilde\rho,\bar\rho,\delta)\in (0, \eps^*]$ such that for $\eps_4\leq \eps^\dagger$, the following holds. Suppose that for some $x\in \Sfrak$ and $r\in [\gamma^{k+1},\gamma^k]$, there exists a collection of points $X=\{x_i\}_{i=0}^{m-2}\subset\Bbf_r(x)\cap\Sfrak$ satisfying the properties
	\begin{itemize}
		\item $X$ is $\rho r$-linearly independent;
		\item the nearest points $z_i$ to $x_i$ such that $\gamma^{-k}(z_i-x) \in \Mcal_{x,k}$ satisfy
		\[
		W^{2r}_{\tilde\rho r}(x,j(k),z_i) < \eps^\dagger.
		\]
	\end{itemize}
	Then for each $\zeta_1,\zeta_2\in \Bbf_r(x)\cap\Bbf_{\eps^\dagger r}(V(X))$ and each pair of radii $r_1,r_2\in [\bar\rho,1]$, letting $w_j$ denote the nearest point to $\gamma^{-k}(\zeta_j-x)$ that belongs to $\Mcal_{x,k}$, the following estimate holds:
    \[
        |\Ibf_{x,k}(w_1,r_1) - \Ibf_{x,k}(w_2,r_2)| \leq \delta.
    \]
\end{lemma}

Given the compactness argument in Section \ref{ss:compactness}, the proof of both Lemma \ref{l:Q-pts-clearing} and Lemma \ref{l:spine-splitting} follows in exactly the same way as that of \cite[Lemma 6.2, Lemma 6.3]{DLSk2} respectively. We therefore omit the arguments here.

\section{Jones' $\beta_2$ control and rectifiability}
In this section, we combine all of the previous estimates of the preceding sections in this part, in order to gain control on Jones' $\beta_2$ coefficients associated to the measure $\Hcal^{m-2}\mres\Sfrak$, providing a quantitative $L^2$-flatness control on the flat density $Q$ singularities of $T$ with degree strictly larger than 1.

We begin by recalling the following definition.

\begin{definition}\label{def:beta2}
        Given a Radon measure $\mu$ on $\R^{m+n}$, we define the $(m-2)$-dimensional Jones' $\beta_2$ coefficient of $\mu$ as
        \[
            \beta_{2,\mu}^{m-2}(x,r) := \inf_{\text{affine $(m-2)$-planes $L$}} \left[r^{-(m-2)} \int_{\Bbf_r(x)} \left(\frac{\dist(y,L)}{r}\right)^2 \, d\mu(y)\right]^{1/2}.
        \]
    \end{definition}
    We have the following key estimate on $\beta_{2,\mu}^{m-2}$ for a measure $\mu$ supported on $\Sfrak$.

    \begin{proposition}\label{prop:beta2control-low}
    There exist $\alpha_0 = \alpha_0(m,n,Q) > 0$, $\hat \eta= \hat \eta (m) \in (0, \tfrac{1}{8})$, $\hat\eps=\hat\eps(m,n,Q,K)\in (0,\eps^\dagger]$, $C(m,n,Q,K) > 0$ with the following property. Suppose that $\eps_4\in (0,\hat\eps]$, $\eta\in (0,\hat\eta]$ and let $T$ and $\Sfrak$ be as in Theorem \ref{t:main-low-one-piece}. Suppose that $\mu$ is a finite non-negative Radon measure with $\spt (\mu) \subset \Sfrak$ and let $x_0 \in \Sfrak$. Then for all $r \in (8\eta\gamma^{k+1} ,\gamma^k]$ we have
    \begin{align*}
            [\beta_{2,\mu}^{m-2}(x_0, r/8)]^2 &\leq Cr^{-(m-2)} \int_{\Bbf_{r/8} (x_0)} W^{4r}_{r/8}\left(x_0, k,\mathbf{p}_{x_0, k} (x)\right)\, d\mu (x)\\
            &\qquad + C \boldsymbol{m}_{x_0,k}^{\alpha_0} r^{-(m-2-\alpha_0)} \mu(\Bbf_{r/8}(x_0)).
    \end{align*}
\end{proposition}

With the estimates of Proposition \ref{p:distfromhomog-low}, Lemma \ref{l:simplify-low} and Lemma \ref{l:spatialvarI-low} at hand, the proof of Proposition \ref{prop:beta2control-low} follows by exactly the same reasoning as that of \cite[Proposition 7.2]{DLSk2} (cf. \cite[Proposition 13.3]{DLSk2}). We thus simply refer the reader to the argument therein.

\subsection{Proof of Theorem \ref{t:main-low-one-piece}}
The proof of rectifiability and the content bound \eqref{e:Minkowski-low} follows via the same procedure as that in \cite{DLSk2}, crucially making use of Proposition \ref{prop:beta2control-low}, the quantitative splitting results of Section \ref{s:splitting} and the BV estimate of Proposition \ref{p:bv-adapted}. Note that in order to establish the rectifiability alone, one may make use of \cite{DLF} in place of the rectifiable Reifenberg arguments of Naber-Valtorta, but this \emph{does not} allow one to obtain the Minkowski content bound \eqref{e:Minkowski-low}. We do not include the details here.

\part{Points with singularity degree 1}\label{pt:deg1}
In this part we conclude the proof of the main result of this work, Theorem \ref{t:main}, by showing rectifiability of the remaining part of $\mathfrak{F}_Q(T)$, as well as the $\mathcal{H}^{m - 2}$-uniqueness of tangent cones. Namely, we prove Theorem \ref{t: main negligible}. We follow the same outline as in \cite{DMS}; a key preliminary result is a decay theorem for the excess to $(m-2)$-invariant cones formed from superpositions of planes, whenever $T$ is much closer to such a cone than any single plane, under the assumption of no density gaps for $T$ near the spines of such cones. Before coming to the statement of this theorem, let us first recall some notation introduced in \cite{DMS}. We begin by defining the cones of interest. As done in the other parts, we will merely point out the differences with \textit{loc. cit}, and explain the changes needed. 

\begin{definition}\label{def:cones}
Let $Q\geq 2$ be a fixed integer. We denote by $\mathscr{C} (Q)$ those subsets of $\mathbb R^{m+n}$ which are unions of $1 \leq N\leq Q$ $m$-dimensional planes (affine subspaces) $\pi_1, \ldots, \pi_N$ 
 for which $\pi_i \cap \pi_j$ is the same $(m-2)$-dimensional plane $V$ for every pair of indices $(i,j)$ with $i<j$.
 
We will use the notation $\mathscr{P}$ for the subset of those elements of $\mathscr{C} (Q)$ which consist of a single plane; namely, with $N=1$. For $\mathbf{S}\in \mathscr{C} (Q)\setminus \mathscr{P}$, the $(m-2)$-dimensional plane $V$ described in (i) above is referred to as the {\em spine} of $\mathbf{S}$ and will often be denoted by $V (\mathbf{S})$.
\end{definition}

Let us now recall the \emph{conical $L^2$ height excess} between $T$ and elements in $\Cscr(Q)$.
\begin{definition}\label{def:L2_height_excess}
	Given a ball $\Bbf_r(q) \subset \R^{m+n}$ and a cone $\mathbf{S}\in \mathscr{C} (Q)$, the \emph{one-sided conical $L^2$ height excess of $T$ relative to $\Sbf$ in $\Bbf_r(q)$}, denoted $\hat{\Ebf}(T, \mathbf{S}, \Bbf_r(q))$, is defined by
	\[
		\hat{\Ebf}(T, \mathbf{S}, \Bbf_r(q)) := \frac{1}{r^{m+2}} \int_{\Bbf_r (q)} \dist^2 (p, \mathbf{S})\, d\|T\|(p).
	\]
At the risk of abusing notation, we further define the corresponding \emph{reverse one-sided excess} as
 \[
\hat{\Ebf} (\mathbf{S}, T, \Bbf_r (q)) := \frac{1}{r^{m+2}}\int_{\Bbf_r (q)\cap \mathbf{S}\setminus \Bbf_{ar} (V (\mathbf{S}))}
\dist^2 (x, {\rm spt}\, (T))\, d\mathcal{H}^m (x)\, ,
\]
where $a=a(Q,m)$ is a geometric constant, to be determined later (see the discussion preceding Remark \ref{r:contradictio-argument-limit-area-min}). The \emph{two-sided conical $L^2$ height excess} is then defined by
\[
    \mathbb{E} (T, \mathbf{S}, \Bbf_r (q)) :=
\hat{\Ebf} (T, \mathbf{S}, \Bbf_r (q)) + \hat{\Ebf} (\mathbf{S}, T, \Bbf_r (q))\, .
\]
We finally recall the notion of \emph{planar $L^2$ height excess}, is given by
\[
\Ebf^p (T, \Bbf_r (q)) = \min_{\pi\in \mathscr{P} (q)} \hat{\Ebf} (T, \pi, \Bbf_r (q))\, .
\]
\end{definition}

We may now state our key excess decay theorem. This is based on the excess decay theorem \cite[Lemma 1]{Simon_cylindrical}, but in the latter, there is a built-in multiplicity one assumption, ruling out branch point singularities a priori. Such a decay theorem was more recently proven in \cite{W14_annals} (see Section 13 therein and also \cite[Lemma 5.6 and Lemma 12.1]{KW} and \cite[Theorem 3.1]{MW}) in a higher multiplicity setting for stable minimal hypersurfaces, but in codimension 1, where one has a sheeting theorem. On the other hand, our version of this theorem is both in a higher multiplicity and higher codimension setting, thus requiring new techniques {as in \cite{DMS}} that overcome the lack of sheeting.

Throughout this part, we will often work with error terms involving the quantity $\|d\omega\|_{C^0}$, where $\omega$ is as in Assumption \ref{a:main-2}. Thus, for the purpose of convenience, we will henceforth use the notation
\[
    \Omega := \|d\omega\|_{C^0(\Bbf_{6\sqrt{m}})}.
\]

\begin{theorem}[Fine Excess Decay Theorem]\label{c:decay}
Let $\delta_3 = \frac{\delta_2}{2}$, for the positive parameter $\delta_2$ fixed as in \cite{Spolaor_15} (cf. Parts \ref{pt:sing-deg} and \ref{pt:NV}). For every $Q,m,n$, and $\varsigma>0$, there are positive constants $\varepsilon_0 = \varepsilon_0(Q,m,n, \varsigma) \leq \frac{1}{2}$, $r_0 = r_0(Q,m,n, \varsigma) \leq \frac{1}{2}$ and $C = C(Q,m,n)>0$ with the following property. Assume that 
\begin{itemize}
\item[(i)] $T$ and $\omega$ are as in Assumption \ref{a:main};
\item[(ii)] $\|T\| (\Bbf_1) \leq (Q+\frac{1}{2}) \omega_m$;
\item[(iii)] There is $\mathbf{S}\in \mathscr{C} (Q)\setminus\Pscr$ such that 
\begin{equation}\label{e:smallness}
\mathbb{E} (T, \mathbf{S}, \Bbf_1) \leq \varepsilon_0^2 \mathbf{E}^p (T, \Bbf_1)\, 
\end{equation}
and 
\begin{equation}\label{e:no-gaps}
\Bbf_{ \varepsilon_0} (\xi) \cap \{p: \Theta (T,p)\geq Q\}\neq \emptyset \qquad \forall \xi \in V (\mathbf{S})\cap \Bbf_{1/2}\, ;
\end{equation}
\item[(iv)] $\Omega^{2-2\delta_3} \leq \varepsilon_0^2 \mathbb{E} (T, \mathbf{S}', \Bbf_1)$ for {\em any} $\mathbf{S}'\in \mathscr{C} (Q)$.
\end{itemize}
Then there is a $\mathbf{S}'\in \mathscr{C} (Q) \setminus \mathscr{P}$ such that 
\begin{enumerate}%[\normalfont(a)]
    \item [\textnormal{(a)}] $\mathbb{E} (T, \mathbf{S}', \Bbf_{r_0}) \leq \varsigma \mathbb{E} (T, \mathbf{S}, \Bbf_1)\,$ \label{e:decay} \\
    \item [\textnormal{(b)}] $\dfrac{\mathbb{E} (T, \mathbf{S}', \Bbf_{r_0})}{\mathbf{E}^p (T, \Bbf_{r_0})} 
\leq 2 \varsigma \dfrac{\mathbb{E} (T, \mathbf{S}, \Bbf_1)}{\mathbf{E}^p (T, \Bbf_1)}$ \\
    \item [\textnormal{(c)}] $\dist^2 (\Sbf^\prime \cap \Bbf_1,\Sbf\cap \Bbf_1) \leq C \mathbb{E} (T, \mathbf{S}, \Bbf_1)$\label{e:cone-change}
    \item[\textnormal{(d)}] $\dist^2 (V (\mathbf{S}) \cap \Bbf_1, V (\mathbf{S}')\cap \Bbf_1) \leq C \dfrac{\mathbb{E}(T,\mathbf{S},\Bbf_1)}{\Ebf^p(T,\Bbf_1)}$\, .\label{e:spine-change}
    \end{enumerate}
\end{theorem}

With Theorem \ref{c:decay} at hand, the conclusion of Theorem \ref{t: main negligible} follows by combining {it} with a covering procedure analogous to the one in \cite{Simon_cylindrical}; see \cite[Section 14]{DMS} for the details.

\subsection{Outline of proof of Theorem \ref{c:decay}}
The proof of Theorem \ref{c:decay} follows the same outline as that of \cite[Theorem 2.5]{DMS}. We first establish an $L^2-L^\infty$ height bound and tilt excess estimate, analogous to \cite[Theorem 3.2]{DMS}. However, all instances of $\Abf$ in the error terms are replaced by $\Omega$ and $\Omega^{1-\delta_3}$ in the height bound and tilt excess estimate respectively. This will be done in Section \ref{s:height-bd}. 
In Section \ref{s:graphs} we then use the height bound to verify that the graphical parameterization results of \cite[Section 8]{DMS} relative to \emph{balanced cones} in $\Cscr(Q)$ (see Definition \ref{d:frank}) still hold true when $T$ is semicalibrated, again with $\Abf$ replaced by $\Omega^{1-\delta_3}$ in the errors. In Section \ref{s:balancing} we provide the analogues of the cone balancing results of \cite[Section 9]{DMS}, which are required in order to guarantee the hypotheses on the cones $\Sbf\in\Cscr(Q)$ in order to build the graphical parameterizations of the preceding section. In Section \ref{s:nonconc} we then verify that the Simon estimates at the spine \cite[Section 11]{DMS} remain valid; the key difference is again the fact that all appearances of $\Abf$ in the errors become $\Omega^{1-\delta_3}$. In Section \ref{s:blowup}, we conclude with a final blow-up procedure, analogous to that in \cite[Section 13]{DMS}.

\section{$L^2-L^\infty$ height bound}\label{s:height-bd}
In this section we establish Allard-type tilt-excess and $L^\infty$ estimates relative to disjoint collections of parallel planes, analogous to those in \cite[Part 1]{DMS}, but for the class of semicalibrated currents. For the remainder of this section, we make the following additional assumption.
\begin{assumption}\label{a:height-main}
    $Q$, $m \geq 3$, $n \geq 2$ are fixed positive integers. $T$ and $\omega$ are as in Assumption \ref{a:main-2}. For some oriented $m$-dimensional plane $\pi_0 \equiv \R^m\times \{0\} \subset\R^{m+n}$ passing through the origin and some positive integer $Q$, we have 
    \[
    (\mathbf{p}_{\pi_0})_\sharp T\res \Cbf_2 = Q \llbracket B_2\rrbracket\, ,
    \]
    and $\|T\|(\Cbf_2) \leq (Q+\frac{1}{2})\omega_m 2^m$.
\end{assumption}
The main result of this section is the following (note that a scaling argument gives the corresponding estimates for arbitrary centers and scales).

\begin{theorem}[$L^\infty$ and Tilt-Excess Estimates]\label{thm:main-estimate} For every $1\leq r < 2$, $Q$, and $N$, there is a positive constant $\bar{C} = \bar{C} (Q,m,n,N,r)>0$ with the following property. Suppose that $T$, $\omega$ and $\pi_0$ are as in Assumption \ref{a:height-main}, let $p_1, \ldots , p_N \in \pi_0^\perp$ be distinct points, and set $\boldsymbol{\pi}:= \bigcup_i p_i+\pi_0$. 
Let
\begin{equation}\label{e:L2-excess}
E := \int_{\mathbf{C}_2} \dist^2 (p, \boldsymbol{\pi}) \, d\|T\| (p)\, .
\end{equation}
Then
\begin{equation}\label{e:tilt-estimate}
\Ebf(T,\Cbf_r, \pi_0)\leq \bar{C} (E + \Omega^{2})\, 
\end{equation}
and, if $E\leq 1$,
\begin{equation}\label{e:Linfty-estimate}
\spt (T) \cap \mathbf{C}_r \subset \{p : \dist (p, \boldsymbol{\pi})\leq \bar{C} (E^{1/2} + \Omega^{1 - \delta_3})\}\, .
\end{equation}
\end{theorem}

A consequence of Theorem \ref{thm:main-estimate} is the following, which replaces \cite[Corollary 3.3]{DMS}. 
\begin{corollary}\label{c:splitting-0}
Let $N$ be a positive integer. There is a positive constant $\delta = \delta (Q,m,n, N)$ with the following properties. Assume that:
\begin{itemize}
    \item[(i)] $T$ and $\omega$ are as in Assumption \ref{a:main-2}, and for some positive $r \leq \frac{1}{4}$ and $q\in \spt(T)\cap\Bbf_1$ we have 
    \begin{itemize}
        \item[$\bullet$] $\partial T \res \mathbf{C}_{4r} (q) = 0$;
        \item[$\bullet$] $(\mathbf{p}_{\pi_0})_\sharp T \res \Cbf_{4r}(q)= Q \llbracket B_{4r} (q)\rrbracket$;
        \item[$\bullet$] $\|T\| (\mathbf{C}_{2r} (q)) \leq \omega_m (Q+\frac{1}{2}) (2r)^m$;
    \end{itemize}
    \item[(iii)] $p_1, \ldots, p_N\in \R^{m+n}$ are distinct points with $\mathbf{p}_{\pi_0} (p_i)= q$ and $\varkappa:=\min \{|p_i-p_j|: i<j\}$;
    \item[(iv)] $\pi_1, \ldots, \pi_N$ are oriented planes passing through the origin with 
    \begin{equation}\label{e:smallness-tilted-pi-1}
    \tau := \max_i |\pi_i-\pi_0|\leq \delta \min\{1, r^{-1} \varkappa\}\, ;
    \end{equation}
    \item[(v)] Upon setting $\boldsymbol{\pi} = \bigcup_i (p_i+ \pi_i)$, we have 
    \begin{align}
    & (r\Omega)^2 + (2r)^{-m-2} \int_{\mathbf{C}_{2r} (q)} \dist^2 (p, \boldsymbol{\pi}) d\|T\| 
    \leq \delta^2 \min \{1, r^{-2}\varkappa^2\}\, .\label{e:smallness-tilted-pi-2}
    \end{align}
\end{itemize}
Then $T\res \mathbf{C}_r (q) = \sum_{i=1}^N T_i$ where
\begin{itemize}
\item[(a)] Each $T_i$ is an integral current with $\partial T_i \res \mathbf{C}_r (q) = 0$;
\item[(b)] $\dist (q, \boldsymbol{\pi}) = \dist (q, p_i+\pi_i)$ for each $q\in \spt (T_i)$;
\item[(c)] $(\mathbf{p}_{\pi_0})_\sharp T_i = Q_i \llbracket B_r (q)\rrbracket$ for some non-negative integer $Q_i$.
\end{itemize}
\end{corollary}
The proof of Corollary \ref{c:splitting-0} follows verbatim the one of \cite[Corollary 3.3]{DMS}, replacing \cite[Lemma 1.6]{DLS16centermfld} with \cite[Lemma 2.2]{Spolaor_15}. 

We now recall the notion of non-oriented tilt-excess, previously introduced in \cite{DLHMS, DMS}. More precisely, given an $m$-dimensional plane $\pi$ and a cylinder $\mathbf{C} = \mathbf{C}_r (q, \pi)$, recall that the \textit{non-oriented tilt excess} is given by
\begin{equation}\label{e:nonoriented}
\mathbf{E}^{no} (T, \mathbf{C}):= \frac{1}{2\omega_m r^m} \int_{\mathbf{C}} |\mathbf{p}_{T} - \mathbf{p}_{\pi}|^2\, d\|T\| \, ,
\end{equation}
where $T (x)$ denotes the (approximate) tangent plane to $T$ at $x$. Note that here, neither plane is oriented for the projections. In particular, we have $\Ebf^{no}(T,\Cbf) \leq C\Ebf(T,\Cbf)$. In contrast, the reverse inequality is more subtle due to possible cancellation phenomena. Nonetheless, we have the following for semicalibrated currents.  
\begin{proposition}\label{p:o<no}
For every $1\leq r<2$ there is a constant $\bar{C}= \bar{C}(Q,m,n, r)$ such that, if $T$, and $\omega$ are as in Assumption \ref{a:height-main}, then
\begin{equation}\label{e:o<no}
\mathbf{E} (T, \mathbf{C}_r) \leq 
\bar{C} (\mathbf{E}^{no} (T, \mathbf{C}_2) + \Omega^2)\, .
\end{equation}
\end{proposition}

The proof of this follows that of \cite[Proposition 4.1]{DMS} (see also \cite[Theorem 16.1]{DLHMS}), replacing the height bound \cite[Theorem A.1]{DLS16centermfld} with \cite[Theorem 1.4]{Spolaor_15} in the case in which the supports of the currents are not equibounded. Furthermore, instead of Almgren's strong Lipschitz approximation for area-minimizing integral currents, we invoke its variant \cite[Theorem 1.4]{DLSS1} for $\Omega$-minimal currents (see Definition 1.1 therein). One can then refine the approximation in the same way to deduce the desired contradiction, and conclude the proof. 

In the case $N = 1$, the conclusions of Theorem \ref{thm:main-estimate} are given by Allard's tilt excess estimate for varifolds with bounded generalized mean curvature (see \cite[Proposition 4.1]{DL-All}), together with \cite[Lemma 1.7]{Spolaor_15}. It suffices to verify that the generalized mean curvature of $T$ can be controlled uniformly by $\Omega$, which indeed is the case by the following reasoning. Recall that $T$ satisfies the first variation identity \eqref{eq: first variation} for any test vector field $\chi\in C_c^\infty(\Bbf_{6\sqrt{m}};\R^{m+n})$. Thus, after applying the Riesz Representation Theorem, we infer 
    \begin{align*}
        - \int \chi \cdot \vec{H}_T \, d\Vert T \Vert = T(d\omega \mres \chi) = \int \langle d\omega \mres \chi, \vec{T} \rangle \, d \Vert T \Vert \leq \Vert d\omega \Vert_{C^0} \int \vert \chi \vert \, d\Vert T \Vert. 
    \end{align*}
    Taking the supremum in the above, and recalling the definition of $L^\infty$ norm in terms of its dual $L^1$ norm (both with respect to the local Radon measure $\|T\|$), we deduce that the generalized mean curvature vector $\vec{H}_T$ of $T$ satisfies the estimate
    \begin{equation}\label{e:H-Linfty-est}
        \Vert \vec{H}_T \Vert_{L^\infty(\Bbf_{6\sqrt m},\|T\|)} \leq \Vert d\omega \Vert_{C^0(\Bbf_{6\sqrt{m}})} = \Omega \, .
    \end{equation}
Note that the error term in the latter is indeed quadratic in $\Omega$, unlike Theorem 1.5 therein. In addition, note that \cite[Lemma 1.7]{Spolaor_15} does not require any smallness on the tilt excess.

\begin{proposition}[$N=1$ case of Theorem \ref{thm:main-estimate}]\label{p:step1}
Theorem \ref{thm:main-estimate} holds when $N=1$, for any $Q$ and $N$.
\end{proposition}

The remainder of this section is dedicated to the proof of Theorem \ref{thm:main-estimate}. First of all, observe that the results of \cite[Section 4.4]{DMS}, handling the proof of Theorem \ref{thm:main-estimate} in the case where the planes in $\boldsymbol\pi$ are well-separated, remain valid when $T$ is semicalibrated. Indeed, they are all either reliant on the following variant of \cite[Lemma 4.7]{DMS}, or are proven by analogous reasoning to it.

\begin{lemma}\label{l:simpler}
For every $1\leq \bar r<2$ there is a constant $\sigma_2 = \sigma_2 (Q,m,n,2-\bar r)>0$ with the following property. Let $T$ and $\omega
$ be as in Assumption \ref{a:height-main}, suppose that $p_1, \ldots, p_N\in \pi_0^\perp$ are distinct points, and let $\boldsymbol{\pi}:=\bigcup_i p_i+\pi_0$. Assume that $E$ is as in \eqref{e:L2-excess}, let $H:= \min \{|p_i-p_j|:i \neq j\}$ and suppose that 
\begin{equation}\label{e:scaling-not-broken}
E  \leq \sigma_2 \qquad \mbox{and} \qquad H \geq 1\, . 
\end{equation}
Then 
\begin{equation}\label{e:separated}
{\rm spt}\, (T) \cap \mathbf{C}_{\bar r} \subset \{q: \dist (q, \boldsymbol{\pi})\leq \textstyle{\frac{H}{4}}\},
\end{equation}
and, in particular, all the conclusions of Theorem \ref{thm:main-estimate} hold in $\Cbf_{\bar r}$. 
\end{lemma}

We defer the reader to \cite{DMS} for the proof of this, which remains completely unchanged in the setting herein, in light of the almost-monotonicity of mass ratios that holds for all almost-minimizing currents (see, for instance, \cite[Proposition 2.1]{DLSS-uniqueness}).

As in \cite{DMS}, we then prove the estimates \eqref{e:tilt-estimate} and \eqref{e:Linfty-estimate} separately. For the former, we need two approximate estimates on the oriented tilt-excess; \cite[Lemma 5.1]{DMS} and \cite[Proposition 5.2]{DMS}, but rewritten for a semicalibrated current in $\R^{m+n}$, in which case the all instances of $\Abf$ are replaced with $\Omega$. More precisely, the former estimate reads as follows.

\begin{lemma}\label{l:tilt-1}
For every pair of radii $1\leq r<R \leq 2$ there are constants $\bar{C}=\bar{C}(Q,m,n,R-r)>0$ and $\gamma = \gamma (Q,m,n)>0$ such that the following holds. Let $T$ and $\omega$ be as in Assumption \ref{a:height-main} and let $\boldsymbol\pi$, $E$, $H$ be as in Lemma \ref{l:simpler}. Then
\begin{equation}\label{e:tilt-est-1}
\mathbf{E} (T, \mathbf{C}_r)\leq \bar{C} (E+ \Omega^2) + \bar{C} \left(\frac{E}{H^2}\right)^\gamma \mathbf{E} (T, \mathbf{C}_R) + \bar{C}
\mathbf{E} (T, \mathbf{C}_R)^{1+\gamma}\, .
\end{equation}
\end{lemma}

Observe that the majority of the proof of \cite[Lemma 5.1]{DMS} is in fact written for currents with bounded generalized mean curvature with $\| \vec H_T \|_{L^\infty} \leq C\Abf$, where $\Abf$ is the second fundamental form of the ambient Riemannian manifold. Thus, armed instead with the estimate \eqref{e:H-Linfty-est} (which allows us to replace $\Abf^2$ by $\Omega^2$) and replacing the use of Almgren's strong excess estimate with \cite[Theorem 4.1]{DLSS1} we are able to follow the proof verbatim, obtaining the desired estimate \eqref{e:tilt-est-1}. Note that, analogously to \cite{DMS}, we may indeed apply \cite[Theorem 4.1]{DLSS1} since we may assume that the tilt excess $\Ebf(T, \Cbf_{r_2})$ falls below the threshold $\eps_{21}$ therein, for $r_2=\tfrac{2R+r}{3}$ as in the proof of \cite[Lemma 5.1]{DMS}. Indeed, the reasoning for this remains unchanged, given Lemma \ref{l:simpler} (in place of \cite[Lemma 4.7]{DMS}). 
The estimate of Lemma \ref{l:tilt-1} in turn yields the following bootstrapped estimate, under the assumption that $E$ is sufficiently small relative to the minimal separation of the planes in $\boldsymbol\pi$, which is the analogue of \cite[Proposition 5.2]{DMS}.

\begin{proposition}\label{p:tilt-2}
For every pair of scales $1\leq r<r_0 < 2$, there are constants $\bar{C}=\bar{C}(Q,m,n,N,r_0-r,2-r_0)>0$ and $\sigma_4=\sigma_4 (Q,m,n,N,r_0-r,2-r_0)>0$ with the following properties.
Let $T$ and $\omega$ be as in Assumption \ref{a:height-main} and let $\boldsymbol\pi$, $E$, $H$ be as in Lemma \ref{l:simpler}. If in addition we have
\begin{equation}\label{e:very-small}
E \leq \sigma_4 \min\{H^2, 1\} \, ,
\end{equation}
then
\begin{equation}\label{e:tilt-estimate-2}
\mathbf{E} (T, \mathbf{C}_r) \leq \bar{C} (E + \Omega^2) + \bar{C} \left(\frac{E}{H^2}\right) \mathbf{E} (T, \mathbf{C}_{r_0})\, .
\end{equation}
\end{proposition}

Observe that the proof of Proposition \ref{p:tilt-2} remains unchanged, given Lemma \ref{l:tilt-1}, Proposition \ref{p:step1} and the semicalibrated analogues of the results from \cite[Section 4.4]{DMS} (recall the discussion above regarding the latter). 

\subsection{Tilt excess estimate}
Given Proposition \ref{p:tilt-2} and the combinatorial lemmas of \cite[Section 4.3]{DMS}, the tilt excess estimate \eqref{e:tilt-estimate} of Theorem \ref{thm:main-estimate} follows exactly as in \cite[Section 5.3]{DMS}. 

\subsection{$L^2-L^\infty$ height bound}
Following \cite{DMS}, the proof of the $L^\infty$-bound \eqref{e:Linfty-estimate} is proven by induction on $N$, relying on the validity of the tilt excess estimate \eqref{e:tilt-estimate} that we have just established.

\begin{proposition}\label{p:induction}
    Let $N\geq 2$ and suppose that \eqref{e:Linfty-estimate} holds for any $N' \leq N - 1$ and any $Q' \leq Q$. Then it holds for $N$ and $Q$.
\end{proposition}

The key to the proof of Proposition \ref{p:induction} is the following adaptation of \cite[Lemma 6.2]{DMS} to the semicalibrated setting.

\begin{lemma}\label{l:decay}
There are constants $\rho_0 = \rho_0 (m,n,Q)>0$ and $C = C(Q,m,n)>0$ such that, for every fixed $0<\rho\leq \rho_0$, there are
$\sigma_5 = \sigma_5 (Q,m,n,N, \rho) \in (0,1]$ and $0 < \beta_0 = \beta_0(Q,m) < 1$ such that the following holds. 
Assume $T$, $E$, and $\boldsymbol{\pi}$ are as in Theorem \ref{thm:main-estimate} with $P= \{p_1, \ldots, p_N\}$ and that
\begin{equation}\label{e:very-small-again}
E + \Omega^{2 - 2 \delta_3} \leq \sigma_5.
\end{equation} 
Then there is another set of points $P':= \{q_1, \ldots, q_{N'}\}$ with $N'\leq Q$ such that:
\begin{itemize}
\item[(A)] $\dist (q_i, P) \leq C (E+ \Omega^{2-2\delta_3})^{1/2}$ for each $i$;
\item[(B)] If we set $\boldsymbol{\pi}':= \bigcup (q_i+\pi_0)$, then
\begin{equation}\label{e:decay-estimate}
\int_{\mathbf{C}_{2\rho}} \dist^2 (x, \boldsymbol{\pi}')\, d\|T\| (x) 
\leq \rho^{m+2\beta_0} (E+ \Omega^{2-2\delta_3}).
\end{equation}
\end{itemize}
\end{lemma}

Given Lemma \ref{l:decay}, Proposition \ref{p:step1}, the combinatorial lemma \cite[Lemma 4.5]{DMS} and \cite[Lemma 4.9]{DMS}, the proof of Proposition \ref{p:induction} follows exactly by exactly the same reasoning as that in \cite{DMS}, so we omit this concluding argument here. Observe that the semicalibrated analogue of \cite[Lemma 4.9]{DMS} follows by making the same minor modifications to the proof as those in Lemma \ref{l:simpler}; see the discussion there for more details. 

The main difference in the proof of Lemma \ref{l:decay} relative to its counterpart in the area-minimizing setting is the application of \cite[Theorem 3.1]{DLSS1}, which yields a nearby harmonic approximation for a given strong Lipschitz approximation to $T$ as given by \cite[Theorem 1.4]{DLSS1}. To apply the former, instead of checking the hypothesis $\Abf \leq \Ebf^{1/4 + \bar{\delta}}$ (as is done in the area-minimizing case), we need to make sure that $\Omega \leq \eps_{23} \Ebf^{1/2}$, for the geometric constant $\eps_{23}$ therein (with $\eta_1$ fixed appropriately). We thus need to amend the case analysis within the proof of Lemma \ref{l:decay} accordingly, and so we provide an outline of the proof here, for the benefit of the reader.

\begin{proof}[Proof of Lemma \ref{l:decay}]
As in \cite[Lemma 6.2]{DMS}, we divide the proof in two cases, only here it will be based on the relative sizes of $\Ebf := \Ebf(T,\Cbf_1)$ and $\Omega^2$. Note that now we have the validity of the tilt-excess estimate \eqref{e:tilt-estimate}, i.e. $\Ebf \leq C(E + \Omega^{2})$. Thus, for $\sigma_5$ small enough (depending on $Q,m,n$), we may assume that $\Ebf < \eps_{21}$, where $\eps_{21}$ is as in \cite[Theorem 1.4]{DLSS1}, allowing us to obtain a map $f: B_{1/4}(0,\pi_0) \to \Acal_Q(\pi_0^\perp)$ satisfying
\begin{itemize}
\item[(i)] ${\rm Lip}\, (f) \leq C \mathbf{E}^{\beta} \leq C(E+\Omega^2)^{\beta}$;
\item[(ii)] There is a closed set $K\subset B_{1/4}$ with $\Hcal^m(K) \leq \frac{1}{2}\Hcal^m(B_{1/4})$ such that $\mathbf{G}_f \res (K\times \mathbb R^n)=T\res (K\times \mathbb R^n)$ and 
\[
\|T\| ((B_{1/4}\setminus K)\times \mathbb R^n) \leq C (\mathbf{E}+\Omega^2)^{1+\beta} \leq C(E+\Omega^2)^{1+\beta} \, ;
\]
\end{itemize}
where $C = C(Q,m,n)$, $\beta=\beta(Q,m,n) \in (0, \tfrac{1}{2m})$ and we have used \eqref{e:tilt-estimate} to obtain the estimates in terms of $E$. This in turn yields
\begin{align}
    \int_{B_{1/4}} |Df|^2 &\leq C \int_K |Df|^2 + C(\Ebf+\Omega^2)^{1+\beta} \label{e:Dir-energy-f} \\
    &\leq C \Ebf + C(E+\Omega^2)^{1+\beta} \leq C(E + \Omega^2)\, .
\end{align}

Fix $\rho_0 \in (0,\tfrac{1}{8})$ to be determined at the end of Case 1 below, and fix $\rho \in (0,\rho_0]$ arbitrarily. Fix $\eta_1$, also to be determined at the end of Case 1 (dependent on $Q,m,n,\rho$). Let $\eps_{23}$ denote the parameter of \cite[Theorem 3.1]{DLSS1}, applied in $\Cbf_1$ with this $\eta_1$ and the map $f$ above taken to be the $E^{\beta}$-approximation therein. Note that, in particular, $\eps_{23}$ depends on $\eta_1$ and thus on $\rho$. We may take $\sigma_5$ even smaller such that that $\Ebf < \eps_{23}$ (now additionally depending on $\rho$).

\emph{Case 1: $\Omega^2 \leq \eps_{23}(\rho) \Ebf$.}

Applying \cite[Theorem 3.1]{DLSS1} as mentioned above, we obtain a Dir-minimizer $g:B_{1/4}(0,\pi_0)\to \Acal_Q(\pi_0^\perp)$ satisfying
\begin{align*}
    \int_{B_{1/4}} \Gcal(f,g)^2 &\leq \eta_1 \omega_m \Ebf; \\
    \int_{B_{1/4}} |Dg|^2 &\leq C(E+\Omega)^2\, .
\end{align*}

We then proceed as in \cite[Proof of Lemma 6.2, Case 1]{DMS}, letting $q_i:=g_i(0)$, where $g_i$ are the distinct functions in a selection for $g$ (possibly with multiplicities) and propagating the decay coming from the $\alpha$-H\"{o}lder regularity of this Dir-minimizer $g$ to $T$. This choice of $q_i$ satisfy the desired estimate (A), and we further arrive at the final estimate
\[
    \int_{\Cbf_{2\rho}} \dist^2(x,\boldsymbol\pi')\, d\|T\|(x) \leq C\sigma_5^{\beta}(E+\Omega^2) + C\eta_1(E+\Omega^2) + C\rho^{m+2\alpha}(E+\Omega^2)\, ,
\]
where $C=C(Q,m,n)$. Note that $\alpha=\alpha(Q,m)$ and let $\beta_0 = \frac{\alpha}{2}$. Now choose $\rho_0 \leq (3C)^{\frac{1}{\alpha}}$ and $\eta_1 \leq \frac{\rho^{m+2\beta_0}}{3C}$. Given this choice of $\eta_1$, we may now further take $\sigma_5 \leq \left(\frac{\rho^{m+2\beta_0}}{3C}\right)^{\frac{1}{\beta}}$. This yields the estimate (B), therefore completing the proof in this case. Note that we are now fixing this choice of $\eta_1(\rho)$, and therefore also fixing $\eps_{23}(\rho)$, throughout this proof.

\emph{Case 2: $\Omega^2 > \eps_{23}(\rho)\Ebf$}

In this case, we use the estimate \eqref{e:Dir-energy-f} combined with the assumption, to obtain
\begin{equation*}
        \int_{B_{1/4}} \vert Df \vert^2 \leq C\Ebf + C(\Ebf + \Omega^2)^{1 + \beta} \leq C(\Ebf + \Omega^{2 + 2\beta}) \leq C\eps_{23}^{-2}\sigma_5^{2\delta_3}\Omega^{2-2\delta_3}\,.
\end{equation*}
Combining this with the Poincaré inequality for $Q$-valued functions, and H\"older's inequality, we infer 
\begin{equation}
    \int_{B_{1/4}} \mathcal{G}(f, Y) \leq C\eps_{23}^{-2}\sigma_5^{2\delta_3}\Omega^{2-2\delta_3}\, ,\label{e:Poincare-f}
\end{equation}
for some point $Y = \sum_i Q_i \llbracket q_i \rrbracket \in \Acal_Q$, where the $q_i$ are distinct. Setting $\boldsymbol{\pi}':= \bigcup_i (q_i+\pi_0)$ and combining \eqref{e:Poincare-f} with (ii) and \cite[Lemma 4.8]{DMS} (which, as previously remarked, remains valid here, since merely almost-monotonicity of mass ratios in place of monotonicity suffices) we thus have
\[
\int_{\mathbf{C}_{2\rho}} \dist^2 (x, \boldsymbol{\pi}')\, d\|T\| (x) 
\leq C \eps_{23}^{-2}\sigma_5^{2\delta_3}\Omega^{2-2\delta_3} + \|T\| ((B_{1/4}\setminus K)\times \mathbb R^n)
\leq C\eps_{23}^{-2}\sigma_5^{2\delta_3}\Omega^{2-2\delta_3} ,
\]
where $K\subset B_{1/4}$ is the closed set over which $T$ over which $T$ is graphical, as in \cite[Theorem 1.4]{DLSS1}. In particular, for $\beta_0$ fixed as in Case 1, we may choose $\sigma_5\leq \left(\frac{\eps_{23}^2\rho^{m+2\beta_0}}{C}\right)^{\frac{1}{2\delta_3}}$.
\[
\int_{\mathbf{C}_{2\rho}} \dist^2 (x, \boldsymbol{\pi}')\, d\|T\| (x) \leq \rho^{m+2\beta_0} \Omega^{2-2\delta_3}\, .
\]
This proves conclusion (B) of the lemma, in this case. Given \eqref{e:Poincare-f}, the proof of (A) in this regime follows in the same way as in \cite[Lemma 6.2]{DMS}. 
\end{proof}

\begin{remark}\label{r:error-loss}
    Note that Lemma \ref{l:decay} is the reason behind the fact that we have $\Omega^{2-2\delta_3}$ in our error estimates in Theorem \ref{thm:main-estimate}, and thus throughout the majority of Part 3, rather than $\Omega^2$. Indeed, notice that Case 2 in the proof above requires the tilt excess of $T$ to be sufficiently small relative to the relevant power of $\Omega$; this cannot be ensured if such a power is quadratic. However, at the cost of decreasing this power slightly, we obtain the desired conclusion.
\end{remark}

\section{Graphical approximations}\label{s:graphs}
Given Theorem \ref{thm:main-estimate}, the graphical approximation results of \cite[Section 8]{DMS} relative to a balanced cone follow immediately, after merely replacing any application of Almgren's strong Lipschitz approximation \cite[Theorem 1.4]{DLS14Lp} with its semicalibrated variant \cite[Theorem 1.4]{DLSS1}. We provide the main conclusions here, for clarity.

Let us begin by recalling the notions of \emph{Morgan angles} and \emph{$M$-balanced cones}, which are key for the results in this section.

\begin{definition}\label{d:frank}
Given two $m$-dimensional linear subspaces $\alpha, \beta$ of $\mathbb R^{m+n}$ whose intersection has dimension $m-2$, we consider the two positive eigenvalues $\lambda_1 \leq \lambda_2$ of the quadratic form $Q_1: \alpha \to \R$ given by $Q_1(v) := \dist^2(v,\beta)$. The \textit{Morgan angles} of the pair $\alpha$ and $\beta$ are the numbers $\theta_i (\alpha, \beta) := \arcsin \sqrt{\lambda_i}$ for $i=1,2$.

Let $M\geq 1$, $N\in \N$. We say that $\Sbf = \alpha_1\cup\cdots\cup\alpha_N\in \Cscr(Q)$ is \emph{$M$-balanced} if for every $i\neq j$, the inequality
\begin{equation}\label{e:balanced}
\theta_2 (\alpha_i, \alpha_j) \leq M \theta_1 (\alpha_i, \alpha_j)
\end{equation}
holds for the two Morgan angles of the pair $\alpha_i, \alpha_j$.
\end{definition}
\begin{remark}
    As communicated to us by F. Morgan, the notion of Morgan angles dates back to Cimille \cite{Cimille}, cf. \cite[4.1]{MorganGMT} for further details. 
\end{remark}
Furthermore, for $\Sbf = \alpha_1\cup\cdots\cup\alpha_N\in \Cscr(Q)$, recall the following notation for the minimal separation between the planes in $\Sbf$:
\begin{equation}\label{e:min-sep}
    \boldsymbol{\sigma}(\Sbf) := \min_{1 \leq i<j \leq N}\dist(\alpha_i\cap \Bbf_1,\alpha_j\cap\Bbf_1)\, .
\end{equation}
We additionally recall the layering subdivision of \cite[Lemma 8.3]{DMS}. More precisely, we apply \cite[Lemma 8.3]{DMS} with a parameter $\bar{\delta}$, to be fixed in Assumption \ref{a:refined} below, in place of $\delta$ therein. This yields a family of sub-cones $\mathbf{S} = \mathbf{S}_0 \supsetneq \mathbf{S}_1 \supsetneq \cdots \supsetneq \mathbf{S}_{\kappa}$ where $\mathbf{S}_k$ consists of the union of the planes $\alpha_i$ with $i\in I (k)$ for the set of indices $I(k)$ therein. We then distinguish two cases:
\begin{itemize}
\item[(a)] if $\max_{i<j\in I (\kappa)} \dist (\alpha_i \cap \Bbf_1, \alpha_j \cap \Bbf_1) < \bar \delta$, we define an additional cone $\mathbf{S}_{\kappa+1}$ consisting of a single plane, given by the smallest index in $I (\kappa)$ and we set $\bar\kappa := \kappa+1$ and $I(\bar\kappa):= \{\min I(\kappa)\}$;
\item[(b)] otherwise, we select no smaller cone and set $\bar \kappa := \kappa$.
\end{itemize}

In this section, we work with the following underlying assumption.

\begin{assumption}\label{a:refined}
Suppose $T$ and $\omega$ are as in Assumption \ref{a:main} and $\|T\| (\Bbf_4) \leq 4^m (Q+\frac{1}{2}) \omega_m$. Suppose $\mathbf{S}=\alpha_1\cup \cdots \cup \alpha_N \in \mathscr{C} (Q)\setminus \mathscr{P}$ is $M$-balanced, where $M\geq 1$ is a given fixed constant. 

Firstly, we denote by $\delta^*$ the minimum of the parameters $\delta$ needed to ensure that the semicalibrated versions (in ambient space $\R^{m+n}$) of \cite[Lemma 8.5, Proposition 8.6, Lemma 8.7, Proposition 8.8]{DMS} are applicable to all the cones $\mathbf{S}_k${, $k\in \{0,1,\dotsc,\bar{\kappa}\}$;} note that all the $\Sbf_k$ are $M$-balanced by construction and that therefore $\delta^* = \delta^*(m,n,Q,M)>0$. Subsequently, we fix a parameter $\tau = \tau(m,n,Q,M)>0$ smaller than $c \delta^*$ for the small constant $c = c(m,n,Q)>0$ determined by \cite[Lemma 8.13]{DMS} (with $\bar n = n$ and $\Sigma = \R^{m+n}$); note that this constant remains completely unchanged in the setting herein. 

We then fix the parameter $\bar\delta$ {smaller than $c \tau$ for this same constant $c$}; so $\bar\delta = \bar\delta(m,n,Q,M)>0$. In particular, $\bar\delta \leq \delta^*$. Finally, $\varepsilon = \varepsilon(m,n,Q,\delta^*,\bar\delta,\tau)>0$ is determined in {Proposition \ref{p:first-blow-up} below, and will be smaller than both {$c \bar\delta$} for the same parameter $c$ above, and the parameter $\eps$ of \cite[Lemma 8.7]{DMS}. Note that $\eps$ is implicitly additionally dependent on the two parameters $\sigma,\varsigma$ of Proposition \ref{p:first-blow-up}, which are fixed arbitrarily.} We assume that $\{\Theta (T,\cdot) \geq Q\} \cap \Bbf_\varepsilon (0) \neq \emptyset$ and suppose that
\begin{equation}\label{e:smallness-alg}
\mathbb{E} (T, \mathbf{S}, \Bbf_4) + \Omega^{2-2\delta_3} \leq \varepsilon^2 \boldsymbol{\sigma} (\Sbf)^2\, ,
\end{equation}
where $\mathbb{E}(T,\Sbf,\Bbf_4)$ is defined as in Definition \ref{def:L2_height_excess}.
\end{assumption}

\subsection{Whitney decomposition}
We recall here the main aspects of the Whitney decomposition of \cite[Section 8.5.1]{DMS} and the associated notation. We refer the reader therein for more details, including figures illustrating the decomposition and associated regions in the ambient space. 

Let $\Sbf=\alpha_1\cup\cdots\cup\alpha_N\in \Cscr(S)$. Let $L_0$ be the closed cube in $V=V(\Sbf)$ with side-length $\frac{2}{\sqrt{m-2}}$ centered at $0$ and let $R$ be the rotationally invariant (around $V$) region given by
\begin{equation}\label{e:region-R}
R :=\{p: \mathbf{p}_V (p) \in L_0\, \text{ and }\, 0<|\mathbf{p}_{V^\perp} (p)|\leq 1\}\, .
\end{equation}
We recall here that we are assuming $m\geq 3$ (cf. Assumption \ref{a:main}); note that this is the only reason why such a restriction is necessary. 

For every $\ell\in \mathbb N$ denote by $\mathcal{G}_\ell$ the collection of $(m-2)$-dimensional cubes in the spine $V$ obtained by subdividing $L_0$ into $2^{\ell (m-2)}$ cubes of side-length $\frac{2^{1-\ell}}{\sqrt{m-2}}$, and we let $\mathcal{G} = \bigcup_\ell \mathcal{G}_\ell$. We write $L$ for a cube in $\mathcal{G}$, so $L\in \mathcal{G}_\ell$ for some $\ell\in \N$. When we want to emphasize the dependence of the integer $\ell$ on $L$ we will write $\ell (L)$. We use the standard terminology \emph{parent}, \emph{child}, \emph{ancestor}, \emph{descendant} to describe relations of cubes; see \cite{DMS} for details. For every $L\in \mathcal{G}_\ell$ we let 
\[
R (L) := \{p: \mathbf{p}_V (p)\in L \quad \mbox{and} \quad 2^{-\ell-1} \leq |\mathbf{p}_{V^\perp} (p)|\leq 2^{-\ell}\}\, .
\]
For each $L\in \mathcal{G}_\ell$ we let $y_L\in V$ be its center and denote by $\Bbf (L)$ the ball $\mathbf{B}_{2^{2-\ell(L)}} (y_L)$ {(in $\R^{m+n}$)}
and by $\mathbf{B}^h (L)$ the set $\Bbf (L) \setminus B_{\rho_*2^{-\ell(L)}} (V)$, where $\rho_*$ is as in \cite[Lemma 8.7]{DMS}. We identify three mutually dijoint subfamilies of cubes in $\Gcal$; \emph{outer cubes}, \emph{central cubes} and \emph{inner cubes}, defined precisely in \cite[Definition 8.12]{DMS}. These families of cubes will be denoted by $\mathcal{G}^o$, $\mathcal{G}^c$, and $\mathcal{G}^{in}$, respectively. By construction, any cube $L\in \mathcal{G}$ is either an outer cube, or a central cube, or an inner cube, or a descendant of inner cube. We in turn define three subregions of $R$:
\begin{itemize}
    \item The \emph{outer region}, denoted $R^o$, is the union of $R (L)$ for $L$ varying over elements of $\mathcal{G}^o$.
    \item The \emph{central region}, denoted $R^c$, is the union of $R (L)$ for $L$ varying over elements of $\mathcal{G}^c$.
    \item Finally, the \emph{inner region}, denoted $R^{in}$, is the union of $R(L)$ for $L$ ranging over the elements of $\mathcal{G}$ which are neither outer nor central cubes.
\end{itemize}

For every $i\in \{1, \ldots, N\}$ we further define 
\[
R^o_i := \bigcup_{L\in \mathcal{G}^o} L_i \equiv \alpha_i\cap\bigcup_{L\in \mathcal{G}^o} R (L)\, 
\]
and let $Q_i := Q_{L_0, i}$. 

We refer the reader to \cite[Lemma 8.10, Lemma 8.13]{DMS} for key properties about the Whitney cubes and the conical excess of $T$ associated to each of them; the conclusions remain unchanged herein. 

Observe that the results of \cite[Section 8.5.4]{DMS} remain valid here also, with all instances of $\Abf$ replaced by $\Omega^{1-\delta_3}$ in the estimates, given the conclusions of Theorem \ref{thm:main-estimate}. In particular, note that the choice of $a(Q,m)$ is determined by \cite[Lemma 8.7]{DMS}, with the proof and the constant $\rho_*(Q,m)$ remaining unchanged herein; namely, $a(Q,m) = \tfrac{\rho_*}{4}$. Indeed, the compactness procedure therein still yields a limiting area-minimizing current in this setting, since $\Omega_k$ converges to zero along the sequence, and thus we may proceed to exploit the monotonicity of mass ratios in the same way. Since such a compactness will be exploited numerous times in the following sections, we elaborate on it in the following remark, which we will refer back to.

\begin{remark} \label{r:contradictio-argument-limit-area-min}
    We make a note on a procedure that will be often used in the rest of the article. When proving some statements, for instance the \textit{crude splitting} lemma building up to Proposition \ref{p:first-blow-up} above and Proposition \ref{p:balancing-2} below, we argue by contradiction. In particular, we consider a sequence of currents $T_k$, semicalibrated by forms $\omega_k$, and extract a converging subsequence limiting to a certain $T_\infty$, (a valid procedure under our standing hypothesis, e.g. mass bounds on the $T_k$'s). As written, we have no hope to obtain further information on the limit $T_\infty$. However, in all the statement we will prove, we will also have the corresponding $\Omega_k$ converging to zero, whence allowing us to deduce that $T_\infty$ is area-minimizing, and proceed as in the relevant proofs of \cite{DMS}. 
\end{remark}

The final conclusion of the results in \cite[Section 8]{DMS}, rewritten for a semicalibrated current $T$ is the following, which is the analogue of \cite[Proposition 8.18 \& Proposition 8.19]{DMS}.

\begin{proposition}[Coherent outer approximation and final blowup]\label{p:first-blow-up}
Let $T$, $\omega$ and $\mathbf{S}=\alpha_1\cup\cdots\cup\alpha_N$ be as in Assumption \ref{a:refined}. Then, for every $\sigma, \varsigma>0$ there are constants $C = C(m,n,Q,\delta^*, \tau, \bar\delta)>0$ and $\varepsilon = \varepsilon(m,n,Q,\delta^*, \tau, \bar\delta, \sigma, \varsigma)>0$ such that the following properties hold.

\begin{itemize}
    \item[(i)] $R\setminus B_\sigma (V)$ is contained in the outer region $R^o$.
    \item[(ii)] There are Lipschitz multi-valued maps $u_i : R^o_i \to \mathcal{A}_{Q_i} (\alpha_i^\perp)$ and closed subsets $\bar{K}_i (L)\subset L_i$ satisfying 
    \[
        T_{L,i} \res \mathbf{p}_{\alpha_i}^{-1} (\bar{K}_i(L))= \mathbf{G}_{u_i}\res \mathbf{p}_{\alpha_i}^{-1} (\bar{K}_i (L)) \qquad \forall L\in \mathcal{G}^o\, ,
    \]
    as well as the estimates \cite[Proposition 8.18, (8.49)-(8.51)]{DMS}, and
    \begin{equation}\label{e:first-Dir-control}
        \int_{R_i} |Du_i|^2 \leq C \sigma^{-2} \hat{\mathbf{E}} (T, \mathbf{S}, \Bbf_4) + C \Omega^{2-2\delta_3} \, ,
    \end{equation}
    for $R_i := (R\setminus B_\sigma(V))\cap \alpha_i$.
    \item[(iii)] If additionally $\Omega^{2-2\delta_3} \leq \varepsilon^2 \hat{\mathbf{E}}(T, \mathbf{S}, \Bbf_4)$ and we set $v_i := \hat{\mathbf{E}} (T, \mathbf{S}, \Bbf_4)^{-1/2} u_i$, then there is a map $w_i: R_i \to \mathcal{A}_{Q_i} (\alpha_i^\perp)$ which is Dir-minimizing and such that 
    \begin{equation}\label{e:first-blowup}
    d_{W^{1,2}} (v_i, w_i) \leq \varsigma\, ,
    \end{equation}
    where $d_{W^{1,2}}$ is the $W^{1,2}$ distance between $Q$-valued maps; see for instance \cite{DLS_MAMS} for a definition. 
\end{itemize}
\end{proposition}

Note that in order to obtain the conclusion (iii) of Proposition \ref{p:first-blow-up}, we apply \cite[Theorem 3.1]{DLSS1} in place of \cite[Theorem 2.6]{DLS14Lp}, which is used to establish this conclusion when $T$ is area-minimizing. 

\section{Cone balancing}\label{s:balancing}
In this section, we observe that all of the results of \cite[Section 9]{DMS} remain valid in the case when $T$ is semicalibrated, with all instances of $\Abf$ replaced with $\Omega^{1- \delta_3}$. This is again due to an application of the height bound from Lemma \ref{l:decay} when proving Proposition \ref{p:balancing} below assuming Proposition \ref{p:balancing-2}. For the convenience of the reader, and since it will be useful for the succeeding sections, we provide the statements here. First of all, we recall the minimal separation $\boldsymbol\sigma(\Sbf)$ (see \eqref{e:min-sep}) between the planes within a given cone $\Sbf = \alpha_1\cup\cdots\cup \alpha_N$ as introduced in the preceding section, as well as the maximal separation
\begin{align*}
    \boldsymbol{\mu}(\Sbf) &:= \max_{1\leq i<j \leq N}\dist(\alpha_i\cap \Bbf_1,\alpha_j\cap\Bbf_1) \, .
\end{align*}

\begin{proposition}[Cone balancing]\label{p:balancing}
Assume that $T$ and $\omega$ are as in Assumption \ref{a:main}, $\mathbf{S} = \alpha_1\cup \cdots \cup \alpha_N \in \Cscr(Q)$. Then, there are constants $M = M(Q,m,n)>0$ and $\varepsilon_0 = \varepsilon_0(Q,m,n)>0$ with the following property. Assume that
\begin{equation}\label{e:smallexcess}
    \Omega^{2-2\delta_3} \leq \varepsilon_0^2 \mathbb{E} (T, \mathbf{S}, \Bbf_1) \leq \eps_0^4 \Ebf^p(T,\Bbf_1)\, .
\end{equation}
 Then there is a subset $\{i_1, \ldots , i_k\}\subset \{1, \ldots , N\}$ with $k\geq 2$ such that, upon setting $\mathbf{S}' = \alpha_{i_1} \cup \cdots \cup \alpha_{i_k}$, the following holds:
\begin{itemize}
\item[(a)] $\mathbf{S}'$ is $M$-balanced;
\item[(b)] $\mathbb{E} (T, \mathbf{S}', \Bbf_1) \leq M \mathbb{E} (T, \mathbf{S}, \Bbf_1)$;
\item[(c)] $\dist^2(\Sbf\cap \Bbf_1,\Sbf^\prime\cap\Bbf_1) \leq M\mathbb{E}(T,\Sbf,\Bbf_1)$;
\item[(d)] $M^{-1} \Ebf^p (T, \Bbf_1) \leq \boldsymbol{\mu} (\Sbf)^2 =
\boldsymbol{\mu} (\Sbf^\prime)^2\leq M \Ebf^p (T, \Bbf_1)$.
\end{itemize}
\end{proposition}

The proof of Proposition \ref{p:balancing} reduces to the validity of the following proposition, by exactly the same reasoning as that in \cite[Section 9]{DMS}, given that the crude approximation \cite[Proposition 8.6]{DMS} remains valid in this setting with $\Abf$ replaced by $\Omega^{2-2\delta_2}$.

\begin{proposition}\label{p:balancing-2}
Assume that 
$T$, $\omega$, and $\mathbf{S}$ are as in Proposition \ref{p:balancing}. Then there are constants $C = C(m,n,N)$ and $\varepsilon = \varepsilon (m,n,N)$ with the following property. If we additionally have that
\begin{equation}\label{e:evensmallerexcess}
    N\geq 2 \quad \text{and} \quad {\Omega^{2-2\delta_3} + \mathbb{E} (T, \mathbf{S}, \Bbf_1) \leq \varepsilon^2 \boldsymbol{\sigma}(\Sbf)^2\,} ,
\end{equation} 
then $\mathbf{S}$ is $C$-balanced.
\end{proposition}

To prove Proposition \ref{p:balancing-2}, one argues as in \cite[Proof of Proposition 9.2]{DMS}, namely one first proves a version of the proposition when $\boldsymbol{\mu}(\Sbf)$ and $\boldsymbol{\sigma}(\Sbf)$ are comparable (see \cite[Proposition 9.3]{DMS}) and arguing later by induction. The argument follows by contradiction, assuming that the balancing assumption is failing along a sequence of semicalibrated currents $T_k$ with associated semicalibration forms $\omega_k$ and cones $\Sbf_k$. Note that we have $\Omega_k^{2-2\delta_3} \leq \eps_k^2 \boldsymbol\sigma(\Sbf_k)$ with $\eps_k \downarrow 0$. Thus, we still obtain a limiting current $T_\infty$ that is area-minimizing (see Remark \ref{r:contradictio-argument-limit-area-min} for a similar argument) and supported in some $\Sbf_\infty \in \Cscr(Q)\setminus \Pscr$ in the non-collapsed case, i.e. when up to subsequence we have $\lim_{k\to \infty}\boldsymbol\sigma(\Sbf_k) > 0$. In light of the main result in \cite{Morgan} (see also \cite[Lemma 7.5]{DMS}), this contradicts the failure of balancing. Similarly, in the collapsed case $\lim_{k\to \infty}\boldsymbol\sigma(\Sbf_k) = 0$, the compactness argument in \cite{DMS} still yields a limiting Dir-minimizer (the domain of which depends on the case analysis therein), and  thus \cite[Proposition 7.6]{DMS} may be applied to contradict the failure of balancing. 

\section{Estimates at the spine}\label{s:nonconc}

This section is dedicated to the nonconcentration estimates for $T$ near the spine of $\Sbf$, much analogous to those in \cite[Section 11]{DMS}, but appearing first in a multiplicity one setting in the seminal work \cite{Simon_cylindrical} of Simon. Our underlying assumption throughout this section will be the following.

\begin{assumption}\label{a:HS}
Suppose $T$ and $\omega$ are as in Assumption \ref{a:main} and $\|T\| (\Bbf_4) \leq 4^m (Q+\frac{1}{2}) \omega_m$. Suppose $\mathbf{S}=\alpha_1\cup \cdots \cup \alpha_N$ is a cone in $\mathscr{C} (Q)\setminus \mathscr{P}$ which is $M$-balanced, where $M\geq 1$ is a given fixed constant, and let $V$ denote the spine of $\mathbf{S}$. For a sufficiently small constant $\eps = \eps(Q,m,n,M)$ smaller than the {$\eps$-}threshold  in Assumption \ref{a:refined}, whose choice will be fixed by the statements of Theorem \ref{t:HS}, Corollary \ref{c:HS-patch}, and Proposition \ref{p:HS-3} below, suppose that
\begin{equation}\label{e:smallness-alg-2}
\mathbb{E} (T, \mathbf{S}, \Bbf_4) + \Omega^{2-2\delta_3} \leq \varepsilon^2 \boldsymbol{\sigma} (\Sbf)^2\, .
\end{equation}
\end{assumption}

\begin{theorem}[Simon's error and gradient estimates]\label{t:HS}
Assume $T$, $\omega$, and $\mathbf{S}$ are as in Assumption \ref{a:HS}, suppose that $\Theta (T, 0) \geq Q$, and set $r= \frac{1}{3\sqrt{m-2}}$. Then there is a constant $C=C(Q,m,n,M)>0$ and a choice of $\varepsilon = \varepsilon(Q,m,n,M)>0$ in Assumption \ref{a:HS} sufficiently small such that
\begin{align}
\int_{\Bbf_r} \frac{|q^\perp|^2}{|q|^{m+2}}\, d\|T\| (q)
&\leq C (\Omega^{2-2\delta_3} + {\hat{\Ebf} (T, \mathbf{S}, \Bbf_4))}\, \label{e:HS}\\
\int_{\Bbf_r} |\mathbf{p}_V\circ \mathbf{p}_{\vec{T}}^\perp|^2\, d\|T\| &\leq C (\Omega^{2-2\delta_3} + {\hat{\Ebf} (T, \mathbf{S}, \Bbf_4))}\, .\label{e:HS-additional}
\end{align}
\end{theorem}

 This is the analogue of \cite[Theorem 11.2]{DMS}. However, note that in order to get the quadratic $\Omega$ error improvement in the estimates \eqref{e:HS} and \eqref{e:HS-additional}, we crucially exploit the nature of the error in the first variation of $T$, and we recall the generalized mean curvature estimate \eqref{e:H-Linfty-est}. This is very much analogous to the observation made in \cite[Remark 1.10]{Spolaor_15} in the case when $T$ is area-minimizing, but since the main term on the right-hand side of \eqref{e:HS} and \eqref{e:HS-additional} is in terms of the conical excess of $T$ rather than the tilt excess, we provide a sketch proof, for the purpose of clarity.
 
\begin{proof}
    For any $0 < s < \rho \leq 4$, we may test the first variation identity \eqref{eq: first variation} for $T$ with the radial vector field $\eta\big(\tfrac{\vert p \vert}{\tau}\big) p$, where $\tau \in [s, \rho]$ and $\eta$ is a smooth cut-off function that we take to converge to the characteristic function on $[0,1]$, which yields the classical monotonicity formula estimate (see e.g. \cite[Theorem 3.8]{DL-survey-JDG}). Now for any $0< r < R <4$, let $\chi\in C^\infty([0, \infty); \mathbb R)$ be monotone non-increasing with $\chi \equiv 1$ on $[0,r]$ and $\chi \equiv 0$ on $[{R}, \infty)$. Taking $s\downarrow 0$, exploiting the hypothesis $\Theta(T,0) \geq Q$, multiplying by $\rho^m$ and differentiating in $\rho$, then multiplying by $\chi (\rho)^2$ and integrating over $\rho \in[0,R]$, we obtain the estimate
\begin{align}
    \int_{\Bbf_r} \frac{|q^\perp|^2}{|q|^{m+2}}\, d\|T\| (q)
    &\leq C\left[\int \chi^2 (|q|)\, d\|T\| (q) - \sum_i Q_i \int_{\alpha_i} \chi^2 (|q|)\, d\mathcal{H}^m (q)\right]\nonumber\\
    &\qquad + {C \int_{\Bbf_{R}} \frac{|\Gamma (|q|)q^\perp\cdot \vec{H}_T (q)|}{|q|^m}\,  d\|T\| (q)}\, ,\label{e:monot-10}
\end{align}
where $C=C(m,r)>0$ and $\Gamma$ is a suitable locally bounded function with $\|\Gamma\|_{L^\infty} \leq \bar C (R^m-r^m)$ for $\bar C = \bar C(m)$. See \cite{DMS} for a more detailed computation. Combining this with the elementary identity $ab \leq \tfrac{\delta a^2}{2} + \tfrac{b^2}{2\delta}$ for $\delta>0$ (to be determined) and the estimate \eqref{e:H-Linfty-est}, the last term on the right-hand side can then be estimated as follows:
    \begin{align}
        \int_{\mathbf{B}_{R}} \vert \Gamma(|q|) q^\perp \cdot \vec{H}_T(q) \vert \frac{1}{\vert q \vert^m} \, d\Vert T \Vert(q) &\leq \frac{\delta}{2} \int_{\mathbf{B}_{R}} \frac{|{q^\perp}|^2}{\vert q \vert^{m+2}} \, d\Vert T \Vert(q) \notag \\
        &\qquad+ C(\delta,m) r^m \|\vec H_T\|_{L^\infty}^2 \int_{\mathbf{B}_{R}}\frac{1}{\vert q \vert^{m-2}} \, d\Vert T \Vert(q) \notag \\
        &\leq \frac{\delta}{2} \int_{\mathbf{B}_{R}} \frac{|q^\perp|^2}{\vert q \vert^{m+2}} \, d\Vert T \Vert(q) + C(\delta,m) \Omega^2\, .\label{e:monotonicity-error-est}
    \end{align}
Inserting this into \eqref{e:monot-10}, we obtain
\begin{align}
    \int_{\Bbf_r} \frac{|q^\perp|^2}{|q|^{m+2}}\, d\|T\| (q)
    &\leq C\underbrace{\left[\int \chi^2 (|q|)\, d\|T\| (q) - \sum_i Q_i \int_{\alpha_i} \chi^2 (|q|)\, d\mathcal{H}^m (q)\right]}_{(I)} \notag \\
    &\qquad+ C\delta \int_{\mathbf{B}_{R}} \frac{|q^\perp|^2}{\vert q \vert^{m+2}} \, d\Vert T \Vert(q) + C(\delta,m) \Omega^2\, .\label{e:HS-first-est}
\end{align}
Note that this is of the form \cite[Proof of Theorem 11.2, (11.15)]{DMS} but with $\Abf^2$ replaced with $\Omega^2$, and the extra term on the right-hand side, which we will eventually absorb into the left-hand side, but due to the domain being larger, this is slightly delicate and needs to be done at the end. The remainder of the proof therein deals with estimating the term (I) above. Observe that the estimate \cite[(11.24)]{DMS} may in fact be more precisely written as
\begin{align}
\int \chi^2(|q|)\, d\|T\|&(q) - \sum_i Q_i \int_{\alpha_i} \chi^2(|q|)\, d\Hcal^m(q) \nonumber \\
&\leq C \int {\rm div}_{\vec{T}} X (q)\, d\|T\| (q) + C \int_{\Bbf_{R}} |x^\perp|^2 d\|T\| + \int \chi (|q|) x \cdot \nabla_{V^\perp} \chi (|q|)\, d\|S\|(q) \notag \\
&\qquad - \int \chi (|q|)\, \mathbf{p}_{\vec{T}} (x)\cdot \nabla_{V^\perp} \chi (|q|)\, d\|T\| (q)
\, ,\label{e:first-variation-5}
\end{align}
for the vector field $X (q) = \chi (|q|)^2 \mathbf{p}_{V}^\perp (q)$ and $S=\spt(\Sbf)$. Thus, in order to follow the arguments of \cite{DMS} verbatim, it merely remains to verify that
\begin{equation}\label{e:first-variation-2}
\int \diverg_{\vec{T}} X (q)\, d\|T\| (q) \leq \tilde\delta\int_{\Bbf_{R}} \frac{|q^\perp|^2}{|q|^{m+2}}\, d\|T\|(q) + C(\tilde\delta) \Omega^2\, ,
\end{equation}
for $\tilde\delta>0$ small enough such that, after combining \eqref{e:first-variation-5} with \eqref{e:HS-first-est}, the first term on the right-hand side of \eqref{e:first-variation-5} may be reabsorbed into the left-hand side of \eqref{e:HS-first-est}.

To see that \eqref{e:first-variation-2} holds, we use the first variation of $T$ to write
\begin{equation}\label{e:first-var-HS}
    \int {\rm div}_{\vec{T}}\, X \, d\|T\| = - \int X^\perp \cdot \vec{H}_T \, d\|T\| \, ,
\end{equation}
and argue as in \eqref{e:monotonicity-error-est} to obtain the estimate
\begin{equation}\label{e:first-var-error-2-HS}
    \int |X^\perp\cdot \vec{H}_T | \, d\|T\| \leq \tilde\delta \int_{\Bbf_{R}} \frac{|q^\perp|^2}{|q|^{m+2}}\, d\|T\|(q) + C(\tilde\delta) \|\vec H_T \|_{L^\infty}^2 \int_{\Bbf_{R}} |q|^{m+2} \, d\|T\|(q)\, ,
\end{equation}
from which \eqref{e:first-variation-2} follows immediately. We may then use the estimates corresponding to those in \cite[Sections 11.1.2-11.1.4]{DMS}, for which we must use the semicalibrated analogues of the relevant results in \cite[Section 8]{DMS} (with $\Sigma = \R^{m+n}$ and all occurrences of $\Abf$ replaced with $\Omega^{1-\delta_3}$). The majority of such results were omitted in Section \ref{s:graphs} herein for brevity, but nevertheless hold true.

In summary, when combined with \eqref{e:HS-first-est}, we arrive at the final estimate
\begin{align*}
    \int_{\Bbf_r} \frac{|q^\perp|^2}{|q|^{m+2}}\, d\|T\| (q)
    &\leq C\hat{\Ebf}(T,\Sbf,\Bbf_4) + C(\delta,\tilde\delta,m)\Omega^{2-2\delta_3} \\
    &\qquad+ C(\delta +\tilde\delta) \int_{\mathbf{B}_{R}} \frac{|q^\perp|^2}{\vert q \vert^{m+2}} \, d\Vert T \Vert(q)\, .
\end{align*}
Invoking, for instance, \cite[Lemma 4.3]{HanLin}, for $\delta,\tilde\delta$ sufficiently small (depending on $Q,m,n$), we may indeed absorb the final term on the right-hand side above into the left-hand side, concluding the proof.
\end{proof}
\begin{remark}
    In the proof of Theorem \ref{t:HS}, one can alternatively notice that the following identity holds, for any vector field $X\in C_c^\infty(\R^{m+n}\setminus \spt\partial T;\R^{m+n})$:
    \begin{equation} \label{eq: estimate bracket}
        \langle d\omega \mres X, \vec{T} \rangle = \langle d\omega, X \wedge \vec{T} \rangle = \langle d\omega, X^\perp \wedge \vec{T} \rangle, 
    \end{equation}
    where the first equality follows by definition of restriction, and where perpendicularity is with respect to the approximate tangent space of $T$. In particular, this can be used in place of the estimate \eqref{e:H-Linfty-est} when establishing \eqref{e:monotonicity-error-est} and \eqref{e:first-variation-2} above.
\end{remark}

From Theorem \ref{t:HS}, we further deduce the following, which corresponds to \cite[Corollary 11.3]{DMS}.
\begin{corollary}[Simon's non-concentration estimate]\label{c:HS-patch}
Assume $T$, $\omega$, $\mathbf{S}$ and $r$ are as in Theorem \ref{t:HS}. Then, there is a choice of $\varepsilon = \varepsilon(Q,m,n,M)$ in Assumption \ref{a:HS}, possibly smaller than that in Theorem \ref{t:HS}, such that for every $\kappa\in (0,m+2)$, 
\begin{equation}\label{e:HS-patch}
\int_{\Bbf_r} \frac{\dist^2 (q, \mathbf{S})}{|q|^{m+2-\kappa}}\, d\|T\| (q)
\leq C_\kappa (\Omega^{2-2\delta_3} + {\hat{\Ebf} (T, \mathbf{S}, \Bbf_4))}\, ,
\end{equation}
where here $C_\kappa = C_\kappa(Q,m,n,M,\kappa)$. 
\end{corollary}

Corollary \ref{c:HS-patch} follows immediately from the following lemma, combined with Theorem \ref{t:HS}.

\begin{lemma}\label{l:variation-test-2}
Let $T$, $\omega$ and $\Sbf$ be as in Assumption \ref{a:HS} with $\mathbf{B}_1 \subset \Omega$ and $\Theta (T, 0)\geq Q$. Then we may choose $\eps$ sufficiently small in Assumption \ref{a:HS} such that for each $\kappa \in (0,m+2)$ we have
\begin{equation}\label{e:HS-20}
\int_{\Bbf_1} \frac{\dist ^2(q, \mathbf{S})}{|q|^{m+2-\kappa}}\, d\|T\| (q)
\leq C_\kappa \int_{\Bbf_1} \frac{|q^\perp|^2}{|q|^{m+2}}\, d\|T\| (q)+ C_\kappa (\hat{\mathbf{E}} (T, \mathbf{S}, \Bbf_4) + \Omega^2)\, ,
\end{equation}
for some constant $C_\kappa=C_\kappa(Q,m,n,M,\kappa)>0$.
\end{lemma}

\begin{proof}[Proof of Lemma \ref{l:variation-test-2}]
    Fix $\kappa \in (0,m+2)$. As in the proof of Theorem \ref{t:HS}, we aim to follow the reasoning of the area-minimizing counterpart \cite[Lemma 11.6]{DMS} of this lemma, which we may indeed do, provided that we verify
\begin{equation}\label{e:first-variation-21}
    \int {\rm div}_{\vec{T}} X\, \, d\|T\| \leq \delta \int_{\Bbf_{1}} \frac{\dist^2(q,\Sbf)}{|q|^{m+2-\kappa}} \, d\|T\|(q) + C(\delta) \Omega^2  \, ,
\end{equation}
for the vector field
\[
    X (q) := \dist^2 (q,\mathbf{S}) (\max \{r, |q|\}^{-m-2+\kappa} - 1)_+ q\, ,
\]
supported in $\Bbf_1$ (and constant on $\Bbf_r$). To see the validity of \eqref{e:first-variation-21}, we simply exploit the first variation identity \eqref{e:first-var-HS} and an estimate analogous to \eqref{e:first-var-error-2-HS} to obtain
\[
    \int |X^\perp\cdot\vec H_T|\, d\|T\| \leq \delta \int_{\Bbf_1 } \frac{\dist^2(q,\Sbf)}{|q|^{m+2-\kappa}} \, d\|T\|(q) + C(\delta) \Omega^2 \int_{\Bbf_1} \dist^2(q,\Sbf) |q^\perp|^2\, d\|T\|(q) \, .
\]
\end{proof}

Finally we have the following shifted estimates around a given point of high multiplicity (cf. \cite[Proposition 11.4]{DMS}).

\begin{proposition}[Simon's shift inequality]\label{p:HS-3}
Assume $T$, $\omega$, and $\mathbf{S}$ are as in Assumption \ref{a:HS} and in addition $\{\Theta (T, \cdot)\geq Q\}\cap \Bbf_\varepsilon (0) \neq \emptyset$. Then there is a radius $r=r(Q,m,n)$ and a choice of $\varepsilon = \varepsilon(Q,m,n,M)$ in Assumption \ref{a:HS}, possibly smaller than those in Theorem \ref{t:HS} and Corollary \ref{c:HS-patch}, such that for each $\kappa\in (0,m+2)$, there are constants $\bar{C}_\kappa=\bar{C}_\kappa(Q,m,n,M,\kappa)>0$ and $C=C(Q,m,n,M)$ such that the following holds. If $q_0\in \Bbf_r (0)$ and $\Theta (T, q_0)\geq Q$, then
\begin{align}
\int_{\Bbf_{4r} (q_0)} \frac{\dist^2 (q, q_0 + \mathbf{S})}{|q-q_0|^{m+2-\kappa}}\, d\|T\| (q) &\leq
\bar{C}_\kappa (\Omega^{2-2\delta_3} + {\hat{\Ebf} (T, \mathbf{S}, \Bbf_4))}\, .\label{e:HS-3} \\
|\mathbf{p}_{\alpha_1}^\perp (q_0)|^2 + \boldsymbol{\mu} (\mathbf{S})^2 |\mathbf{p}_{V^\perp\cap \alpha_1} (q_0)|^2 &\leq C (\Omega^{2-2\delta_3} + {\hat{\Ebf} (T, \mathbf{S}, \Bbf_4))}\label{e:HS-4}\, .
\end{align}
\end{proposition}

The proof of Proposition \ref{p:HS-3} is entirely analogous to that of \cite[Proposition 11.4]{DMS}, again recalling that the variant of \cite[Proposition 8.14]{DMS} herein has $\Omega^{1-\delta_3}$ in place of each instance of $\Abf$.

\section{Final blow-up and conclusion}\label{s:blowup}
We are now in a position to outline how to conclude the validity of Theorem \ref{c:decay}, given everything in the preceding sections of this part.

We begin with the following weaker conical excess decay result (see \cite[Theorem 10.2]{DMS}), whose validity implies that of Theorem \ref{c:decay}.
\begin{theorem}[Weak Excess Decay Theorem]\label{t:weak-decay}
    Fix $Q,m,n$ as before, and let $M\geq 1$ be as in Proposition \ref{p:balancing}. Fix also $\varsigma_1>0$. Then, there are constants $\varepsilon_1 = \varepsilon_1(Q,m,n,\varsigma_1)\in (0,1/2]$, $r^1_1 = r^1_1(Q,m,n,\varsigma_1)\in (0,1/2]$ and $r^2_1 = r^2_1(Q,m,n,\varsigma_1)\in (0,1/2]$, such that the following holds. Suppose that
    \begin{enumerate}
        \item [\textnormal{(i)}] $T$ and $\Sigma$ are as in Assumption \ref{a:main};
        \item [\textnormal{(ii)}] $\|T\|(\Bbf_1)\leq (Q+\frac{1}{2})\omega_m$;
        \item [\textnormal{(iii)}] There is $\Sbf\in \mathscr{C}(Q)$ which is $M$-balanced, such that
        \begin{equation}\label{e:smallness-weak}
        \mathbb{E}(T,\Sbf,\Bbf_1)\leq \eps_1^2\boldsymbol{\sigma}(\Sbf)^2
        \end{equation}
        and
        \begin{equation}\label{e:no-gaps-weak}
            \Bbf_{\eps_1}(\xi)\cap \{p:\Theta(T,p){\geq\ } Q\}\neq\emptyset \qquad \forall\xi\in V(\Sbf)\cap \Bbf_{1/2}\, ;
        \end{equation}
        \item [\textnormal{(iv)}] $\Omega^{2-2\delta_3} \leq \eps_1^2 \mathbb{E} (T, \tilde{\Sbf}, \Bbf_1)$ for every $\tilde{\Sbf} \in \mathscr{C} (Q)$.
    \end{enumerate}
    Then, there is $\Sbf^\prime\in \mathscr{C}(Q)\setminus\mathscr{P}$ such that for some $i\in\{1,2\}$ we have
    \begin{equation}\label{e:weak-decay}
    \mathbb{E}(T,\Sbf^\prime,\Bbf_{r_1^i}) \leq \varsigma_1\mathbb{E}(T,\Sbf,\Bbf_1)\, .
    \end{equation}
\end{theorem}
To prove that if \ref{t:weak-decay} holds, then Theorem \ref{c:decay} holds, one may proceed exactly as in \cite[Section 10]{DMS}. Indeed, the argument remains unchanged, after replacing $\Abf$ with $\Omega^{1-\delta_3}$ everywhere.

It thus remains to demonstrate that Theorem \ref{t:weak-decay} holds. With this in mind, we show decay at one of two possible radii as stated therein by considering two possible cases for the cone $\Sbf\in \Cscr(Q)$; collapsed and non-collapsed.

\begin{proposition}[Collapsed decay]\label{p:decay-collapsed}
For every $Q, m, n$ and $\varsigma_1>0$ there are positive constants $\varepsilon_c = \eps_c(Q,m,n,\varsigma_1)\leq 1/2$ and $r_c = r_c(Q,m,n,\varsigma_1) \leq 1/2$ with the following property. Assume that 
\begin{itemize}
\item[(i)] $T$ and $\omega$ are as in Assumption \ref{a:main}, and $\|T\| (\Bbf_1) \leq \omega_m (Q+\frac{1}{2})$;
\item[(ii)] There is a cone $\mathbf{S}\in \mathscr{C} (Q)$ which is $M$-balanced (with $M$ as in Proposition \ref{p:balancing}), such that \eqref{e:smallness-weak} and \eqref{e:no-gaps-weak} hold with $\varepsilon_c$ in place of $\varepsilon_1$, and in addition
$\boldsymbol{\mu} (\mathbf{S}) \leq \varepsilon_c$;
\item [(iii)] $\Omega^{2-2\delta_3}\leq \eps_c^2\mathbb{E}(T,\tilde{\Sbf},\Bbf_1)$ for every $\tilde{\Sbf}\in \mathscr{C}(Q)$.
\end{itemize}
Then, there is another cone $\mathbf{S}'\in \mathscr{C} (Q) \setminus \mathscr{P}$ such that 
\begin{equation}\label{e:decay-collapsed}
\mathbb{E} (T, \mathbf{S}', \Bbf_{r_c})
\leq \varsigma_1 \mathbb{E} (T, \mathbf{S}, \Bbf_1)\, .
\end{equation}
\end{proposition}

\begin{proposition}[Non-collapsed decay]\label{p:decay-noncollapsed}
For every $Q, m, n$, $\varepsilon^\star_c >0$ and $\varsigma_1>0$, there are positive constants $\varepsilon_{nc} = \eps_{nc}(Q,m,n,\varepsilon^\star_c,\varsigma_1)\leq 1/2$ and $r_{nc} = r_{nc}(Q,m,n,\varepsilon^\star_c,\varsigma_1) \leq \frac{1}{2}$ with the following property. Assume that 
\begin{itemize}
\item[(i)] $T$ and $\omega$ are as in Assumption \ref{a:main} and $\|T\| (\Bbf_1) \leq \omega_m (Q+\frac{1}{2})$;
\item[(ii)] There is $\mathbf{S}\in \mathscr{C} (Q)$ which is $M$-balanced (with $M$ as in Proposition \ref{p:balancing}), such that 
\eqref{e:smallness-weak} and \eqref{e:no-gaps-weak} hold with $\varepsilon_{nc}$ in place of $\varepsilon_1$ and in addition
$\boldsymbol{\mu} (\mathbf{S}) \geq \varepsilon^\star_c$;
\item [(iii)] $\Omega^{2-2\delta_3}\leq \eps^2_{nc}\mathbb{E}(T,\tilde{\Sbf},\Bbf_1)$ for every $\tilde{\Sbf}\in \mathscr{C}(Q)$.
\end{itemize}
Then, there is another cone $\mathbf{S}'\in \mathscr{C} (Q) \setminus \mathscr{P}$ such that 
\begin{equation}\label{e:decay-noncollapsed}
\mathbb{E} (T, \mathbf{S}', \Bbf_{r_{nc}})
\leq \varsigma_1 \mathbb{E} (T, \mathbf{S}, \Bbf_1)\, .
\end{equation}
\end{proposition}

Since the proofs of Proposition \ref{p:decay-collapsed} and Proposition \ref{p:decay-noncollapsed} exploit the results of Sections \ref{s:height-bd}-\ref{s:nonconc} in the same way as their counterparts in \cite{DMS}, we will merely provide a brief outline of the argument here.

We will argue by contradiction in both the collapsed and the non-collapsed case. 
Namely, we suppose that we have a sequence of currents $T_k$, corresponding semicalibrations $\omega_k$ and cones $\Sbf_k\in \Cscr(Q)$ satisfying the hypotheses of either Proposition \ref{p:decay-collapsed} with $\eps_c^{(k)} = \tfrac{1}{k} \to 0$ or Proposition \ref{p:decay-noncollapsed} for \emph{some fixed} $\eps_c^\star>0$ but $\eps_{nc}^{(k)} = \tfrac{1}{k}\to 0$, but such that the respective decay conclusions \eqref{e:decay-collapsed}, \eqref{e:decay-noncollapsed} fail for any possible choice of radii $r_c,r_{nc}$.

In particular,
\begin{equation}\label{e:blow-up-1}
\left(\frac{\Omega^{2-2\delta_3}}{\mathbb{E} (T_k, \mathbf{S}_k, \Bbf_1)} + \frac{\mathbb{E} (T_k, \mathbf{S}_k, \Bbf_1)}{\boldsymbol{\sigma} (\mathbf{S}_k)^2}\right)\rightarrow 0 \, . 
\end{equation}

For the non-collapsed case, we recall the coherent outer approximations $u_i^{k}$ of Proposition \ref{p:first-blow-up}. Meanwhile, in the collapsed case, we introduce \emph{transverse coherent approximations} as in \cite[Proposition 13.4]{DMS} as follows. Up to extracting a subsequence, we may write $\Sbf_k=\alpha_1\cup\cdots\cup\alpha_N^{k}$ for $\alpha_1$ and $N$ fixed, and write each plane $\alpha_i^{k}$ as the graph of a linear map $A_i^{k}:\alpha_1 \to \alpha_1^\perp$. We may in addition reparameterize $u_i^{k}$ over $\alpha_1$ to obtain a map $v_i^{k} : \widetilde{R}^o_1 \to \mathcal{A}_{Q_i^k} (\alpha_1^\perp)$ whose graph coincides with that of $u_i^{k}$ over $\mathbf{p}_{\alpha_1}^{-1} (\widetilde{R}_1^o)$, where $\widetilde{R}^o_1$ is a suitable graphicality region, defined rigorously in \cite[Proposition 13.4]{DMS}. The transverse coherent approximations are then defined to be the collection of maps
\[
    w_i^{k} := v_i^{k} \ominus A_i^{k}: \tilde{R}_1^o \to \Acal_{Q_i^k}(\alpha_1^\perp)\, , \qquad i=1,\dots, N\, ,
\]
where we use the notational shorthand $g\ominus f$ for the multivalued map $ \sum_i \llbracket g_i - f \rrbracket$. Observe that the estimates of \cite[Proposition 13.4]{DMS}, with $\Abf_k$ replaced by $\Omega_k^{1-\delta_3}$, are valid for $A_i^{k}$, $v_i^{k}$ and $w_i^{k}$. Furthermore, the nonconcentration estimates \cite[Proposition 13.7]{DMS} hold, again with $\Omega_k^{1-\delta_3}$ in place of $\Abf_k$ in the errors.

We in turn define the normalizations
\begin{align*}
\bar{u}^k_i &:= \frac{u^k_i}{\sqrt{\mathbb{E} (T_k, \mathbf{S}_k, \Bbf_1)}}\\
\bar{w}^k_i &:= \frac{w^k_i}{\sqrt{\mathbb{E} (T_k, \mathbf{S}_k, \Bbf_1)}}\, ,
\end{align*}
on $(\Bbf_{\bar r} \cap \alpha^k_i) \setminus B_{1/k} (V)$ and $(\Bbf_{\bar r} \cap \alpha_1)\setminus B_{1/k} (V)$ respectively, for $\bar r := \tfrac{r}{4}$ with $r=r(Q,m,n)$ as in Proposition \ref{p:HS-3} (which we may assume is contained in the domains of definition for $u_i^k$ and $w_i^k$). Arguing as in \cite{DMS}, we may then assume that, up to subsequence, $\bar{u}^k_i$ and $\bar{w}^k_i$ converge strongly in $W^{1,2}$ locally away from $V$ and strongly in $L^2$ on the entirety of $B_r$ to $W^{1,2}$ maps $\bar{u}_i$ and $\bar{w}_i$ that are Dir-minimizing on $(\Bbf_{\bar r} \cap \alpha^k_i) \setminus V$ and $(\Bbf_{\bar r} \cap \alpha_1)\setminus V$, and since $V$ has capacity zero, they are in fact Dir-minimizing on $(\Bbf_{\bar r} \cap \alpha^k_i)$ and $(\Bbf_{\bar r} \cap \alpha_1)$ respectively. Moreover, the conclusions of \cite[Proposition 13.8]{DMS} are satisfied. Exploiting \cite[Theorem 12.2]{DMS} and proceeding as in \cite[Section 13.5]{DMS} to propagate this decay (up to subtracting an appropriate superposition of linear maps), we arrive at the desired contradiction.

\part{Failure of monotonicity for semicalibrated intrinsic planar frequency}\label{pt:counterex}

The aim of this last part is to provide a simple example of a semicalibrated current with good decay properties towards a flat tangent plane, but that does \emph{not} exhibit an almost monotone planar frequency function as introduced in \cite{KW1}. This therefore illustrates a difference between the present approach to Theorem \ref{t:main} when compared to trying to adapt the corresponding one in \cite{KW1}, which we have been unable to adapt to the semicalibrated setting. En route to this, we will point out some differences between the setting herein and the area-minimizing one. 

\section{Planar Frequency Function}\label{s:planarfreq}
We begin by introducing the analogue of the intrinsic planar frequency function of \cite{KW1} for a semicalibrated current $T$. Let $z \in \mathbb{R}^{n + m}$, let $\pi \subset \R^{m+n}$ be an $m$-dimensional plane and let $\rho_0 > 0$. 

We will henceforth work under the following assumption.

\begin{assumption}\label{a:counterex}
   Let $\rho_0>0$. Suppose that $T$ is a semicalibrated rectifiable current in $\mathbf{C}_{\rho_0}(z, \pi)$ satisfying 
\begin{equation} \label{eq: hyp 0}
    \partial T \mres \mathbf{C}_{\rho_0}(z, \pi) = 0, \quad \text{and} \quad \sup_{p \in \spt T \cap \mathbf{C}_{\rho_0}(z, \pi)} \dist(p, \pi + z) < \infty.
\end{equation}
\end{assumption}

Let $\phi \colon [0, \infty) \to [0,1]$ be the monotone Lipschitz cutoff function defined in \eqref{e:def_phi}. For $r \in (0, \rho_0]$, we can define the \textit{intrinsic $L^2$ height} of $T$ at scale $r$ around $z$ with respect to the plane $\pi$ to be
\begin{equation} \label{eq: intrinsic height}
    H_{T, \pi, z}(r) := \frac{2}{r^{m-1}} \int_{\mathbf{C}_r(z, \pi) \setminus \overline{\mathbf{C}_{r/2}(z, \pi)} } \dist^2(p, \pi + z) \frac{\vert \nabla_{\vec T} |\mathbf{p}_{\pi}(p-z)| \vert^2}{|\mathbf{p}_{\pi}(p-z)|}  \, d\Vert T \Vert(p)\, ,  
\end{equation}
Note that $\eqref{eq: intrinsic height}$ can be rewritten as 
\begin{equation}\label{e:H}
     H_{T, \pi, z}(r) = - \frac{1}{r^{m-1}} \int \dist^2(p, \pi + z) \frac{\vert \nabla_{\vec T} |\mathbf{p}_{\pi}(p-z)| \vert^2}{ |\mathbf{p}_{\pi}(p-z)|} \phi^\prime\left(\frac{|\mathbf{p}_{\pi}(p-z)|}{r}\right) \, d\Vert T \Vert(p)\,. 
\end{equation}

Furthermore, we can introduce the \textit{intrinsic Dirichlet energy} of $T$ at scale $r$ around $z$ with respect to the plane $\pi$ as
\begin{equation} \label{eq: intrinsic Dir}
    D_{T, \pi, z}(r) := \frac{1}{2r^{m-2}}  \int \vert \mathbf{p}_{T}(p) - \mathbf{p}_{\pi} (p) \vert^2 \phi\left(\frac{|\mathbf{p}_{\pi}(p-z)|}{r}\right) \, d \Vert T \Vert(p)\,. 
\end{equation}
Note that $D_{T,\pi,z}(r)$ may be considered as a regularization of the non-oriented tilt excess of $T$ in the cylinder $\Cbf_r(z,\pi)$, cf. \eqref{e:nonoriented}. Pertinent to our setting, we define the corresponding intrinsic semicalibrated term
\begin{equation}
    L_{T,\pi,z}(r):= \frac{1}{2r^{m-2}}\int T(d \omega \mres \pbf_{\pi^\perp}(p))\phi\left(\frac{|\mathbf{p}_{\pi}(p-z)|}{r}\right) \, d \Vert T \Vert(p)
\end{equation}
and
\begin{equation}
    \Gamma_{T,\pi,z}(r) := D_{T,\pi,z}(r) + L_{T,\pi,z}(r).
\end{equation}
Notice that the additional term $T(d \omega \mres \pbf_{\pi^\perp}(p))$ above arises from testing the first variation of $T$ with the variation vector field  $X(p)=\phi\left(\tfrac{|\mathbf{p}_{\pi}(p-z)|}{r}\right)\pbf_{\pi^\perp}(p)$; see \cite[Section 7]{Spolaor_15}. In analogy with \cite[Section 3]{KW1}, whenever $ H_{T, \pi, z}(r) > 0$, we may define the \textit{intrinsic planar frequency function} $N_{T, \pi, z}(r)$ of $T$ at $Z$ relative to $P$ by 
\begin{equation} \label{equation: semicalibrated planar frequency}
    N_{T, \pi, z}(r) = \frac{\Gamma_{T, \pi, z}(r)}{H_{T, \pi, z}(r)}\, . 
\end{equation}
Note that when $T$ is calibrated (and in particular area-minimizing), i.e. when $d\omega \equiv 0$, the above frequency indeed reduces to the one introduced in \cite{KW1} for area-minimizing currents. In \cite{KW1}, it is shown that the intrinsic planar frequency associated to area-minimizing integral currents is almost-monotone under a suitable decay hypothesis (see Section \ref{ss:KW} below for a more precise statement), thus laying the foundation for a more refined analysis of the singular set of area-minimizers. 

The main result of this part is the following.
\begin{theorem}\label{t:counterex}
    There exists a smooth, radially symmetric function $f: B_1 \subset \pi \equiv \mathbb{R}^m \times \{0\} \to \R$ such that the associated semicalibrated current $\Gbf_f$ has the property
    \[
        N_{T, \pi, 0}(r) \rightarrow + \infty \qquad \text{ as $r \downarrow 0$.}
    \]
\end{theorem}

Namely, we provide the construction of an example that not only violates almost-monotonicity of the intrinsic planar frequency function in the semicalibrated setting, but allows for it to diverge to $+\infty$ as the scale goes to zero. Before we proceed with the construction, let us provide a more detailed heuristic explanation, together with a comparison with what happens in the area-minimizing setting.

\subsection{Comparison with Proposition \ref{p:coarse=fine}}
One may consider the statement of Proposition \ref{p:coarse=fine} as a conclusion that a posteriori, one does not require center manifolds to take blow ups at points $x\in \Ffrak_Q(T)$ with singularity degree $\Irm(T,x) \in [1,2-\delta_2)$ (and indeed, this is the case in the work \cite{KW1} when the planar frequency is in $[1,2)$). This suggests that one may possibly use a planar frequency function like $N_{T,\pi,z}$ (in place of a frequency relative to center manifolds) to analyze such points. This is \emph{not inconsistent} with the validity of Theorem \ref{t:counterex}. Indeed, in the latter, the smooth submanifold $\graph(f)$ \emph{coincides} with the center manifold associated to $T=\Gbf_f$ locally around the origin. On the other hand, our choice of $f$ will have \emph{infinite order of vanishing} at the origin, which corresponds to blow-up of the planar frequency function there. Note, however, that $0$ is not a flat singular point of $\Gbf_f$ in this case; in fact $\Gbf_f$ has no singularities. For the same reason, no single-sheeted example will be inconsistent with the validity of Proposition \ref{p:coarse=fine}. However, since it is not clear how to meaningfully restrict the notion of intrinsic planar frequency to such a specific scenario in order to improve its properties, we do not pursue this any further.

\subsection{Comparison with \cite{KW1}}\label{ss:KW}
The area-minimizing hypothesis is crucially used in \cite{KW1} to control the error terms arising when differentiating $H_{T, \pi, z}$, and $D_{T, \pi, z}$, which in turn produce errors for the radial derivative of $N_{T,\pi,z}$. More precisely, the area-minimizing property of $T$ is exploited therein to infer the following bounds, cf. Lemma 3.9 and Lemma 3.11 of \textit{loc. cit.}, 
\begin{align} \label{e: monotonicity of planar height}
    &\left\vert H_{T, \pi, z}^\prime(r) + 2 r^{- m} \int \left\vert \nabla^{\perp} |\mathbf{p}_{\pi}(p-z)| \right\vert^2 |\mathbf{p}_{\pi}(p-z)| \phi'\left(\frac{|\mathbf{p}_{\pi}(p-z)|}{r}\right) \, d\Vert T \Vert(p) \right\vert \\
    &\qquad\qquad\leq C \eta^{2\gamma} r^{2 \alpha \gamma - 1} H_{T, \pi, z}(r)\, ,\notag
\end{align}
where $\nabla^\perp \equiv \nabla_{T}^\perp$, and 
\begin{align}
\begin{split} \label{e: monotonicity of planar Dirichlet}
    & \left\vert D_{T, \pi, z}^\prime(r) +  2 r^{- m} \int \frac{\big\vert \pbf_{\pi}^\perp(\nabla_{T} |\mathbf{p}_{\pi}(p-z)|) \big\vert^2}{\big\vert \nabla_T |\mathbf{p}_{\pi}(p-z)| \big\vert^2} |\mathbf{p}_{\pi}(p-z)| \phi'\left(\frac{|\mathbf{p}_{\pi}(p-z)|}{r}\right) \, d\Vert T \Vert(p) \right\vert \\
    & \qquad \qquad \qquad \leq \frac{C}{r}D_{T,\pi,z}(r)^{\gamma}\left((m - 1)D_{T, \pi, z}(r) + r D_{T, \pi, z}^\prime(r)\right), 
    \end{split}
\end{align}
for some $\gamma=\gamma(Q,m,n)>0$ and $C=C(Q,m,n)>0$, whenever $T$ as in Assumption \ref{a:counterex} satisfies the mass ratio bounds 
\[
    \Theta(T,z) \geq Q, \qquad \|T\|(\Cbf_{7\rho_0/4}(z,\pi)) \leq (Q+\delta) \omega_m \left(\frac{7\rho_0}{4}\right)^m\, ,
\]
for some $\delta=\delta(Q,m,n) > 0$ and the additional planar decay hypothesis
\begin{equation}\label{e:KW-decay}
    \frac{1}{\omega_m (7s /4)^{m+2}}\int_{\Cbf_{7s /4}(z,\pi)} \dist^2(p, \pi +z) \, d\|T\|(p) \leq \eta^2 \left(\frac{s}{\rho_0}\right)^{2\alpha} \qquad \forall s \in [\sigma_0,\rho_0]\, ,
\end{equation}
for some $\sigma_0\in (0,\rho_0)$, some $\eta_0=\eta_0(Q,m,n)>0$ and $\eta\in (0,\eta_0]$.

The estimates \eqref{e: monotonicity of planar height} and \eqref{e: monotonicity of planar Dirichlet}, together with the variational identities for $H_{T,\pi,z}$, $D_{T,\pi,z}$, in turn can be used to prove the almost-monotonicity 
\[
    N_{T,\pi,z}(r) \leq e^{C\eta^\gamma (s/\rho_0)^{\alpha\gamma}} N_{T,\pi,z}(s)\qquad \forall \sigma_0 \leq r < s \leq \rho_0\,,
\]
of the intrinsic planar frequency function in the area-minimizing case, provided that $H_{T,\pi,z}(\tau) > 0$ on $[r,s]$. The first challenge when trying to adapt this argument to the semicalibrated setting is that the terms $H_{T, \pi, z}^\prime$ and $D_{T, \pi, z}^\prime$ will now contain extra errors depending on the semicalibration $d\omega$ (due the fact that the first variation of a semicalibrated current has a non-vanishing right-hand side). Moreover, one must additionally consider the behavior of $L_{T, \pi, z}^\prime$ when controlling the variational error terms. Complications arise when one tries to bound all of these error terms by powers of the intrinsic Dirichlet energy and $L^2$ height, as in \eqref{e: monotonicity of planar height} and \eqref{e: monotonicity of planar Dirichlet}. 

\section{Proof of Theorem \ref{t:counterex}}\label{s:counterex}
We are now in a position to construct a counterexample to the almost-monotonicity of the intrinsic planar frequency $N_{T,\pi,0}$ associated to a semicalibrated current $T=\Gbf_f$ associated to the graph of a smooth radially symmetric function $f$ relative to the plane $\pi\equiv \R^m\times \{0\}\subset \R^{m+1}$, as claimed in Theorem \ref{t:counterex}. 

To this end, consider $f \in C^{\infty}(B_1)$, where $B_1 \subset \R^m \times \{0\} \subset \R^{m+1}$. Consider then the graph of $f$ as a submanifold of $\mathbb{R}^{m + 1}$: 
\begin{equation*}
    \graph(f) = \{(x, f(x)); \; x \in B_1\}. 
\end{equation*}
We claim that $\graph(f)$ is a semicalibrated submanifold of $\mathbb{R}^{m +1}$. Indeed, consider the unit normal to $\graph(f)$ given by 
\begin{equation*}
    \nu_x = \frac{1}{\sqrt{1 + \vert \nabla f(x) \vert^2}}(- \nabla f(x), 1), 
\end{equation*}
and define the $m$-form $\omega(X_1, \ldots, X_m) = \det(X_1, \ldots, X_m, \nu)$ for vectors $X_1, \ldots, X_m$. It is then straightforward to check that $\vert \omega(Z_1, \ldots, Z_m) \vert \leq 1$ for unit vectors $Z_1, \ldots, Z_m$. Furthermore, we have that $\omega(Y_1, \ldots, Y_m) = 1$, if $Y_i$ are orthonormal vectors belonging to $T_{(x, f(x))} \graph(f)$, so that the submanifold $\graph(f) \subset \R^{m+1}$ is semicalibrated. In particular, the current $T=\Gbf_f$ associated to it (see e.g. \cite{DLS14Lp,DLS_multiple_valued}) is also semicalibrated. From this, it is clear that not every semicalibrated current can have an almost-monotone planar frequency function. We will however provide a natural example to illustrate the blow-up of planar frequency. We recall the quantities $D_{T, \pi, 0}(r)$, $ H_{T, \pi, 0}(r)$, and $L_{T,\pi,z}(r)$ introduced in the preceding section for this particular choice of $T$, with $\pi \equiv \mathbb{R}^m\times \{0\}$. 

\subsection{Intrinsic quantities for a graph}
We start by unpacking the definitions of intrinsic Dirichlet energy, $L^2$ height, and the semicalibrated term in the case where $T=\Gbf_f$ for $f$ as above; we will define $f$ later. Let us begin with the energy $D_{T, \pi, 0}(r)$. The projection matrix associated with $\pi$ is given by 
\begin{equation*}
    \pbf_\pi = \begin{pmatrix}
        \Id_{m \times m} & 0 \\
        0 & 0 
    \end{pmatrix}, 
\end{equation*}
while the projection matrix for the tangent space $T_{(x, f(x))} \graph(f)$ is 
\begin{align*}
    \pbf_{T_{(x, f(x))} \graph(f)} & = \Id_{(m + 1) \times (m + 1)} - \nu_x \otimes \nu_x 
\end{align*} 
In particular, we can compute 
\begin{align*}
\frac{1}{2} \vert \pbf_{T_{(x, f(x))} \graph(f)} - \pbf_\pi \vert^2 = 1- (\pbf_{T_{(x, f(x))} \graph(f)} : \pbf_\pi) = \frac{\vert \nabla f \vert^2}{1 + \vert \nabla f \vert^2}, 
\end{align*}
where $A:B$ denotes the Hilbert-Schmidt inner product between matrices $A,B$, so that 
\begin{align*}
    D_{T, \pi, 0}(r) & = r^{2 - m} \int \phi\left(\frac{|x|}{r}\right) \frac{\vert \nabla f(x) \vert^2}{1 + \vert \nabla f(x) \vert^2} \, d \Vert T\Vert(x,f(x)) \,. 
\end{align*}
Letting $\phi$ converge to the characteristic function of the unit interval from below, and recalling the area formula for a graph, we obtain
\begin{align}\label{e:D-graph}     
    D_{T, \pi, 0}(r) = r^{2 - m} \int_{B_r} \frac{\vert \nabla f \vert^2}{\sqrt{1 + \vert \nabla f \vert^2}} \; d\mathcal{L}^m. 
\end{align}
We can now turn to the height $H_{T, \pi, 0}(r)$. Write 
\begin{align*}
     H_{T, \pi, 0}(r) = - r^{1 - m} \int \vert f(x)\vert^2 \frac{ \vert \nabla_T |x| \vert^2}{|x|} \phi'\left(\frac{|x|}{r}\right) \, d \Vert T \Vert(x,f(x)),
\end{align*}
and, after recalling $\nabla r = \frac{(x, 0)}{|x|},$ we can compute 
\begin{align*}
    \vert \nabla_{T} |x| \vert^2 = \vert \pbf_{T_{(x, f(x))} \graph(f)} (\nabla |x|) \vert^2 = \frac{1 + \vert \nabla_\theta f \vert^2}{1 + \vert \nabla f \vert^2},  
\end{align*}
where $\nabla_\theta$ denotes the angular part of the gradient. Thus, after letting $\phi$ converge to the characteristic function of the unit interval, we infer 
\begin{align}\label{e:H-graph}
     H_{T, \pi, 0}(r) = r^{1 - m} \int_{\partial B_\rho(0)} \vert f \vert^2 \frac{1 + \vert \nabla_\theta f \vert^2}{\sqrt{1 + \vert \nabla f \vert^2}} \, d \mathcal{H}^{m - 1}. 
\end{align}
Finally, we rewrite the definition of the semicalibrated term 
\begin{equation*}
    L_{T,\pi,0}(r) = \frac{1}{2r^{m-2}}\int \langle d\omega \mres (0,y), \vec T \rangle  \phi\left(\frac{|x|}{r}\right) \; d\Vert T\Vert(x,y),
\end{equation*}
where we write $\R^{m+1}\ni p=(x,y)\in \pi\times \pi^\perp$. Letting $\phi$ converge to the characteristic function of the unit interval and again using the area formula, we obtain  
\begin{equation}
      L_{T,\pi,0}(r) = \frac{1}{2r^{m-2}}\int_{B_r} \langle d\omega \mres (0, \ldots, 0, f), \vec{T} \rangle \sqrt{1 + \vert \nabla f \vert^2} \; d\mathcal{L}^m\, , 
\end{equation}
where $\vec{T}(x)$ is the $m$-vector $\tau_1 \wedge \ldots \wedge \tau_m$, for $\{\tau_i\}_i$ a basis of the tangent space $T_{(x, f(x))} \graph(f)$. Note that here we have also used that in this particular setting, $\pbf_{\pi^\perp}(x,y) = (0, \ldots, 0, f(x))$. Note that $$\tau_i = \frac{e_i + (\nabla f \cdot e_i) e_{m + 1}}{\sqrt{1 + (\nabla f \cdot e_i)^2}},$$ for $i \in \{1, \ldots, m\}$, where $\{e_i\}_i$ is a basis of $\mathbb{R}^{m + 1}$. The $(m +1)$-form $d\omega$ is given by
\begin{align*}
    d\omega = (-1)^{m}\divergence\left( \frac{\nabla f}{\sqrt{1 + \vert \nabla f \vert^2}}\right) dx_1 \wedge dx_2 \wedge \ldots \wedge dx_{m +1}\, .
\end{align*}
Note that one can alternatively use Cartan's magic formula to arrive at the same final expression for $d\omega$. Then, 
\begin{align*}
     (0, \ldots, 0, f) \wedge \vec{T} = (f e_{m + 1}) \wedge \tau_1 \wedge \ldots \wedge \tau_m = \frac{(- 1)^m f}{\Pi_{i = 1}^{m}\sqrt{1 + (\nabla f \cdot e_i)^2}} e_1 \wedge e_2 \wedge \ldots \wedge e_{m + 1}, 
\end{align*}
so that 
\begin{align*}
    \langle d\omega \mres (0, \ldots, 0, f), \vec{T} \rangle & =\divergence\left( \frac{\nabla f}{\sqrt{1 + \vert \nabla f \vert^2}}\right) \frac{f}{\Pi_{i = 1}^{m}\sqrt{1 + (\nabla f \cdot e_i)^2}}
\end{align*}
Thus, the semicalibrated term is 
\begin{equation}\label{e:L-graph}
    L_{T,\pi,0}(r) = \frac{1}{2r^{m-2}} \int_{B_r} \divergence\left( \frac{\nabla f}{\sqrt{1 + \vert \nabla f \vert^2}}\right) \frac{f}{\Pi_{i = 1}^{m}\sqrt{1 + (\nabla f \cdot e_i)^2}} \sqrt{1 + \vert \nabla f \vert^2} \; d\mathcal{L}^m\, . 
\end{equation}

\subsection{Definition of $f$}
We are now in a position to define our radially symmetric function. Let $f$ given by 
\[
    f(x) = \begin{cases}
            e^{-1/\vert x \vert^2} & x \neq 0 \\
            0 & x=0\, .
            \end{cases}
\]
Note that $\nabla f(0) = 0$ and that $f$ is indeed radially symmetric, so that introducing polar coordinates we can write (abusing notation) $f(r) = e^{-1/r^2}$, for $r \geq 0$. By the argument at the beginning of this section, the current $T=\Gbf_f$ associated to the graph of this function is semicalibrated. Furthermore, note that the hypothesis of \cite[Theorem 3.4]{KW1} are satisfied. More precisely, let $\rho_0 > 0$, and consider $T\res \mathbf{C}_{7\rho_0/4}(0, \pi)$ for $\pi=\R^m\times\{0\}\subset\R^{m+1}$. In particular, $\Theta(T, 0) = 1$, and the almost-monotonicity of mass ratios (see e.g. \cite[Proposition 2.1]{DLSS-uniqueness}) guarantees that $\Vert T \Vert(\mathbf{C}_{7\rho_0/4}(0, \pi)) \leq (1 + \delta)\omega_m (7\rho_0/4)^m$ for $\rho_0$ sufficiently small, where $\delta$ is the parameter of \cite[Theorem 3.4]{KW1}. In addition, thanks to the exponential decay of $f$ towards 0, there exist $\eta > 0$, $\sigma_0 \in (0, \rho_0)$, and $\alpha \in (0, 1)$ such that the decay hypothesis \eqref{e:KW-decay} holds about the origin, namely
\begin{equation*}
    \frac{1}{\omega_m (7s/4)^{m + 2}} \int_{\mathbf{C}_{7s/4}(0, \pi)} \dist^2(p, \pi) \, d\Vert T \Vert(p) \leq \eta^2 \left( \frac{s}{\rho_0} \right)^{2 \alpha}, 
\end{equation*}
for all $\rho \in [\sigma_0, \rho_0]$.
Consider now the planar frequency function $N_{T,\pi,0}$ and use \eqref{e:D-graph} and \eqref{e:H-graph} to write the energy and the height for this particular function:
\begin{equation*}
    D_{T, \pi, 0}(\rho) = \rho^{2 - m} \omega_{m - 1} \int_{0}^{\rho} \frac{4e^{-2/r^2} r^{-6}}{\sqrt{1 + 4e^{-2/r^2} r^{-6}} } r^{m - 1} \, dr\, ,
\end{equation*}
and 
\begin{equation*}
     H_{T, \pi, 0}(\rho) =  \omega_{m - 1} \frac{e^{-2/\rho^2}}{\sqrt{1 + 4e^{-2/\rho^2} \rho^{-6}}}\, . 
\end{equation*}
Note that $H_{T, \pi, 0}(\rho) > 0$ for all $\rho > 0$, implying that $N_{T, \pi, 0}$ is always well-defined. Thus, for any $\rho>0$ sufficiently small, the classical planar frequency function can be estimated from below as follows: 
\begin{align*}
     \frac{D_{T, \pi, 0}(\rho)}{H_{T, \pi, 0}(\rho)} & = e^{2/\rho^2}\sqrt{1 + 4e^{-2/\rho^2} \rho^{-6}} \rho^{2 - m} \int_{0}^{\rho} \frac{4e^{-2/r^2} r^{-6}}{\sqrt{1 + 4e^{-2/r^2} r^{-6}} } r^{m - 1} \, dr \\
     & = \rho^{-1 - m} e^{1/\rho^2} \sqrt{e^{2/\rho^2} \rho^{6} + 4}  \int_{0}^{\rho} \frac{4e^{-1/r^2} r^{-3}}{\sqrt{4 + e^{2/r^2} r^{6}} } r^{m - 1} \, dr \\
    & \geq C \rho^{2 - m} e^{2/\rho^2} \int_{0}^{\rho} e^{- 2/r^2}r^{m - 7} \, dr \\
    & \geq C \rho^{2 - m} e^{2/\rho^2} 2^{m/2 - 4} \Gamma\left(3 - \frac{m}{2}, \frac{2}{\rho^2}\right), 
\end{align*}
where $\Gamma(s, x)$ is the incomplete gamma function 
\begin{equation*}
    \Gamma(s, x) = \int_{x}^{\infty} t^{s - 1} e^{- t} \; dt.
\end{equation*}
Recalling now the asymptotic $\Gamma(s, x)x^{- s + 1}e^{x} \rightarrow 1$, as $x \rightarrow \infty$, we deduce that 
\begin{align*}
    N_{T, \pi, 0}(\rho) \geq \frac{C}{\rho^2} \left( \Gamma\left(3 - \frac{m}{2}, \frac{2}{\rho^2}\right) e^{2/\rho^2} \left( \frac{2}{\rho^2}\right)^{1 + m/2 - 3}\right) =: \frac{1}{\rho^2} \eta(\rho), 
\end{align*}
where $\eta(\rho) \rightarrow 1$ as $\rho \rightarrow 0$, which yields
\[
    \frac{D_{T, \pi, 0}(\rho)}{H_{T, \pi, 0}(\rho)} \rightarrow \infty \qquad \text{as $\rho \rightarrow 0$}\, ,
\]
as desired. We now wish to compute the semicalibrated term in the intrinsic planar frequency, namely $ L_{T,\pi,0}(\rho)/H_{T, \pi, 0}(\rho)$. We record the minimal surface equation for a radial function on $\mathbb{R}^{m}$:
\begin{equation*}
    \divergence\left( \frac{\nabla f}{\sqrt{1 + \vert \nabla f \vert^2}}\right) = \frac{1}{\sqrt{1 + (f^\prime)^2}} \left( \frac{f^{\prime \prime}}{1 + (f^\prime)^2} + \frac{m-1}{r}f^\prime \right). 
\end{equation*}
Thus, we can estimate
\begin{align*}
    \frac{L_{T,P,0}(\rho)}{H_{T, P, 0}(\rho)} &= \frac{e^{1/\rho^2}}{\rho^{m+1}} \sqrt{e^{2/\rho^2} \rho^{6} + 4} \int_{0}^{\rho} \frac{r^{m - 1}e^{- 2/r^2}\sqrt{1+(f'(r))^2}}{\Pi_{i = 1}^{m}\sqrt{1 + (\nabla f \cdot e_i)^2}} \left[\frac{2(m-1)}{r^4} + \frac{4 - 6r^2}{r^6 + 4 e^{-2/r^2}} \right]   \, dr \\
     & \quad \geq \frac{C e^{1/\rho^2}}{\rho^{m+1}} \sqrt{e^{2/\rho^2} \rho^{6} + 4} \int_{0}^{\rho} r^{m - 1}e^{- 2/r^2} \left[\frac{2(m-1)}{r^4} + \frac{4 - 6r^2}{r^6 + 4 e^{-2/r^2}} \right] \, dr.  
\end{align*}
Splitting then the square bracket and analyzing the two integrands separately via the incomplete Gamma function again, one deduces that $L_{T,\pi,0}(\rho)/H_{T, \pi, 0}(\rho)$ also diverges as $\rho \rightarrow 0$.

\bibliographystyle{amsalpha} 
\bibliography{references} 

\providecommand{\bysame}{\leavevmode\hbox to3em{\hrulefill}\thinspace}
\providecommand{\MR}{\relax\ifhmode\unskip\space\fi MR }
\providecommand{\MRhref}[2]{%
  \href{http://www.ams.org/mathscinet-getitem?mr=#1}{#2}
}
\providecommand{\href}[2]{#2}
\begin{thebibliography}{DLDPHM23}

\bibitem[Alm00]{Almgren_regularity}
Frederick~J Almgren, \emph{Almgren's big regularity paper: Q-valued functions
  minimizing dirichlet's integral and the regularity of area-minimizing
  rectifiable currents up to codimension 2}, vol.~1, World scientific, 2000.

\bibitem[Bel14]{Bellettini}
Costante Bellettini, \emph{Uniqueness of tangent cones to positive-(p, p)
  integral cycles}, Duke Mathematical Journal \textbf{163} (2014), no.~4,
  705--732.

\bibitem[BR12]{BellettiniRiviere}
Costante Bellettini and Tristan Rivi\`ere, \emph{The regularity of special
  {L}egendrian integral cycles}, Ann. Sc. Norm. Super. Pisa Cl. Sci. (5)
  \textbf{11} (2012), no.~1, 61--142. \MR{2953045}

\bibitem[CR23]{CaniatoRivière}
Riccardo Caniato and Tristan Rivi{\`e}re, \emph{The unique tangent cone
  property for weakly holomorphic maps into projective algebraic varieties},
  Duke Mathematical Journal \textbf{172} (2023), no.~13, 2471--2536.

\bibitem[CS]{CS}
Gianmarco Caldini and Anna Skorobogatova, \emph{Forthcoming}.

\bibitem[DIW21]{DoWaIo}
Aleksander Doan, Eleny-Nicoleta Ionel, and Thomas Walpuski, \emph{The
  gopakumar-vafa finiteness conjecture}, 2021.

\bibitem[DL16]{DL-survey-JDG}
Camillo De~Lellis, \emph{The size of the singular set of area-minimizing
  currents}, Surveys in differential geometry 2016. {A}dvances in geometry and
  mathematical physics, Surv. Differ. Geom., vol.~21, Int. Press, Somerville,
  MA, 2016, pp.~1--83. \MR{3525093}

\bibitem[DL18]{DL-All}
\bysame, \emph{Allard's interior regularity theorem: an invitation to
  stationary varifolds}, Nonlinear analysis in geometry and applied
  mathematics. {P}art 2, Harv. Univ. Cent. Math. Sci. Appl. Ser. Math., vol.~2,
  Int. Press, Somerville, MA, 2018, pp.~23--49.

\bibitem[DLDPHM23]{DLDPHM}
Camillo De~Lellis, Guido De~Philippis, Jonas Hirsch, and Annalisa Massaccesi,
  \emph{On the boundary behavior of mass-minimizing integral currents}, Mem.
  Amer. Math. Soc. \textbf{291} (2023), no.~1446, v+166. \MR{4672045}

\bibitem[DLF23]{DLF}
Camillo De~Lellis and Ian Fleschler, \emph{An elementary rectifiability lemma
  and some applications}, 2023.

\bibitem[DLHMS20]{DLHMS}
Camillo De~Lellis, Jonas Hirsch, Andrea Marchese, and Salvatore Stuvard,
  \emph{Regularity of area minimizing currents mod p}, Geometric and Functional
  Analysis \textbf{30} (2020), no.~5, 1224--1336.

\bibitem[DLMS23]{DMS}
Camillo De~Lellis, Paul Minter, and Anna Skorobogatova, \emph{The fine
  structure of the singular set of area-minimizing integral currents {III}:
  Frequency 1 flat singular points and {$\mathcal{H}^{m-2}$}-a.e. uniqueness of
  tangent cones}, arXiv preprint (2023).

\bibitem[DLMSV18]{DLMSV}
Camillo De~Lellis, Andrea Marchese, Emanuele Spadaro, and Daniele Valtorta,
  \emph{Rectifiability and upper {M}inkowski bounds for singularities of
  harmonic {$Q$}-valued maps}, Comment. Math. Helv. \textbf{93} (2018), no.~4,
  737--779. \MR{3880226}

\bibitem[DLS11]{DLS_MAMS}
Camillo De~Lellis and Emanuele Spadaro, \emph{$q$-valued functions revisited},
  Mem. Amer. Math. Soc. \textbf{211} (2011), no.~991, vi+79.

\bibitem[DLS14]{DLS14Lp}
\bysame, \emph{Regularity of area minimizing currents i: gradient $l^p$
  estimates}, Geometric and Functional Analysis \textbf{24} (2014), no.~6,
  1831--1884.

\bibitem[DLS15]{DLS_multiple_valued}
\bysame, \emph{Multiple valued functions and integral currents}, Ann. Sc. Norm.
  Super. Pisa Cl. Sci. (5) \textbf{14} (2015), no.~4, 1239--1269.

\bibitem[DLS16a]{DLS16centermfld}
\bysame, \emph{Regularity of area minimizing currents ii: center manifold},
  annals of Mathematics (2016), 499--575.

\bibitem[DLS16b]{DLS16blowup}
\bysame, \emph{Regularity of area minimizing currents iii: blow-up}, annals of
  Mathematics (2016), 577--617.

\bibitem[DLS23a]{DLSk1}
Camillo De~Lellis and Anna Skorobogatova, \emph{The fine structure of the
  singular set of area-minimizing integral currents i: the singularity degree
  of flat singular points}, arXiv preprint (2023).

\bibitem[DLS23b]{DLSk2}
\bysame, \emph{The fine structure of the singular set of area-minimizing
  integral currents ii: rectifiability of flat singular points with singularity
  degree larger than $1$}, to appear in Commentarii Mathematici Helvetici
  (2023).

\bibitem[DLSS17a]{DLSS2}
Camillo De~Lellis, Emanuele Spadaro, and Luca Spolaor, \emph{Regularity theory
  for 2-dimensional almost minimal currents ii: branched center manifold},
  Annals of PDE \textbf{3} (2017), 1--85.

\bibitem[DLSS17b]{DLSS-uniqueness}
\bysame, \emph{Uniqueness of tangent cones for two-dimensional
  almost-minimizing currents}, Communications on Pure and Applied Mathematics
  \textbf{70} (2017), no.~7, 1402--1421.

\bibitem[DLSS18]{DLSS1}
\bysame, \emph{Regularity theory for 2-dimensional almost minimal currents i:
  Lipschitz approximation}, Transactions of the American Mathematical Society
  \textbf{370} (2018), no.~3, 1783--1801.

\bibitem[DLSS20]{DLSS3}
\bysame, \emph{Regularity theory for $2 $-dimensional almost minimal currents
  iii: Blowup}, Journal of Differential Geometry \textbf{116} (2020), no.~1,
  125--185.

\bibitem[DW21]{DoWa}
Aleksander Doan and Thomas Walpuski, \emph{Castelnuovo's bound and rigidity in
  almost complex geometry}, Adv. Math. \textbf{379} (2021), Paper No. 107550,
  23. \MR{4199270}

\bibitem[GPT99]{Gu1}
J.~Gutowski, G~Papadopoulos, and P.~K. Townsend, \emph{Supersymmetry and
  generalized calibrations}, Phys. Rev. D (3) \textbf{60(10):106006, 11}
  (1999).

\bibitem[Gra06]{Gr}
Mariana Grana, \emph{Flux compactifications in string theory: a comprehensive
  review}, Phys. Rep. \textbf{423(3)} (2006), 91–158.

\bibitem[GS23]{GS}
Max Goering and Anna Skorobogatova, \emph{Flat interior singularities for area
  almost-minimizing currents}, 2023.

\bibitem[Gut01]{Gu2}
Jan. Gutowski, \emph{Generalized calibrations}, In Progress in string theory
  and M-theory (Carg`ese, 1999) \textbf{volume 564 of NATO Sci. Ser. C Math.
  Phys. Sci.} (2001), 343–346.

\bibitem[HL82]{HaLa}
Reese Harvey and H.~Blaine Lawson, Jr., \emph{Calibrated geometries}, Acta
  Math. \textbf{148} (1982), 47--157. \MR{666108}

\bibitem[HL11]{HanLin}
Q.~Han and F.~Lin, \emph{Elliptic partial differential equations}, Courant
  lecture notes in mathematics, Courant Institute of Mathematical Sciences, New
  York University, 2011.

\bibitem[Jor75]{Cimille}
Camille Jordan, \emph{Essai sur la g\'eom\'etrie \`a{} {$n$} dimensions}, Bull.
  Soc. Math. France \textbf{3} (1875), 103--174. \MR{1503705}

\bibitem[Joy07]{Jo}
Dominic~D. Joyce, \emph{Riemannian holonomy groups and calibrated geometry},
  Oxford Graduate Texts in Mathematics, vol.~12, Oxford University Press,
  Oxford, 2007. \MR{2292510}

\bibitem[KW]{KW3}
Brian Krummel and Neshan Wickramasekera, \emph{Analysis of singularities of
  area minimising currents: higher order decay estimates at branch points and
  rectifiability of the singular set}, In preparation.

\bibitem[KW17]{KW}
Brian Krummel and Neshan Wickramasekera, \emph{Fine properties of branch point
  singularities: Dirichlet energy minimizing multi-valued functions}, arXiv
  preprint arXiv:1711.06222 (2017).

\bibitem[KW23a]{KW2}
Brian Krummel and Neshan Wickramasekera, \emph{Analysis of singularities of
  area minimizing currents: a uniform height bound, estimates away from branch
  points of rapid decay, and uniqueness of tangent cones}.

\bibitem[KW23b]{KW1}
\bysame, \emph{Analysis of singularities of area minimizing currents: planar
  frequency, branch points of rapid decay, and weak locally uniform
  approximation}.

\bibitem[Liu21]{Liu}
Zhenhua Liu, \emph{On a conjecture of almgren: area-minimizing surfaces with
  fractal singularities}, arXiv preprint (2021).

\bibitem[Mor82]{Morgan}
Frank Morgan, \emph{On the singular structure of two-dimensional area
  minimizing surfaces in $\mathbb{R}^n$}, Mathematische Annalen \textbf{261}
  (1982), 101--110.

\bibitem[Mor16]{MorganGMT}
\bysame, \emph{Geometric measure theory}, fifth ed., Elsevier/Academic Press,
  Amsterdam, 2016, A beginner's guide, Illustrated by James F. Bredt.
  \MR{3497381}

\bibitem[MW24]{MW}
Paul Minter and Neshan Wickramasekera, \emph{A structure theory for stable
  codimension 1 integral varifolds with applications to area minimising
  hypersurfaces mod $p$}, Journal of the American Mathematical Society
  \textbf{37} (2024), no.~3, 861--927.

\bibitem[NV17]{NV_Annals}
Aaron Naber and Daniele Valtorta, \emph{Rectifiable-reifenberg and the
  regularity of stationary and minimizing harmonic maps}, Annals of Mathematics
  \textbf{185} (2017), no.~1, 131--227.

\bibitem[NV20]{NV_varifolds}
\bysame, \emph{The singular structure and regularity of stationary varifolds},
  Journal of the European Mathematical Society \textbf{22} (2020), no.~10,
  3305--3382.

\bibitem[PR10]{PumbergerRivière}
David Pumberger and Tristan Rivi{\`e}re, \emph{Uniqueness of tangent cones for
  semicalibrated integral $2 $-cycles}, Duke Math. J. \textbf{151} (2010),
  no.~1, 441--480.

\bibitem[RT09]{TianRivière}
Tristan Riviere and Gang Tian, \emph{The singular set of 1-1 integral
  currents}, Annals of mathematics (2009), 741--794.

\bibitem[Sim93]{Simon_cylindrical}
Leon Simon, \emph{Cylindrical tangent cones and the singular set of minimal
  submanifolds}, Journal of Differential Geometry \textbf{38} (1993), no.~3,
  585--652.

\bibitem[Sko24]{Sk21}
Anna Skorobogatova, \emph{An upper minkowski dimension estimate for the
  interior singular set of area minimizing currents}, Communications on Pure
  and Applied Mathematics \textbf{77} (2024), no.~2, 1509--1572.

\bibitem[Spo19]{Spolaor_15}
Luca Spolaor, \emph{Almgren's type regularity for semicalibrated currents},
  Advances in Mathematics \textbf{350} (2019), 747--815.

\bibitem[SYZ96]{SYZ}
Andrew Strominger, Shing-Tung Yau, and Eric Zaslow, \emph{Mirror symmetry is t
  -duality}, Nuclear Phys. B \textbf{479(1-2)} (1996), 243–259.

\bibitem[Tia00]{Ti}
Gang Tian, \emph{Gauge theory and calibrated geometry. {I}}, Ann. of Math. (2)
  \textbf{151} (2000), no.~1, 193--268. \MR{1745014}

\bibitem[Wic14]{W14_annals}
Neshan Wickramasekera, \emph{A general regularity theory for stable codimension
  1 integral varifolds}, Annals of mathematics (2014), 843--1007.

\end{thebibliography}

\end{document}